\documentclass[reqno]{amsart}
\usepackage{amssymb}
\usepackage{amsmath}
\usepackage{amsfonts}

\newtheorem{theorem}{Theorem}
\theoremstyle{plain}

\newtheorem{claim}{Claim}

\newtheorem{corollary}{Corollary}

\newtheorem{definition}{Definition}

\newtheorem{lemma}{Lemma}

\newtheorem{proposition}{Proposition}
\newtheorem{remark}{Remark}

\DeclareMathOperator{\Div}{div}
\numberwithin{equation}{section}
\input{tcilatex}

\begin{document}
\title[Introverted algebras with mean value]{Introverted algebras with mean
value and applications}
\author{Jean Louis Woukeng}
\address{Department of Mathematics and Computer Science, University of
Dschang, P.O. Box 67, Dschang, Cameroon}
\email{jwoukeng@yahoo.fr}
\date{February, 2013}
\subjclass[2000]{Primary 46J10, 45K05; Secondary 35B40}
\keywords{Introverted algebras with mean value, semigroups,
sigma-convergence, convolution, homogenization}

\begin{abstract}
Let $A$ be an introverted algebra with mean value. We prove that its
spectrum $\Delta (A)$ is a compact topological semigroup, and that the
kernel $K(\Delta (A))$ of $\Delta (A)$ is a compact topological group over
which the mean value on $A$ can be identified as the Haar integral. Based on
these facts and also on the fact that $K(\Delta (A))$ is an ideal of $\Delta
(A)$, we define the convolution over $\Delta (A)$. We then use it to derive
some new convergence results involving the convolution product of sequences.
These convergence results provide us with an efficient method for studying
the asymptotics of nonlocal problems. The obtained results systematically
establish the connection between the abstract harmonic analysis and the
homogenization theory. To illustrate this, we work out some homogenization
problems in connection with nonlocal partial differential equations.
\end{abstract}

\maketitle

\section{Introduction and the main results}

Let $A$ be an algebra with mean value on $\mathbb{R}^{N}$, that is, a closed
subalgebra of the commutative Banach algebra BUC$(\mathbb{R}^{N})$ (of
bounded uniformly continuous functions on $\mathbb{R}^{N}$) that contains
the constants, is closed under complex conjugation, is translation invariant
and has an invariant mean value. Thus $A$ is a commutative Banach algebra
with spectrum denoted by $\Delta (A)$. We consider each element of $\Delta
(A)$ as a multiplicative linear functional on $A$. The usual (or Gelfand)
topology of $\Delta (A)$ is the relative weak$\ast $ topology induced on $%
\Delta (A)$ by $\sigma (A^{\prime },A)$. The properties of the Gelfand space 
$\Delta (A)$ are well known and can be found in any text book about Banach
algebras, see for instance \cite{Larsen}. The commonly known property is
that $\Delta (A)$ is a compact topological space. It is also known that it
becomes metrizable provided that $A$ is separable. In some special cases, $%
\Delta (A)$ is well characterized. For example, when $A$ is the algebra of
almost periodic functions, $\Delta (A)$ is a compact topological abelian
group, and in particular if $A$ is the algebra of periodic functions, $%
\Delta (A)$ is the $N$-dimensional torus $\mathbb{T}^{N}$.

On the other hand, it seems that almost nothing is known about $\Delta (A)$
for general algebras with mean value $A$. In this paper our aim is to
characterize the space $\Delta (A)$ for some general algebras $A$ and
present some applications. The relevance and importance of this
characterization is due, among other things, to the fact that if $A$ is 
\emph{introverted} \cite[p. 121]{TM} (see also \cite[p. 540]{Day} for the
general concept of introversion) then $\Delta (A)$ is a compact topological
semigroup (see Theorem \ref{th2} below). We are particularly interested in
the following questions:

\begin{itemize}
\item[(a)] When is the space $\Delta (A)$ a topological semigroup?

\item[(b)] Under which conditions is $\Delta (A)$ a topological group?

\item[(c)] In the case $\Delta (A)$ is a topological semigroup, what are the
properties of its kernel?

\item[(d)] Is it possible to define the convolution on $\Delta (A)$?

\item[(e)] In the case the convolution can be defined on $\Delta (A)$, what
is the connection between it and the $\Sigma $-convergence method?
\end{itemize}

In this paper we try to answer these questions and present some applications
of the obtained results, and the notions introduced to study them. In
particular we test these questions on the algebra of almost periodic
functions, the algebra of functions that converge at infinity, and in
general on any closed subalgebra of the algebra of weakly almost periodic
functions.

The content of the paper is summarized as follows. Section 2 deals with the
fundamental notions about algebras with mean value, the generalized
Besicovitch spaces and the derivation theory both associated to them. In
Section 3 we study the introverted algebras with mean value. We give the
answer to the questions (a), (b) and (c) raised above, but this in the
special setting of introverted algebras with mean value. The main results of
this section are as follows.

(1) \emph{If the algebra }$A$\emph{\ is introverted, then }$\Delta (A)$\emph{%
\ is a compact topological semigroup}. This result is known in the general
theory of Banach algebras of uniformly continuous functions; see e.g. \cite%
{Lau1, TM}. Our proof relies only on the compactness of $\Delta (A)$. We
also show that, \emph{if further the multiplication defined on }$\Delta (A)$%
\emph{\ is jointly continuous, then }$\Delta (A)$\emph{\ is a compact
topological group}. The second main result of Section 3 reads as

(2) \emph{If }$A$\emph{\ is introverted, then }$A$\emph{\ is a subalgebra of
the weakly almost periodic functions, and moreover, if the multiplication in 
}$\Delta (A)$\emph{\ is jointly continuous, then }$A$\emph{\ is a subalgebra
of the almost periodic functions}. The third main result of Section 3 is the
answer to the question (c) raised above, and is this

(3) \emph{If }$A$\emph{\ is introverted, then the kernel }$K(\Delta (A))$%
\emph{\ of }$\Delta (A)$\emph{\ is a compact topological group, and the mean
value on }$A$\emph{\ can be identified as the Haar integral over }$K(\Delta
(A))$.

In all the previous works dealing with algebras with mean value, the mean
value were identified as the integral over the spectrum only, see for
instance \cite{Hom1, CMP, NA}. Here we go further and we will see that this
result is of first importance when defining the convolution over the
spectrum of such algebras. The last main result of Section 3 is the basic
tool that enables us to establish the connection between the convolution and
the $\Sigma $-convergence method. It is new and constitutes the point of
departure of all the results dealing with convergence of sequences involving
delay. It reads as

(4) \emph{Let }$A$\emph{\ be an introverted algebra with mean value on }$%
\mathbb{R}^{N}$\emph{. Then if }$\delta _{y}$\emph{\ denotes the Dirac mass
at }$y$\emph{, we have }$\delta _{y}\in K(\Delta (A))$\emph{\ for almost all 
}$y\in \mathbb{R}^{N}$. We end Section 3 with the answer to question (d). We
define the convolution on the spectrum $\Delta (A)$ in terms of its kernel $%
K(\Delta (A))$.

In Section 4, in order to deal with some applications in homogenization
theory, we state and prove a De Rham type result. More precisely, the main
result of this section is this

(5) \emph{If }$A$\emph{\ is an algebra with mean value on }$\mathbb{R}^{N}$%
\emph{\ and }$L$\emph{\ is a bounded linear functional on }$(\mathcal{B}%
_{A}^{1,p^{\prime }})^{N}$\emph{\ which vanishes on the kernel of the
divergence, then there exists a function }$f\in \mathcal{B}_{A}^{p}$\emph{\
such that }$L=\overline{\nabla }_{y}f$. Although classically known in the
general framework of the distribution theory and in the special setting of
periodic functions, the above result is new in the framework of general
algebras with mean value.

In Section 5 we gather the notation and basic facts we need about the $%
\Sigma $-convergence method. Section 6 deals with the answer to question
(e). We study therein the connection between the convolution and the $\Sigma 
$-convergence method. The main result of this section is the following.

(6) \emph{Let }$(u_{\varepsilon })_{\varepsilon >0}\subset L^{p}(\Omega )$%
\emph{\ and }$(v_{\varepsilon })_{\varepsilon >0}\subset L^{q}(\mathbb{R}%
^{N})$\emph{\ be two sequences with }$p\geq 1$\emph{, }$q\geq 1$\emph{\ and }%
$\frac{1}{p}+\frac{1}{q}=1+\frac{1}{m}$\emph{. Assume that, as }$\varepsilon
\rightarrow 0$\emph{, }$u_{\varepsilon }\rightarrow u_{0}$\emph{\ in }$%
L^{p}(\Omega )$\emph{-weak }$\Sigma $\emph{\ and }$v_{\varepsilon
}\rightarrow v_{0}$\emph{\ in }$L^{q}(\mathbb{R}^{N})$\emph{-strong }$\Sigma 
$\emph{, where }$u_{0}$\emph{\ and }$v_{0}$\emph{\ are in }$L^{p}(\Omega ;%
\mathcal{B}_{A}^{p})$ and $L^{q}(\mathbb{R}^{N};\mathcal{B}_{A}^{q})$\emph{\
respectively. Then, as }$\varepsilon \rightarrow 0$\emph{, }$u_{\varepsilon
}\ast v_{\varepsilon }\rightarrow u_{0}\ast \ast v_{0}$\emph{\ in }$%
L^{m}(\Omega )$\emph{-weak }$\Sigma $ \emph{where }$\ast \ast $\emph{\
stands for the double convolution.} This result is also new.

One of the main motivation of the present study arises from the importance
and applications of the phenomena with delay in the real life. The results
of Section 6 have applications in mathematical neuroscience, engineering
sciences etc. In order to show how it works, we present in Sections 7 and 8
two applications of the results of the previous sections to homogenization
theory. In particular in Section 8, we give the answer to a question raised
by Attouch and Damlamian \cite{AD86} about the homogenization of nonlinear
operators involving convolution. This is a true advance as far as the
comprehension of the spectrum of an algebra with mean value as well as the
homogenization theory in general, are concerned.

We end this section with some preliminary notions. A \emph{directed set} is
a set $E$ equipped with a binary relation $\lesssim $ such that

\begin{itemize}
\item $\alpha \lesssim \alpha $ for all $\alpha \in E$;

\item if $\alpha \lesssim \beta $ and $\beta \lesssim \gamma $ then $\alpha
\lesssim \gamma $;

\item for any $\alpha ,\beta \in E$ there exists $\gamma \in E$ such that $%
\alpha \lesssim \gamma $ and $\beta \lesssim \gamma $.
\end{itemize}

\noindent A \emph{net} in a set $X$ is a mapping $\alpha \mapsto x_{\alpha }$
from a directed set $E$ into $X$. We usually denote such a mapping by $%
(x_{\alpha })_{\alpha \in E}$, or just by $(x_{\alpha })$ if $E$ is
understood, and we say that $(x_{\alpha })$ is indexed by $E$. Here below
are some examples of directed sets:

\begin{itemize}
\item[(i)] The set of positive integers $\mathbb{N}$, with $j\lesssim k$ if
and only if $j\leq k$.

\item[(ii)] The set $\mathbb{R}\backslash \{a\}$ ($a\in \mathbb{R}$), with $%
x\lesssim y$ if and only if $\left\vert x-a\right\vert \geq \left\vert
y-a\right\vert $.
\end{itemize}

Unless otherwise stated, vector spaces throughout are assumed to be real
vector spaces, and scalar functions are assumed to take real values. The
results obtained here easily carry over mutatis mutandis to the complex
setting.

The results of Section 6 were announced in \cite{SW}.

\section{Fundamentals of algebras with mean value}

We refer the reader to \cite{Deterhom} for details regarding some of the
results of this section.

A closed subalgebra $A$ of the $\mathcal{C}$*-algebra of bounded uniformly
continuous functions BUC$(\mathbb{R}^{N})$ is an \emph{algebra with mean
value} on $\mathbb{R}^{N}$ \cite{Jikov, NA, Zhikov4} if it contains the
constants, is translation invariant ($u(\cdot +a)\in A$ for any $u\in A$ and
each $a\in \mathbb{R}^{N}$) and is such that any of its elements possesses a
mean value, that is, for any $u\in A$, the sequence $(u^{\varepsilon
})_{\varepsilon >0}$ (defined by $u^{\varepsilon }(x)=u(x/\varepsilon )$, $%
x\in \mathbb{R}^{N}$) weakly$\ast $-converges in $L^{\infty }(\mathbb{R}^{N})
$ to some constant real function $M(u)$ as $\varepsilon \rightarrow 0$.

It is known that $A$ (endowed with the sup norm topology) is a commutative $%
\mathcal{C}$*-algebra with identity. We denote by $\Delta (A)$ the spectrum
of $A$ and by $\mathcal{G}$ the Gelfand transformation on $A$. We recall
that $\Delta (A)$ (a subset of the topological dual $A^{\prime }$ of $A$) is
the set of all nonzero multiplicative linear functionals on $A$, and $%
\mathcal{G}$ is the mapping of $A$ into $\mathcal{C}(\Delta (A))$ such that $%
\mathcal{G}(u)(s)=\left\langle s,u\right\rangle $ ($s\in \Delta (A)$), where 
$\left\langle ,\right\rangle $ denotes the duality pairing between $%
A^{\prime }$ and $A$. When equipped with the relative weak$\ast $ topology
on $A^{\prime }$ (the topological dual of $A$), $\Delta (A)$ is a compact
topological space, and the Gelfand transformation $\mathcal{G}$ is an
isometric $\ast $-isomorphism identifying $A$ with $\mathcal{C}(\Delta (A))$
as $\mathcal{C}$*-algebras. Moreover the mean value $M$ defined on $A$ is a
nonnegative continuous linear functional that can be expressed in terms of a
Radon measure $\beta $ (of total mass $1$) in $\Delta (A)$ (called the $M$%
\textit{-measure} for $A$ \cite{Hom1}) satisfying the property that $%
M(u)=\int_{\Delta (A)}\mathcal{G}(u)d\beta $\ for $u\in A$.

To any algebra with mean value $A$ we associate the following subspaces: $%
A^{m}=\{\psi \in \mathcal{C}^{m}(\mathbb{R}^{N}):$ $D_{y}^{\alpha }\psi \in
A $ $\forall \alpha =(\alpha _{1},...,\alpha _{N})\in \mathbb{N}^{N}$ with $%
\left\vert \alpha \right\vert \leq m\}$ (where $D_{y}^{\alpha }\psi
=\partial ^{\left\vert \alpha \right\vert }\psi /\partial y_{1}^{\alpha
_{1}}\cdot \cdot \cdot \partial y_{N}^{\alpha _{N}}$). Under the norm $%
\left\Vert \left\vert u\right\vert \right\Vert _{m}=\sup_{\left\vert \alpha
\right\vert \leq m}\left\Vert D_{y}^{\alpha }\psi \right\Vert _{\infty }$, $%
A^{m}$ is a Banach space. We also define the space $A^{\infty }=\{\psi \in 
\mathcal{C}^{\infty }(\mathbb{R}^{N}):$ $D_{y}^{\alpha }\psi \in A$ $\forall
\alpha =(\alpha _{1},...,\alpha _{N})\in \mathbb{N}^{N}\}$, a Fr\'{e}chet
space when endowed with the locally convex topology defined by the family of
norms $\left\Vert \left\vert \cdot \right\vert \right\Vert _{m}$.

Next, let $B_{A}^{p}$ ($1\leq p<\infty $) denote the Besicovitch space
associated to $A$, that is the closure of $A$ with respect to the
Besicovitch seminorm 
\begin{equation*}
\left\Vert u\right\Vert _{p}=\left( \underset{r\rightarrow +\infty }{\lim
\sup }\frac{1}{\left\vert B_{r}\right\vert }\int_{B_{r}}\left\vert
u(y)\right\vert ^{p}dy\right) ^{1/p}
\end{equation*}%
where $B_{r}$ is the open ball of $\mathbb{R}^{N}$ centered at the origin
and of radius $r>0$. It is known that $B_{A}^{p}$ is a complete seminormed
vector space verifying $B_{A}^{q}\subset B_{A}^{p}$ for $1\leq p\leq
q<\infty $. From this last property one may naturally define the space $%
B_{A}^{\infty }$ as follows: 
\begin{equation*}
B_{A}^{\infty }=\{f\in \cap _{1\leq p<\infty }B_{A}^{p}:\sup_{1\leq p<\infty
}\left\Vert f\right\Vert _{p}<\infty \}\text{.}\;\;\;\;\;\;\;\;\;
\end{equation*}%
We endow $B_{A}^{\infty }$ with the seminorm $\left[ f\right] _{\infty
}=\sup_{1\leq p<\infty }\left\Vert f\right\Vert _{p}$, which makes it a
complete seminormed space. We recall that the spaces $B_{A}^{p}$ ($1\leq
p\leq \infty $) are not in general Fr\'{e}chet spaces since they are not
separated in general. The following properties are worth noticing \cite{CMP,
NA}:

\begin{itemize}
\item[(\textbf{1)}] The Gelfand transformation $\mathcal{G}:A\rightarrow 
\mathcal{C}(\Delta (A))$ extends by continuity to a unique continuous linear
mapping (still denoted by $\mathcal{G}$) of $B_{A}^{p}$ into $L^{p}(\Delta
(A))$, which in turn induces an isometric isomorphism $\mathcal{G}_{1}$ of $%
B_{A}^{p}/\mathcal{N}=\mathcal{B}_{A}^{p}$ onto $L^{p}(\Delta (A))$ (where $%
\mathcal{N}=\{u\in B_{A}^{p}:\mathcal{G}(u)=0\}$). Moreover if $u\in
B_{A}^{p}\cap L^{\infty }(\mathbb{R}^{N})$ then $\mathcal{G}(u)\in L^{\infty
}(\Delta (A))$ and $\left\Vert \mathcal{G}(u)\right\Vert _{L^{\infty
}(\Delta (A))}\leq \left\Vert u\right\Vert _{L^{\infty }(\mathbb{R}^{N})}$.

\item[(\textbf{2)}] The mean value $M$ defined on $A$, extends by continuity
to a positive continuous linear form (still denoted by $M$) on $B_{A}^{p}$
satisfying $M(u)=\int_{\Delta (A)}\mathcal{G}(u)d\beta $ ($u\in B_{A}^{p}$).
Furthermore, $M(\tau _{a}u)=M(u)$ for each $u\in B_{A}^{p}$ and all $a\in 
\mathbb{R}^{N}$, where $\tau _{a}u=u(\cdot +a)$. Moreover for $u\in B_{A}^{p}
$ we have $\left\Vert u\right\Vert _{p}=\left[ M(\left\vert u\right\vert
^{p})\right] ^{1/p}$, and for $u+\mathcal{N}\in \mathcal{B}_{A}^{p}$ we may
still define its mean value once again denoted by $M$, as $M(u+\mathcal{N}%
)=M(u)$.
\end{itemize}

Let $1\leq p\leq \infty $. In order to define the Sobolev type spaces
associated to the algebra $A$, we consider the $N$-parameter group of
isometries $\{T(y):y\in \mathbb{R}^{N}\}$ defined by 
\begin{equation*}
T(y):\mathcal{B}_{A}^{p}\rightarrow \mathcal{B}_{A}^{p}\text{,\ }T(y)(u+%
\mathcal{N})=\tau _{y}u+\mathcal{N}\text{ for }u\in B_{A}^{p}\text{.}
\end{equation*}%
Since the elements of $A$ are uniformly continuous, $\{T(y):y\in \mathbb{R}%
^{N}\}$ is a strongly continuous group of operators in $\mathcal{L}(\mathcal{%
B}_{A}^{p},\mathcal{B}_{A}^{p})$ (the Banach space of continuous linear
functionals of $\mathcal{B}_{A}^{p}$ into $\mathcal{B}_{A}^{p}$): $T(y)(u+%
\mathcal{N})\rightarrow u+\mathcal{N}$ in $\mathcal{B}_{A}^{p}$ as $%
\left\vert y\right\vert \rightarrow 0$. We also associate to $\{T(y):y\in 
\mathbb{R}^{N}\}$ the following $N$-parameter group $\{\overline{T}(y):y\in 
\mathbb{R}^{N}\}$ defined by 
\begin{equation*}
\overline{T}(y):L^{p}(\Delta (A))\rightarrow L^{p}(\Delta (A));\ \overline{T}%
(y)\mathcal{G}_{1}(u+\mathcal{N})=\mathcal{G}_{1}(T(y)(u+\mathcal{N}))\text{
for }u\in B_{A}^{p}\text{.}
\end{equation*}%
The group $\{\overline{T}(y):y\in \mathbb{R}^{N}\}$ is also strongly
continuous. The infinitesimal generator of $T(y)$ (resp. $\overline{T}(y)$)
along the $i$th coordinate direction, denoted by $D_{i,p}$ (resp. $\partial
_{i,p}$), is defined as 
\begin{equation*}
D_{i,p}u=\lim_{t\rightarrow 0}\left( \frac{T(te_{i})u-u}{t}\right) \text{\
in }\mathcal{B}_{A}^{p}\text{ (resp. }\partial _{i,p}v=\lim_{t\rightarrow
0}\left( \frac{\overline{T}(te_{i})v-v}{t}\right) \text{\ in }L^{p}(\Delta
(A))\text{)}
\end{equation*}%
where here we have used the same letter $u$ to denote the equivalence class
of an element $u\in B_{A}^{p}$ in $\mathcal{B}_{A}^{p}$ and $e_{i}=(\delta
_{ij})_{1\leq j\leq N}$ ($\delta _{ij}$ being the Kronecker $\delta $). The
domain of $D_{i,p}$ (resp. $\partial _{i,p}$) in $\mathcal{B}_{A}^{p}$
(resp. $L^{p}(\Delta (A))$) is denoted by $\mathcal{D}_{i,p}$ (resp. $%
\mathcal{W}_{i,p}$). In the sequel we denote by $\varrho $ the canonical
mapping of $B_{A}^{p}$ onto $\mathcal{B}_{A}^{p}$, that is, $\varrho (u)=u+%
\mathcal{N}$ for $u\in B_{A}^{p}$. The following results are justified in 
\cite{Deterhom}. We refer the reader to the above-mentioned paper for their
justification.

\begin{proposition}
\label{p2.1}$\mathcal{D}_{i,p}$ (resp. $\mathcal{W}_{i,p}$) is a vector
subspace of $\mathcal{B}_{A}^{p}$ (resp. $L^{p}(\Delta (A))$), $D_{i,p}:%
\mathcal{D}_{i,p}\rightarrow \mathcal{B}_{A}^{p}$ (resp. $\partial _{i,p}:%
\mathcal{W}_{i,p}\rightarrow L^{p}(\Delta (A))$) is a linear operator, $%
\mathcal{D}_{i,p}$ (resp. $\mathcal{W}_{i,p}$) is dense in $\mathcal{B}%
_{A}^{p}$ (resp. $L^{p}(\Delta (A))$), and the graph of $D_{i,p}$ (resp. $%
\partial _{i,p}$) is closed in $\mathcal{B}_{A}^{p}\times \mathcal{B}%
_{A}^{p} $ (resp. $L^{p}(\Delta (A))\times L^{p}(\Delta (A))$).
\end{proposition}

The next result allows us to see $D_{i,p}$ as a generalization of the usual
partial derivative.

\begin{lemma}[{\protect\cite[Lemma 1]{Deterhom}}]
\label{l2.1}Let $1\leq i\leq N$. If $u\in A^{1}$ then $\varrho (u)\in 
\mathcal{D}_{i,p}$ and 
\begin{equation}
D_{i,p}\varrho (u)=\varrho \left( \frac{\partial u}{\partial y_{i}}\right)
.\ \ \ \ \ \ \ \ \ \ \ \ \ \ \ \ \ \ \ \ \ \ \ \ \ \ \ \ \ \ \ \ \ 
\label{2.2}
\end{equation}
\end{lemma}

From (\ref{2.2}) we infer that $D_{i,p}\circ \varrho =\varrho \circ \partial
/\partial y_{i}$, that is, $D_{i,p}$ generalizes the usual partial
derivative $\partial /\partial y_{i}$. One may also define higher order
derivatives by setting $D_{p}^{\alpha }=D_{1,p}^{\alpha _{1}}\circ \cdot
\cdot \cdot \circ D_{N,p}^{\alpha _{N}}$ (resp. $\partial _{p}^{\alpha
}=\partial _{1,p}^{\alpha _{1}}\circ \cdot \cdot \cdot \circ \partial
_{N,p}^{\alpha _{N}}$) for $\alpha =(\alpha _{1},...,\alpha _{N})\in \mathbb{%
N}^{N}$ with $D_{i,p}^{\alpha _{i}}=D_{i,p}\circ \cdot \cdot \cdot \circ
D_{i,p}$, $\alpha _{i}$-times. Now, let 
\begin{equation*}
\mathcal{B}_{A}^{1,p}=\cap _{i=1}^{N}\mathcal{D}_{i,p}=\{u\in \mathcal{B}%
_{A}^{p}:D_{i,p}u\in \mathcal{B}_{A}^{p}\ \forall 1\leq i\leq N\}
\end{equation*}%
and 
\begin{equation*}
\mathcal{D}_{A}(\mathbb{R}^{N})=\{u\in \mathcal{B}_{A}^{\infty }:D_{\infty
}^{\alpha }u\in \mathcal{B}_{A}^{\infty }\ \forall \alpha \in \mathbb{N}%
^{N}\}.
\end{equation*}%
It can be shown that $\mathcal{D}_{A}(\mathbb{R}^{N})$ is dense in $\mathcal{%
B}_{A}^{p}$, $1\leq p<\infty $. We also have that $\mathcal{B}_{A}^{1,p}$ is
a Banach space under the norm 
\begin{equation*}
\left\Vert u\right\Vert _{\mathcal{B}_{A}^{1,p}}=\left( \left\Vert
u\right\Vert _{p}^{p}+\sum_{i=1}^{N}\left\Vert D_{i,p}u\right\Vert
_{p}^{p}\right) ^{1/p}\ \ (u\in \mathcal{B}_{A}^{1,p}).
\end{equation*}

The counter-part of the above properties also holds with 
\begin{equation*}
W^{1,p}(\Delta (A))=\cap _{i=1}^{N}\mathcal{W}_{i,p}\text{\ in place of }%
\mathcal{B}_{A}^{1,p}
\end{equation*}%
and 
\begin{equation*}
\mathcal{D}(\Delta (A))=\{u\in L^{\infty }(\Delta (A)):\partial _{\infty
}^{\alpha }u\in L^{\infty }(\Delta (A))\ \forall \alpha \in \mathbb{N}^{N}\}%
\text{\ in that of }\mathcal{D}_{A}(\mathbb{R}^{N})\text{.}
\end{equation*}%
The following relation between $D_{i,p}$ and $\partial _{i,p}$ holds.

\begin{lemma}[{\protect\cite[Lemma 2]{Deterhom}}]
\label{l2.2}For any $u\in \mathcal{D}_{i,p}$ we have $\mathcal{G}_{1}(u)\in 
\mathcal{W}_{i,p}$ with $\mathcal{G}_{1}(D_{i,p}u)=\partial _{i,p}\mathcal{G}%
_{1}(u)$.
\end{lemma}

Now, let $u\in \mathcal{D}_{i,p}$ ($p\geq 1$, $1\leq i\leq N$). Then the
inequality 
\begin{equation*}
\left\Vert t^{-1}(T(te_{i})u-u)-D_{i,p}u\right\Vert _{1}\leq c\left\Vert
t^{-1}(T(te_{i})u-u)-D_{i,p}u\right\Vert _{p}
\end{equation*}%
for a positive constant $c$ independent of $u$ and $t$, yields $%
D_{i,1}u=D_{i,p}u$, so that $D_{i,p}$ is the restriction to $\mathcal{B}%
_{A}^{p}$ of $D_{i,1}$. Therefore, for all $u\in \mathcal{D}_{i,\infty }$ we
have $u\in \mathcal{D}_{i,p}$ ($p\geq 1$) and $D_{i,\infty }u=D_{i,p}u$ for
all $1\leq i\leq N$. It holds that

\begin{equation*}
\mathcal{D}_{A}(\mathbb{R}^{N})=\varrho (A^{\infty })
\end{equation*}%
and we have the

\begin{proposition}[{\protect\cite[Proposition 4]{Deterhom}}]
\label{p2.2}The following assertions hold.

\begin{itemize}
\item[(i)] $\int_{\Delta (A)}\partial _{\infty }^{\alpha }\widehat{u}d\beta
=0$ for all $u\in \mathcal{D}_{A}(\mathbb{R}^{N})$ and $\alpha \in \mathbb{N}%
^{N}$;

\item[(ii)] $\int_{\Delta (A)}\partial _{i,p}\widehat{u}d\beta =0$ for all $%
u\in \mathcal{D}_{i,p}$ and $1\leq i\leq N$;

\item[(iii)] $D_{i,p}(u\phi )=uD_{i,\infty }\phi +\phi D_{i,p}u$ for all $%
(\phi ,u)\in \mathcal{D}_{A}(\mathbb{R}^{N})\times \mathcal{D}_{i,p}$ and $%
1\leq i\leq N$.
\end{itemize}
\end{proposition}

The formula (iii) in the above proposition leads to the equality 
\begin{equation*}
\int_{\Delta (A)}\widehat{\phi }\partial _{i,p}\widehat{u}d\beta
=-\int_{\Delta (A)}\widehat{u}\partial _{i,\infty }\widehat{\phi }d\beta 
\text{, all }(u,\phi )\in \mathcal{D}_{i,p}\times \mathcal{D}_{A}(\mathbb{R}%
^{N}).
\end{equation*}%
This suggests us to define the concepts of distributions on $A$ and of a
weak derivative. Before we can do that, let us endow $\mathcal{D}_{A}(%
\mathbb{R}^{N})=\varrho (A^{\infty })$ with its natural topology defined by
the family of norms $N_{n}(u)=\sup_{\left\vert \alpha \right\vert \leq
n}\sup_{y\in \mathbb{R}^{N}}\left\vert D_{\infty }^{\alpha }u(y)\right\vert $%
, $n\in \mathbb{N}$. In this topology, $\mathcal{D}_{A}(\mathbb{R}^{N})$ is
a Fr\'{e}chet space. We denote by $\mathcal{D}_{A}^{\prime }(\mathbb{R}^{N})$
the topological dual of $\mathcal{D}_{A}(\mathbb{R}^{N})$. We endow it with
the strong dual topology. The elements of $\mathcal{D}_{A}^{\prime }(\mathbb{%
R}^{N})$ are called \textit{the distributions on }$A$. One can also define
the weak derivative of $f\in \mathcal{D}_{A}^{\prime }(\mathbb{R}^{N})$ as
follows: for any $\alpha \in \mathbb{N}^{N}$, $D^{\alpha }f$ stands for the
distribution defined by the formula 
\begin{equation*}
\left\langle D^{\alpha }f,\phi \right\rangle =(-1)^{\left\vert \alpha
\right\vert }\left\langle f,D_{\infty }^{\alpha }\phi \right\rangle \text{\
for all }\phi \in \mathcal{D}_{A}(\mathbb{R}^{N}).
\end{equation*}%
Since $\mathcal{D}_{A}(\mathbb{R}^{N})$ is dense in $\mathcal{B}_{A}^{p}$ ($%
1\leq p<\infty $), it is immediate that $\mathcal{B}_{A}^{p}\subset \mathcal{%
D}_{A}^{\prime }(\mathbb{R}^{N})$ with continuous embedding, so that one may
define the weak derivative of any $f\in \mathcal{B}_{A}^{p}$, and it
verifies the following functional equation: 
\begin{equation*}
\left\langle D^{\alpha }f,\phi \right\rangle =(-1)^{\left\vert \alpha
\right\vert }\int_{\Delta (A)}\widehat{f}\partial _{\infty }^{\alpha }%
\widehat{\phi }d\beta \text{\ for all }\phi \in \mathcal{D}_{A}(\mathbb{R}%
^{N}).
\end{equation*}%
In particular, for $f\in \mathcal{D}_{i,p}$ we have 
\begin{equation*}
-\int_{\Delta (A)}\widehat{f}\partial _{i,p}\widehat{\phi }d\beta
=\int_{\Delta (A)}\widehat{\phi }\partial _{i,p}\widehat{f}d\beta \ \
\forall \phi \in \mathcal{D}_{A}(\mathbb{R}^{N}),
\end{equation*}%
so that we may identify $D_{i,p}f$ with $D^{\alpha _{i}}f$, $\alpha
_{i}=(\delta _{ij})_{1\leq j\leq N}$. Conversely, if $f\in \mathcal{B}%
_{A}^{p}$ is such that there exists $f_{i}\in \mathcal{B}_{A}^{p}$ with $%
\left\langle D^{\alpha _{i}}f,\phi \right\rangle =-\int_{\Delta (A)}\widehat{%
f}_{i}\widehat{\phi }d\beta $ for all $\phi \in \mathcal{D}_{A}(\mathbb{R}%
^{N})$, then $f\in \mathcal{D}_{i,p}$ and $D_{i,p}f=f_{i}$. We are therefore
justified in saying that $\mathcal{B}_{A}^{1,p}$ is a Banach space under the
norm $\left\Vert \cdot \right\Vert _{\mathcal{B}_{A}^{1,p}}$. The same
result holds for $W^{1,p}(\Delta (A))$. Moreover it is a fact that $\mathcal{%
D}_{A}(\mathbb{R}^{N})$ (resp. $\mathcal{D}(\Delta (A))$) is a dense
subspace of $\mathcal{B}_{A}^{1,p}$ (resp. $W^{1,p}(\Delta (A))$).

We need a further notion. A function $f\in \mathcal{B}_{A}^{1}$ is said to
be \emph{invariant} if for any $y\in \mathbb{R}^{N}$, $T(y)f=f$. It is
immediate that the above notion of invariance is the well-known one relative
to dynamical systems. An algebra with mean value will therefore said to be 
\emph{ergodic} if every invariant function $f$ is constant in $\mathcal{B}%
_{A}^{1}$. As in \cite{BMW} one may show that $f\in \mathcal{B}_{A}^{1}$ is
invariant if and only if $D_{i,1}f=0$ for all $1\leq i\leq N$. We denote by $%
I_{A}^{p}$ the set of $f\in \mathcal{B}_{A}^{p}$ that are invariant. The set 
$I_{A}^{p}$ is a closed vector subspace of $\mathcal{B}_{A}^{p}$ satisfying
the following important property:%
\begin{equation}
f\in I_{A}^{p}\text{ if and only if }D_{i,p}f=0\text{ for all }1\leq i\leq N%
\text{.}  \label{2.5}
\end{equation}%
We therefore endow $\mathcal{B}_{A}^{1,p}/I_{A}^{p}$ with the seminorm 
\begin{equation}
\left\Vert u+I_{A}^{p}\right\Vert _{\#,p}=\left( \sum_{i=1}^{N}\left\Vert
D_{i,p}u\right\Vert _{p}^{p}\right) ^{1/p}\ \ \text{for }u\in \mathcal{B}%
_{A}^{1,p}  \label{2.5'}
\end{equation}%
which, in view of (\ref{2.5}) is actually a norm under which $\mathcal{B}%
_{A}^{1,p}/I_{A}^{p}$ is a reflexive Banach space ($I_{A}^{p}$ is a closed
subspace of $\mathcal{B}_{A}^{p}$).

Let $u\in \mathcal{B}_{A}^{p}$ (resp. $v=(v_{1},...,v_{N})\in (\mathcal{B}%
_{A}^{p})^{N}$). We define the gradient operator $D_{p}u$ and the divergence
operator $\Div_{p}v$ by 
\begin{equation*}
D_{p}u=(D_{1,p}u,...,D_{N,p}u)\text{ and\ }\Div_{p}v=%
\sum_{i=1}^{N}D_{i,p}v_{i}.
\end{equation*}%
We define also the Laplacian operator by $\Delta _{p}u=\Div_{p}(D_{p}u)$,
and it holds that $\Delta _{p}\varrho (u)=\varrho (\Delta _{y}u)$ for all $%
u\in A^{\infty }$, where $\Delta _{y}$ denotes the usual Laplacian operator
on $\mathbb{R}_{y}^{N}$. The following obvious properties are satisfied.

\begin{itemize}
\item[1.] The divergence operator $\Div_{p^{\prime }}$ ($p^{\prime }=p/(p-1)$%
) sends continuously and linearly $(\mathcal{B}_{A}^{p^{\prime }})^{N}$ into 
$(\mathcal{B}_{A}^{1,p})^{\prime }$ and satisfies 
\begin{equation}
\left\langle \Div_{p^{\prime }}u,v\right\rangle =-\left\langle
u,D_{p}v\right\rangle \text{\ for }v\in \mathcal{B}_{A}^{1,p}\text{ and }%
u=(u_{i})\in (\mathcal{B}_{A}^{p^{\prime }})^{N}\text{,}  \label{2.6}
\end{equation}%
where $\left\langle u,D_{p}v\right\rangle =\sum_{i=1}^{N}\int_{\Delta (A)}%
\widehat{u}_{i}\partial _{i,p}\widehat{v}d\beta $.

\item[2.] If in (\ref{2.6}) we take $u=D_{p^{\prime }}w$ with $w\in \mathcal{%
B}_{A}^{p^{\prime }}$ being such that $D_{p^{\prime }}w\in (\mathcal{B}%
_{A}^{p^{\prime }})^{N}$ then we have 
\begin{equation*}
\left\langle \Delta _{p^{\prime }}w,v\right\rangle =\left\langle \Div%
_{p^{\prime }}(D_{p^{\prime }}w),v\right\rangle =-\left\langle D_{p^{\prime
}}w,D_{p}v\right\rangle \text{\ for all }v\in \mathcal{B}_{A}^{1,p}\text{
and }w\in \mathcal{B}_{A}^{1,p^{\prime }}.
\end{equation*}%
If in addition $v=\phi $ with $\phi \in \mathcal{D}_{A}(\mathbb{R}^{N})$
then $\left\langle \Delta _{p^{\prime }}w,\phi \right\rangle =-\left\langle
D_{p^{\prime }}w,D_{p}\phi \right\rangle $, so that, for $p=2$, we get 
\begin{equation}
\left\langle \Delta _{2}w,\phi \right\rangle =\left\langle w,\Delta _{2}\phi
\right\rangle \text{\ for all }w\in \mathcal{B}_{A}^{2}\text{ and }\phi \in 
\mathcal{D}_{A}(\mathbb{R}^{N})\text{.}  \label{2.7}
\end{equation}
\end{itemize}

The same definitions and properties hold true when replacing $\mathcal{B}%
_{A}^{p}$ by $L^{p}(\Delta (A))$, $D_{p}$ by $\partial _{p}$, $\Div_{p}$ by $%
\widehat{\Div}_{p}$ and $\Delta _{p}$ by $\widehat{\Delta }_{p}$, and we
have the following relations between these operators (provided that they
make sense): 
\begin{equation*}
\partial _{p}=\mathcal{G}_{1}\circ D_{p},\ \widehat{\Div}_{p}=\mathcal{G}%
_{1}\circ \Div_{p}\text{ and }\widehat{\Delta }_{p}=\mathcal{G}_{1}\circ
\Delta _{p}.
\end{equation*}

Before we can state one of the most important result of this section, we
still need to make some preliminaries and some notation. To this end let $%
f\in \mathcal{B}_{A}^{p}$. We know that $D^{\alpha _{i}}f$ exists (in the
sense of distributions) and that $D^{\alpha _{i}}f=D_{i,p}f$ if $f\in 
\mathcal{D}_{i,p}$. So we can drop the subscript $p$ and therefore denote $%
D_{i,p}$ (resp. $\partial _{i,p}$) by $\overline{\partial }/\partial y_{i}$
(resp. $\partial _{i}$). Thus, $\overline{D}_{y}\equiv \overline{\nabla }%
_{y} $ will stand for the gradient operator $(\overline{\partial }/\partial
y_{i})_{1\leq i\leq N}$ and $\overline{\Div}_{y}$ for the divergence
operator $\Div_{p}$, with $\mathcal{G}_{1}\circ \overline{\Div}_{y}=\widehat{%
\Div}$. We will also denote $\partial \equiv (\partial _{1},...,\partial
_{N})$. Finally, we shall denote the Laplacian operator on $\mathcal{B}%
_{A}^{p}$ by $\overline{\Delta }_{y}$.

\section{Introverted algebras with mean value}

For useful purposes, we need to characterize the spectrum of an algebra with
mean value $A$. It is known to be the Stone-\v{C}ech compactification of $%
\mathbb{R}^{N}$ provided that $A$ separates the points of $\mathbb{R}^{N}$
as seen below.

\begin{theorem}
\label{th1}Let $A$ be an algebra with mean value. Assume $A$ separates the
points of $\mathbb{R}^{N}$. Then $\Delta (A)$ is the Stone-\v{C}ech
compactification of $\mathbb{R}^{N}$.
\end{theorem}

\begin{proof}
For each $y\in \mathbb{R}^{N}$ let us define $\delta _{y}$, the Dirac mass
at $y$ by setting $\delta _{y}(u)=u(y)$ for $u\in A$. Then the mapping $%
\delta :y\mapsto \delta _{y}$, of $\mathbb{R}^{N}$ into $\Delta (A)$, is
continuous and has dense range. In fact as the topology in $\Delta (A)$ is
the weak$\ast $ one and further the mappings $y\mapsto \delta _{y}(u)=u(y)$, 
$u\in A$, are continuous on $\mathbb{R}^{N}$, it follows that $\delta $ is
continuous. Now assuming that $\delta (\mathbb{R}^{N})$ is not dense in $%
\Delta (A)$ we derive the existence of a non empty open subset $U$ of $%
\Delta (A)$ such that $U\cap \delta (\mathbb{R}^{N})=\emptyset $. Then by
Urysohn's lemma there exists $v\in \mathcal{C}(\Delta (A))$ with $v\neq 0$
and $\left. v\right\vert _{\Delta (A))\backslash U}=0$ where $\left.
v\right\vert _{\Delta (A))\backslash U}$ denotes the restriction of $v$ to $%
\Delta (A)\backslash U$. By the Gelfand representation theorem, $v=\mathcal{G%
}(u)$ for some $u\in A$. But then 
\begin{equation*}
u(y)=\delta _{y}(u)=\mathcal{G}(u)(\delta _{y})=v(\delta _{y})=0
\end{equation*}%
for all $y\in \mathbb{R}^{N}$, contradicting $u\neq 0$. Thus $\delta (%
\mathbb{R}^{N})$ is dense in $\Delta (A)$.

Next, every $f$ in $A$ (viewed as element of $\mathcal{B}(\mathbb{R}^{N})$)
extends continuously to $\Delta (A)$ in the sense that there exists $%
\widehat{f}\in \mathcal{C}(\Delta (A))$ such that $\widehat{f}(\delta
_{y})=f(y)$ for all $y\in \mathbb{R}^{N}$ (just take $\widehat{f}=\mathcal{G}%
(f)$). Finally assume that $A$ separates the points of $\mathbb{R}^{N}$.
Then the mapping $\delta :\mathbb{R}^{N}\rightarrow \delta (\mathbb{R}^{N})$
is a homeomorphism. In fact, we only need to prove that $\delta $ is
injective. For that, let $y,z\in \mathbb{R}^{N}$ with $y\neq z$; since $A$
separates the points of $\mathbb{R}^{N}$, there exists a function $u\in A$
such that $u(y)\neq u(z)$, hence $\delta _{y}\neq \delta _{z}$, and our
claim is justified. We therefore conclude that the couple $(\Delta
(A),\delta )$ is the Stone-\v{C}ech compactification of $\mathbb{R}^{N}$.
\end{proof}

Before dealing with the general situation, let us begin with a summary of
the facts and notation we shall use in connection with the semigroup theory.

A \emph{semigroup} is a set supplied with an associative binary operation
that will be referred to as multiplication. Let $S$ be a semigroup with
multiplication $S\times S\rightarrow S$, $(x,y)\mapsto xy$. If $S$ is a
topological space, the multiplication is said to be \emph{separately
continuous} if the maps $x\mapsto xy$ and $x\mapsto yx$ (for each fixed $y$
in $S$) of $S$ into itself, are continuous. The multiplication in $S$ is
said to be \emph{jointly continuous} if the map $(x,y)\mapsto xy$ is
continuous.

Let $S$ be a semigroup with and identity $e$ ($ex=xe=x$, all $x\in S$). $S$
is said to be a \emph{topological semigroup} if $S$ is a Hausdorff
topological space in which the multiplication is separately continuous.
There is a vast literature on topological semigroups with the stronger
assumption that the multiplication is jointly continuous; see e.g. \cite%
{Glicksberg1, WA}. Although our results do work in that special setting, it
is not necessary for us to consider such kind of semigroups for some obvious
reasons that shall be given later (see e.g. Remark \ref{re3}).

Now, let $A$ be an algebra with mean value on $\mathbb{R}^{N}$. For $\mu \in
\Delta (A)$ and $f\in A$ we define the function $T_{\mu }f$ by $T_{\mu
}f(x)=\mu (\tau _{x}f)$, all $x\in \mathbb{R}^{N}$, where $\tau
_{x}f=f(\cdot +x)$. Then since $A$ is translation invariant, $T_{\mu }f$ is
well defined as an element of BUC$(\mathbb{R}^{N})$. This defines a bounded
linear operator $T_{\mu }:A\rightarrow $BUC$(\mathbb{R}^{N})$. The algebra $A
$ is said to be \emph{introverted} if $T_{\mu }(A)\subset A$ for each $\mu
\in \Delta (A)$. This concept is due to T. Mitchell \cite[p. 121]{TM} (in
which he rather uses the symbol $M$-introverted instead of merely
introverted) and has applications to fixed points theorems \cite{TM}. Here
below is the characterization of the concept of introversion of an algebra
with mean value $A$.

\begin{lemma}
\label{le1}Let $A$ be an algebra with mean value on $\mathbb{R}^{N}$. Then $%
A $ is introverted if and only if for any $f\in A$, the pointwise closure in 
\emph{BUC}$(\mathbb{R}^{N})$ of the orbit $\{\tau _{a}f:a\in \mathbb{R}%
^{N}\} $ of $f$ is included in $A$.
\end{lemma}

\begin{proof}
Assume that $A$ is introverted. Let $g$ be in the pointwise closure of the
orbit of $f$ (where $f$ is fixed in $A$). Then there is a net $(a_{\nu
})_{\nu }\subset \mathbb{R}^{N}$ such that $\tau _{a_{\nu }}f\rightarrow g$
pointwise. $\{\delta _{a_{\nu }}\}_{\nu }$ is a net in the compact space $%
\Delta (A)$, so let $\mu \in \Delta (A)$ be such that a subnet $\{\delta
_{a_{\eta }}\}_{\eta }$ of $\{\delta _{a_{\nu }}\}_{\nu }$ converges weak$%
\ast $ to $\mu $ in $\Delta (A)$. For any $x\in \mathbb{R}^{N}$, $\tau
_{a_{\eta }}f(x)=\delta _{a_{\eta }}(\tau _{x}f)\rightarrow \mu (\tau _{x}f)$%
; thus $g(x)=\mu (\tau _{x}f)$, and since $A$ is introverted, $g\in A$,
i.e., the pointwise closure of the orbit of $f$ is included in $A$.

Conversely, let $\mu \in \Delta (A)$ and let $f\in A$. Define $g(x)=\mu
(\tau _{x}f)$, $x\in \mathbb{R}^{N}$. Finally let $(a_{\nu })_{\nu }\subset 
\mathbb{R}^{N}$ be a net such that $\delta _{a_{\nu }}\rightarrow \mu $ weak$%
\ast $ in $\Delta (A)$; then $\tau _{a_{\nu }}f(x)=\delta _{a_{\nu }}(\tau
_{x}f)\rightarrow g(x)$, hence $\tau _{a_{\nu }}f\rightarrow g$ pointwise,
i.e. $g$ belongs to the pointwise closure of the orbit of $f$. Whence $g\in
A $, and so $A$ is introverted.
\end{proof}

As a consequence of the preceding lemma, some examples of introverted
algebras are quoted below.

\begin{corollary}
\label{co1}The algebras: $\mathcal{B}_{\infty }(\mathbb{R}^{N})$ of
functions that are finite at infinity, \emph{AP}$(\mathbb{R}^{N})$ of almost
periodic functions, and \emph{WAP}$(\mathbb{R}^{N})$ of weakly almost
periodic functions are introverted.
\end{corollary}

\begin{proof}
That these algebras are algebras with mean value is well known; see e.g. 
\cite{CMP, NA}. $\mathcal{B}_{\infty }(\mathbb{R}^{N})$ and \textrm{AP}$(%
\mathbb{R}^{N})$ are both closed subalgebras of \textrm{WAP}$(\mathbb{R}%
^{N}) $ (see \cite{Eberlein}). Thus, to see that $\mathcal{B}_{\infty }(%
\mathbb{R}^{N})$, \textrm{AP}$(\mathbb{R}^{N})$ and \textrm{WAP}$(\mathbb{R}%
^{N})$ are introverted, it is sufficient to show that any closed translation
invariant subalgebra $A$ of \textrm{WAP}$(\mathbb{R}^{N})$ is introverted.
To this end, let $f\in A$, and let $g$ lying in the pointwise closure of the
orbit of $f$. There is a net $(a_{\nu })_{\nu }\subset \mathbb{R}^{N}$ such
that $\tau _{a_{\nu }}f\rightarrow g$ pointwise. But the set $\{\tau
_{a_{\nu }}f:\nu \}$ is weakly relatively compact, hence by passing
eventually to a subnet, we have $\tau _{a_{\nu }}f\rightarrow g$ weakly in
BUC$(\mathbb{R}^{N})$. It readily follows that $g$, viewed as the weak limit
of a net in $A$, belongs to $A$ since $A$ is weakly closed. As a result of
Lemma \ref{le1}, $A$ is introverted.
\end{proof}

We are now in a position to establish the structure theorem for the spectrum
of an algebra with mean value. This result and its corollaries are the basic
results on topological semigroups for the applications made in the following
sections.

\begin{theorem}
\label{th2}Let $A$ be an algebra with mean value on $\mathbb{R}^{N}$. Assume
that $A$ is introverted. Then its spectrum $\Delta (A)$ is a compact
topological semigroup. Moreover if the multiplication in $\Delta (A)$ is
jointly continuous then $\Delta (A)$ is a compact topological group.
\end{theorem}

\begin{proof}
1. Let us first define the Arens product (see e.g. \cite{Day, Lau1}) on $%
\Delta (A)$. For $f\in A$ and $x\in \mathbb{R}^{N}$, we know that the
translate $\tau _{x}f$ (defined by $\tau _{x}f(y)=f(x+y)$ for $y\in \mathbb{R%
}^{N}$) is in $A$. Thus, for $\nu \in \Delta (A)$ we can define 
\begin{equation*}
T_{\nu }f(x)=\nu (\tau _{x}f)\text{\ \ }(x\in \mathbb{R}^{N}).
\end{equation*}%
The operator $T_{\nu }$ is nonnegative, sends the constant function $1$ into
itself and, since $A$ is introverted, maps $A$ into $A$. We therefore find
that, if $\mu \in \Delta (A)$, the product $\mu \nu $ determined by 
\begin{equation*}
\mu \nu (f)=\mu (T_{\nu }f)\ \ \ (f\in A)
\end{equation*}%
is well defined as and element of $\Delta (A)$. That it is associative is
straightforward. Moreover denoting by $\delta $ the Dirac mass at the origin 
$0$, it holds that $\delta \nu =\nu \delta =\nu $ for all $\nu \in \Delta (A)
$, so that $\delta $ is a unity in $\Delta (A)$ for our product. Finally,
for the separate continuity, fix $\nu $ in $\Delta (A)$ and let us check
that $\mu \mapsto \mu \nu $ is continuous. Before we can do this, let us
however observe that $\mu \nu $ is actually the convolution between $\mu $
and $\nu $ defined by 
\begin{equation*}
\mu \ast \nu (f)=\int f(x+y)d\nu (y)d\mu (x)\ \ \ \ (f\in A).
\end{equation*}%
With this in mind, we follow some arguments of \cite{Siebert} (see in
particular Proposition 3.1 therein). Let $(\mu _{i})_{i\in I}$ be a net in $%
\Delta (A)$ converging weak$\ast $ to $\mu $ in $\Delta (A)$. Then the
tensor product $(\mu _{i}\otimes \nu )_{i\in I}$ is a net in $\Delta
(A)\times \Delta (A)$, hence possesses a weak$\ast $ cluster point $\lambda
\in \Delta (A)\times \Delta (A)$. Let $(\mu _{i(j)}\otimes \nu )_{j\in J}$
be a subnet of $(\mu _{i}\otimes \nu )_{i\in I}$ converging weak$\ast $ to $%
\lambda $. Then if for $f,g\in A$ the function $f\times g$ is defined by $%
(f\times g)(x,y)=f(x)g(y)$ ($x,y\in \mathbb{R}^{N}$), it follows that 
\begin{equation*}
\lambda (f\times g)=\lim_{j\in J}(\mu _{i(j)}\otimes \nu )(f\times
g)=\lim_{j\in J}\mu _{i(j)}(f)\nu (g)=\mu (f)\nu (g)=(\mu \otimes \nu
)(f\times g).
\end{equation*}%
Hence by the definition of product measure we must have $\lambda =\mu
\otimes \nu $. This shows that $\lim_{i\in I}(\mu _{i}\otimes \nu )=\mu
\otimes \nu $. Since the map $\alpha \otimes \beta \mapsto \alpha \ast \beta 
$ of $\Delta (A)\times \Delta (A)$ into $\Delta (A)$ is continuous (by
definition of convolution) we finally have $\lim_{i\in I}\mu _{i}\nu =\mu
\nu $. The continuity of $\mu \mapsto \mu \nu $ follows thereby.

The fact that $\mu \mapsto \nu \mu $ is continuous is a consequence of the
commutativity of the product as seen below (see e.g. \cite[Theorem 3.1]%
{Glicksberg1}): 
\begin{equation*}
\int f(x+y)d\nu (y)d\mu (x)=\int f(y+x)d\mu (y)d\nu (x)\text{, all }f\text{
in }A.
\end{equation*}

2. Now assume that the multiplication is jointly continuous. Taking into
account the first part of the proof, we see that it is sufficient to check
that any $\mu \in \Delta (A)$ is invertible. It is clear that if $\mu
=\delta _{y}$ for some $y\in \mathbb{R}^{N}$, then $\nu =\delta _{-y}$ is
the inverse of $\mu $, since $\mu \nu =\nu \mu =\delta $. Now, assume that $%
\mu $ is arbitrary, and let $(y_{i})_{i\in I}\subset \mathbb{R}^{N}$ be a
net such that $\delta _{y_{i}}\rightarrow \mu $ in $\Delta (A)$-weak$\ast $. 
$(\delta _{-y_{i}})_{i\in I}$ is a net in $\Delta (A)$, hence possesses a
weak$\ast $ cluster point $\nu \in \Delta (A)$ which is the limit of a
subnet $(\delta _{-y_{i(j)}})_{i\in J}$ of $(\delta _{-y_{i}})_{i\in I}$.
Invoking both the continuity of the multiplication and the identity $\delta
_{-y_{i(j)}}\delta _{y_{i(j)}}=\delta $, we are led (after passing to the
limit) to $\mu \nu =\delta $. The uniqueness of $\nu $ is a consequence of
the commutativity of the multiplication. Thus $\Delta (A)$ is a group. Since
a compact topological semigroup that is a group must be a topological group
(see \cite[Theorem 2.1]{DG}), it readily follows that $\Delta (A)$ is a
compact topological group. This concludes the proof.
\end{proof}

Let $A$ be an introverted algebra with mean value. Then its spectrum is a
compact topological semigroup. We may therefore define the translation
operator on $\mathcal{C}(\Delta (A))$ as follows. For $f\in \mathcal{C}%
(\Delta (A))$ and $r\in \Delta (A)$, $\tau _{r}f(s)=f(sr)$ \ ($s\in \Delta
(A)$). The following holds true.

\begin{corollary}
\label{co2}Let $A$ be an introverted algebra with mean value on $\mathbb{R}%
^{N}$. The mapping $\varphi :\mathbb{R}^{N}\rightarrow \Delta (A)$ defined
by $\varphi (y)=\delta _{y}$ is a continuous homomorphism with dense range.
Moreover it holds that 
\begin{equation}
\widehat{\tau _{y}u}=\tau _{\varphi (y)}\widehat{u}\text{, all }u\in A\text{
and all }y\in \mathbb{R}^{N}  \label{111}
\end{equation}%
where $\widehat{\cdot }$ denotes the Gelfand transformation on $A$.
\end{corollary}

\begin{proof}
For the continuity of $\varphi $, let $(y_{n})_{n}\subset \mathbb{R}^{N}$ be
a net such that $y_{n}\rightarrow y$; then for any $u\in A$, $%
u(y_{n})\rightarrow u(y)$, hence $\delta _{y_{n}}\rightarrow \delta _{y}$ in
the weak$\ast $ topology of $A^{\prime }$. This proves the continuity of $%
\varphi $. The fact that $\varphi $ is a homomorphism just comes from the
obvious equality $\delta _{x+y}=\delta _{x}\ast \delta _{y}\equiv \delta
_{x}\delta _{y}$ for all $x,y\in \mathbb{R}^{N}$.

Let us check (\ref{111}). Let $y,z\in \mathbb{R}^{N}$ and $u\in A$; then 
\begin{eqnarray*}
\widehat{\tau _{y}u}(\delta _{z}) &=&\delta _{z}(\tau _{y}u)=u(y+z)=\delta
_{y+z}(u)=\widehat{u}(\delta _{y}\delta _{z})=\widehat{u}(\varphi (y)\delta
_{z}) \\
&=&\tau _{\varphi (y)}\widehat{u}(\delta _{z})
\end{eqnarray*}%
Using the continuity of $\varphi $ and the density of $\{\delta _{z}:z\in 
\mathbb{R}^{N}\}$ in $\Delta (A)$ we infer (\ref{111}).
\end{proof}

\begin{remark}
\label{re1}\emph{It follows from (\ref{111}) (in Corollary \ref{co2}) that }$%
A$\emph{\ has an invariant mean if and only if }$\mathcal{C}(\Delta (A))$%
\emph{\ does.}
\end{remark}

The following result is of independent interest. It characterizes in terms
of the weakly almost periodic functions, the algebras with mean value that
are introverted.

\begin{theorem}
\label{th3}Let $A$ be an introverted algebra with mean value on $\mathbb{R}%
^{N}$. Then

\begin{itemize}
\item[(i)] $A$ is a subalgebra of the algebra of weakly almost periodic
functions.

\item[(ii)] If the multiplication in $\Delta (A)$ is jointly continuous,
then $A$ is a subalgebra of the almost periodic functions.
\end{itemize}
\end{theorem}

\begin{proof}
The proof is modeled on the one of Deleeuw and Glicksberg \cite[Theorem 2.7]%
{DG}.

Let $\mathcal{G}^{-1}:\mathcal{C}(\Delta (A))\rightarrow A$ be the inverse
mapping of the Gelfand transformation on $A$. We know from (\ref{111}) (in
Corollary \ref{co2}) that 
\begin{equation*}
\tau _{\varphi (y)}f=\tau _{y}\mathcal{G}^{-1}(f)\text{, all }f\in \mathcal{C%
}(\Delta (A))\text{ and all }y\in \mathbb{R}^{N}.
\end{equation*}%
thus for each $u\in A$, $\{\tau _{y}u:y\in \mathbb{R}^{N}\}$ is the
continuous image of the set $\{\tau _{\varphi (y)}\widehat{u}:y\in \mathbb{R}%
^{N}\}$. Therefore, in order to prove the theorem, it suffices to check (i)
and (ii) hold for $f\in \mathcal{C}(\Delta (A))$ instead of $A$, for if the
set $\mathcal{O}(\widehat{u})=\{\tau _{s}\widehat{u}:s\in \Delta (A)\}$ is
weakly relatively compact in $\mathcal{C}(\Delta (A))$, then so is $\{\tau
_{\varphi (y)}\widehat{u}:y\in \mathbb{R}^{N}\}$ as a subset of $\mathcal{O}(%
\widehat{u})$, and hence $\{\tau _{y}u:y\in \mathbb{R}^{N}\}$ is also weakly
relatively compact in BUC$(\mathbb{R}^{N})$.

Let us now verify (i) and (ii) hold in $\mathcal{C}(\Delta (A))$.

(i) We use a result of Grothendieck \cite{Gro} stating that weak compactness
and compactness in the topology of pointwise convergence agree on bounded
subsets of $\mathcal{C}(X)$ for $X$ compact. Bearing this in mind, let $f\in 
\mathcal{C}(\Delta (A))$. Invoking the separate continuity of the
multiplication in $\Delta (A)$, we get that the mapping $s\mapsto \tau _{s}f$
from $\Delta (A)$ into $\mathcal{C}(\Delta (A))$ is continuous when $%
\mathcal{C}(\Delta (A))$ is taken in the topology of pointwise convergence.
Thus $\mathcal{O}(f)$ is compact in that topology as the continuous image of
the compact set $\Delta (A)$. We infer from the above mentioned result of
Grothendieck that $\mathcal{O}(f)$ is weakly compact.

(ii) It is sufficient to show that the mapping $s\mapsto \tau _{s}f$ is
strongly continuous, for if it is strongly continuous, then $\mathcal{O}(f)$
will be strongly compact as the strongly continuous image of $\Delta (A)$.
So, let us check the strong continuity of that mapping. The function $\Delta
(A)\times \Delta (A)\rightarrow \mathbb{C}$ defined by $(s,r)\mapsto f(sr)$
is continuous. Thus, given $\varepsilon >0$, for a fixed $r_{0}\in \Delta
(A) $ and for any $s\in \Delta (A)$, there is a neighborhood $V_{s}\times
W_{s}$ of $(s,r_{0})$ in $\Delta (A)\times \Delta (A)$ such that 
\begin{equation*}
\left\vert f(sr)-f(sr_{0})\right\vert <\varepsilon \text{\ for each }%
(s,r)\in V_{s}\times W_{s}.
\end{equation*}%
$\Delta (A)$ being compact, we can cover it by finitely many $V_{s}$, say $%
(V_{s_{i}})_{1\leq i\leq n}$. Set $W=\cap _{1\leq i\leq n}W_{s_{i}}$, a
neighborhood of $r_{0}$ which satisfies 
\begin{equation*}
\left\vert f(sr)-f(sr_{0})\right\vert <\varepsilon \text{\ for any }r\in W%
\text{ and all }s\in \Delta (A).
\end{equation*}%
This shows the continuity at $r_{0}$.
\end{proof}

\begin{corollary}
\label{co3}Let $A$ be an algebra with mean value. Then $\Delta (A)$ is a
group if and only if $A$ is a closed subalgebra of the almost periodic
functions.
\end{corollary}

\begin{proof}
It is classically known that if $A$ is a closed subalgebra of the algebra of
almost periodic functions, then $\Delta (A)$ is a topological group.
Conversely, assuming $\Delta (A)$ to be a group, the multiplication is
jointly continuous and so, by part (ii) of Theorem \ref{th3}, $A$ is a
closed subalgebra of the algebra of almost periodic functions.
\end{proof}

\begin{remark}
\label{re3}\emph{Let }$A=\mathcal{B}_{\infty }(\mathbb{R}^{N})$\emph{, the
algebra of those function in BUC}$(\mathbb{R}^{N})$\emph{\ that converge at
infinity. Then the multiplication in }$\Delta (A)$\emph{\ is not jointly
continuous. Indeed, if we denote by }$\mathcal{C}_{0}(\mathbb{R}^{N})$\emph{%
\ the Banach space of functions vanishing at infinity, then }$\mathcal{C}%
_{0}(\mathbb{R}^{N})\subset A$\emph{. Now, let }$y_{n}=(n,...,n)\in \mathbb{N%
}^{N}$\emph{\ (}$\mathbb{N}$\emph{\ the nonnegative integers). Take }$\mu
_{n}=\delta _{y_{n}}$\emph{\ the Dirac mass at }$y_{n}$\emph{, and define }$%
\mu _{-n}=\delta _{-y_{n}}$\emph{. We have, for any }$f\in \mathcal{C}_{0}(%
\mathbb{R}^{N})$\emph{, }$\mu _{n}(f)=f(y_{n})\rightarrow 0$\emph{\ and }$%
\mu _{-n}(f)\rightarrow 0$\emph{\ as }$n\rightarrow \infty $\emph{, while }$%
\mu _{n}\mu _{-n}(f)=f(0)$\emph{. Assuming the multiplication to be
continuous, we must have }$f(0)=0$\emph{\ for each }$f\in \mathcal{C}_{0}(%
\mathbb{R}^{N})$\emph{, which is not true.}
\end{remark}

\begin{remark}
\label{re4}\emph{By virtue of [part (i) of] Theorem \ref{th3}, the
introversion of the algebra with mean value }$A$\emph{\ entails its
ergodicity since any subalgebra of the algebra of weakly almost periodic
functions is ergodic; see e.g. \cite{NA, CMP}. We assume here that }$N=1$%
\emph{. Let }$A$\emph{\ be the algebra generated by the function }$f(y)=\cos 
\sqrt[3]{y}$\emph{\ (}$y\in \mathbb{R}$\emph{) and all its translates }$%
f(\cdot +a)$\emph{, }$a\in \mathbb{R}$\emph{. It is known that }$A$\emph{\
is an algebra with mean value which is not ergodic; see \cite[p. 243]{Jikov}
for details. It follows from Theorem \ref{th3} that }$A$\emph{\ is not
introverted.}
\end{remark}

Let $A$ be an introverted algebra with mean value. For any $s\in \Delta (A)$%
, $s\Delta (A)$ is an ideal of $\Delta (A)$ in the sense that for any $r\in
\Delta (A)$ and each $\mu \in \Delta (A)$, $r(s\mu )\in \Delta (A)$ since $%
r(s\mu )=sr\mu $. Now, set 
\begin{equation*}
K(\Delta (A))=\underset{s\in \Delta (A)}{\cap }s\Delta (A).
\end{equation*}%
Then $K(\Delta (A))$ is not empty In fact $sr\Delta (A)\subset s\Delta
(A)\cap r\Delta (A)$, so that the family $\{s\Delta (A):s\in \Delta (A)\}$
has the finite intersection property while $s\Delta (A)$ is trivially
closed. Invoking the compactness of $\Delta (A)$ we get that $K(\Delta (A))$
is non empty $K(\Delta (A))$ is trivially the smallest ideal of $\Delta (A)$
and is called the \emph{kernel} of $\Delta (A)$.

The following result provides us with the structure of $K(\Delta (A))$.

\begin{theorem}
\label{th4}Let $A$ be an introverted algebra with mean value on $\mathbb{R}%
^{N}$. Then

\begin{itemize}
\item[(i)] $K(\Delta (A))$ is a compact topological group.

\item[(ii)] The mean value $M$ on $A$ can be identified as the Haar integral
over $K(\Delta (A))$.
\end{itemize}
\end{theorem}

\begin{proof}
(i) Set $I=K(\Delta (A))$. For $s\in \Delta (A)$, we have that $sI$ is an
ideal contained in $I$, so $sI=I$. Thus if $s\in I$, we have an element $e$
in $I$ such that $se=s$. We infer that $rse=rs$ for any $r\in I$. Since $%
sI=Is=I$, it follows that $e$ is an identity for $I$. On the other hand, it
follows from the equality\ $rI=I$ (for fixed $r$ in $I$) that there is a $%
\mu \in I$ for which $r\mu =e$, and hence, $I$ is a group. But $I$ is
compact as it is closed in the compact space $\Delta (A)$. Whence $I$ is a
compact topological group.

(ii) We know that for any $f\in A$, $M(f)=\int_{\Delta (A)}\widehat{f}d\beta 
$, $\beta $ being a regular Borel measure on $\Delta (A)$. Assuming that $%
\beta $ is not supported by the group $K(\Delta (A))$, there exists a
compact set $J\subset \Delta (A)$ disjoint from $K(\Delta (A))$ with $\beta
(J)>0$, hence $\beta (K(\Delta (A)))<1$ since $\beta (\Delta (A))=1$. By
Urysohn's lemma we can find a function $g$ in $\mathcal{C}(\Delta (A))$ such
that $g=1$ on $K(\Delta (A))$, $g=0$ on $J$ and $0\leq g\leq 1$. Set $f=%
\mathcal{G}^{-1}(g)\in A$; we have $M(f)=\int_{\Delta (A)}gd\beta <1$. But
for $s\in K(\Delta (A))$, $\tau _{s}g=1$ since $\tau _{s}g(r)=g(sr)=1$ for
any $r\in \Delta (A)$ since $rs\in K(\Delta (A))$ (recall that $K(\Delta (A))
$ is an ideal). Whence $1=\int_{\Delta (A)}d\beta =\int_{\Delta (A)}\tau
_{s}gd\beta =\int_{\Delta (A)}gd\beta $ (since $\beta $ is invariant by
translations; see Remark \ref{re1}$)=M(f)<1$. this contradicts the
assumption that $\beta $ is not supported by $K(\Delta (A))$. Finally, $%
\beta $ being invariant, it coincides with the Haar measure of $K(\Delta (A))
$.
\end{proof}

We shall henceforth consider the mean value as an integral over the kernel $%
K(\Delta (A))$ of $\Delta (A)$ whenever the algebra $A$ is introverted. We
now have in hands all the ingredients necessary to the definition of the
convolution product on the spectrum $\Delta (A)$ of any introverted algebra
with mean value $A$. To do this, let $p,q,m\geq 1$ be real numbers
satisfying $\frac{1}{p}+\frac{1}{q}=1+\frac{1}{m}$. For $u\in L^{p}(\Delta
(A))$ and $v\in L^{q}(\Delta (A))$ we define the convolution product $u%
\widehat{\ast }v$ as follows: 
\begin{equation*}
(u\widehat{\ast }v)(s)=\int_{K(\Delta (A))}u(r)v(sr^{-1})d\beta (r)\text{, \
a.e. }s\in \Delta (A).
\end{equation*}%
Then $\widehat{\ast }$ is well defined since $K(\Delta (A))$ is an ideal of $%
\Delta (A)$. Indeed for $s\in \Delta (A)$ and $r\in K(\Delta (A))$, $r^{-1}$
exists in $K(\Delta (A))$ and $sr^{-1}\in K(\Delta (A))$. from the
definition of $\widehat{\ast }$ it holds that $u\widehat{\ast }v\in
L^{m}(\Delta (A))$ and the following Young inequality holds true: 
\begin{equation*}
\left\Vert u\widehat{\ast }v\right\Vert _{L^{m}(\Delta (A))}\leq \left\Vert
u\right\Vert _{L^{p}(\Delta (A))}\left\Vert v\right\Vert _{L^{q}(\Delta
(A))}.
\end{equation*}%
Now let $u\in L^{p}(\mathbb{R}^{N};L^{p}(\Delta (A)))$ and $v\in L^{q}(%
\mathbb{R}^{N};L^{q}(\Delta (A)))$. We define the double convolution $u%
\widehat{\ast \ast }v$ as follows: 
\begin{eqnarray*}
(u\widehat{\ast \ast }v)(x,s) &=&\int_{\mathbb{R}^{N}}\left[ \left(
u(t,\cdot )\widehat{\ast }v(x-t,\cdot )\right) (s)\right] dt \\
&\equiv &\int_{\mathbb{R}^{N}}\int_{K(\Delta (A))}u(t,r)\widehat{\ast }%
v(x-t,sr^{-1})d\beta (r)dt\text{, a.e. }(x,s)\in \mathbb{R}^{N}\times \Delta
(A).
\end{eqnarray*}%
Then $\widehat{\ast \ast }$ is well defined as an element of $L^{m}(\mathbb{R%
}^{N}\times \Delta (A))$ and satisfies 
\begin{equation*}
\left\Vert u\widehat{\ast \ast }v\right\Vert _{L^{m}(\mathbb{R}^{N}\times
\Delta (A))}\leq \left\Vert u\right\Vert _{L^{p}(\mathbb{R}^{N}\times \Delta
(A))}\left\Vert v\right\Vert _{L^{q}(\mathbb{R}^{N}\times \Delta (A))}.
\end{equation*}%
It is to be noted that if $u\in L^{p}(\Omega ;L^{p}(\Delta (A)))$ where $%
\Omega $ is an open subset of $\mathbb{R}^{N}$, and $v\in L^{q}(\mathbb{R}%
^{N};L^{q}(\Delta (A)))$, we may still define $u\widehat{\ast \ast }v$ by
viewing $u$ as defined in the whole of $\mathbb{R}^{N}\times \Delta (A)$; it
suffices to take the extension by zero of $u$ outside $\Omega \times \Delta
(A)$.

Finally, for $u\in L^{p}(\mathbb{R}^{N};\mathcal{B}_{A}^{p})$ and $v\in
L^{q}(\mathbb{R}^{N};\mathcal{B}_{A}^{q})$ we define the double convolution
denoted by $\ast \ast $ as follows: $u\ast \ast v$ is that element of $L^{m}(%
\mathbb{R}^{N};\mathcal{B}_{A}^{m})$ defined by 
\begin{equation*}
\mathcal{G}_{1}(u\ast \ast v)=\widehat{u}\widehat{\ast \ast }\widehat{v}.
\end{equation*}%
It also satisfies the Young inequality.

For the next result, we need to define a dynamical system on $\Delta (A)$.
We equip $\Delta (A)$ with the $\sigma $-algebra $\mathcal{B}(\Delta (A))$
of Borelians of $\Delta (A)$ which makes $(\Delta (A),\mathcal{B}(\Delta
(A)),\beta )$ a probability space. For each fixed $x\in \mathbb{R}^{N}$, let
the mapping $\mathcal{T}(x):\Delta (A)\rightarrow \Delta (A)$ be defined by $%
\mathcal{T}(x)s=\delta _{x}s$, $s\in \Delta (A)$. Then the family $\mathcal{T%
}=\{\mathcal{T}(x):x\in \mathbb{R}^{N}\}$ defines a continuous $N$%
-dimensional dynamical system in the following sense:

\begin{itemize}
\item[(i)] (\textit{Group property}) $\mathcal{T}\left( 0\right) =id_{\Delta
(A)}$ and $\mathcal{T}\left( x+y\right) =\mathcal{T}(x)\mathcal{T}(y)$ for
all $x,y\in \mathbb{R}^{N}$;

\item[(ii)] (\textit{Invariance}) The mappings $\mathcal{T}\left( x\right)
:\Delta (A)\rightarrow \Delta (A)$ are measurable and $\beta $-measure
preserving, i.e., $\beta \left( T\left( x\right) F\right) =\beta \left(
F\right) $ for each $x\in \mathbb{R}^{N}$ and every $F\in \mathcal{B}(\Delta
(A))$;

\item[(iii)] (\textit{Continuity}) The mapping $(x,s)\mapsto \mathcal{T}(x)s$
is continuous from $\mathbb{R}^{N}\times \Delta (A)$ into $\Delta (A)$.
\end{itemize}

The next result will be of a very first importance in the following
sections. It is new and constitutes the cornerstone of the connection
between the convolution and the $\Sigma $-convergence method.

\begin{theorem}
\label{th5}Let $A$ be an algebra with mean value on $\mathbb{R}^{N}$.
Suppose that $A$ is introverted. Then denoting by $\delta _{y}$ the Dirac
mass at $y$, it holds that $\delta _{y}\in K(\Delta (A))$ for almost all $%
y\in \mathbb{R}^{N}$.
\end{theorem}

The proof of this result heavily relies on the following lemma whose proof
can be found in \cite{Jikov}.

\begin{lemma}[{\protect\cite[Lemma 7.1, p. 224]{Jikov}}]
\label{le3}Let $\Omega _{0}$ be a set of full measure in $\Delta (A)$. Then
there exists a set of full measure $\Omega _{1}\subset \Omega _{0}$ such
that for a given $s\in \Omega _{1}$, we have $\mathcal{T}(x)s\in \Omega _{0}$
for almost all $x\in \mathbb{R}^{N}$.
\end{lemma}

\begin{proof}[Proof of Theorem \protect\ref{th5}]
We infer from Theorem \ref{th4} that $\int_{K(\Delta (A))}d\beta =1$, i.e., $%
K(\Delta (A))$ is a set of full measure in $\Delta (A)$. Therefore applying
Lemma \ref{le3} with $\Omega _{0}=K(\Delta (A))$ (the kernel of $\Delta (A)$%
) we derive the existence of a set $\Omega _{1}\subset K(\Delta (A))$ of
full $\beta $-measure such that, for a given $s\in \Omega _{1}$, $\delta
_{y}s\in K(\Delta (A))$ for almost all $y\in \mathbb{R}^{N}$. But, since $%
\Omega _{1}\subset K(\Delta (A))$, any element of $\Omega _{1}$ is
invertible. Hence, denoting by $s^{-1}$ the inverse of $s$ in $K(\Delta (A))$
we have that $\delta _{y}=(\delta _{y}s)s^{-1}\in K(\Delta
(A))s^{-1}=K(\Delta (A))$ for almost all $y\in \mathbb{R}^{N}$.
\end{proof}

\section{On a De Rham type result}

In this section, we assume $A$ to be an algebra with mean value on $\mathbb{R%
}^{N}$ as defined in Section 2. Let $u\in A$ and let $\varphi \in \mathcal{C}%
_{0}^{\infty }(\mathbb{R}^{N})=\mathcal{D}(\mathbb{R}^{N})$. Since $u$ and $%
\varphi $ are uniformly continuous and $A$ is translation invariant, we have 
$u\ast \varphi \in A$ (see the proof of Proposition 2.3 in \cite{ACAP}),
where here $\ast $ stands for the usual convolution operator. More
precisely, $u\ast \varphi \in A^{\infty }$ since $D_{y}^{\alpha }(u\ast
\varphi )=u\ast D_{y}^{\alpha }\varphi $ for any $\alpha \in \mathbb{N}^{N}$%
. For $1\leq p<\infty $, let $u\in B_{A}^{p}$ and $\eta >0$, and choose $%
v\in A$ such that $\left\Vert u-v\right\Vert _{p}<\eta /(\left\Vert \varphi
\right\Vert _{L^{1}(\mathbb{R}^{N})}+1)$. Using Young's inequality, we have 
\begin{equation*}
\left\Vert u\ast \varphi -v\ast \varphi \right\Vert _{p}\leq \left\Vert
\varphi \right\Vert _{L^{1}(\mathbb{R}^{N})}\left\Vert u-v\right\Vert
_{2}<\eta ,
\end{equation*}%
hence $u\ast \varphi \in B_{A}^{p}$ since $v\ast \varphi \in A$. We may
therefore define the convolution between $\mathcal{B}_{A}^{p}$ and $\mathcal{%
C}_{0}^{\infty }(\mathbb{R}^{N})$. Indeed, for $\mathfrak{u}=u+\mathcal{N}%
\in \mathcal{B}_{A}^{p}(\mathbb{R}^{N})$ (with $u\in B_{A}^{p}$) and $%
\varphi \in \mathcal{C}_{0}^{\infty }(\mathbb{R}^{N})$, we define $\mathfrak{%
u}\circledast \varphi $ as follows 
\begin{equation}
\mathfrak{u}\circledast \varphi :=u\ast \varphi +\mathcal{N}\equiv \varrho
(u\ast \varphi ).  \label{1.11}
\end{equation}%
Indeed, this is well defined as justified below by (\ref{0}). Thus, for $%
\mathfrak{u}\in \mathcal{B}_{A}^{p}$ and $\varphi \in \mathcal{C}%
_{0}^{\infty }(\mathbb{R}^{N})$ we have $\mathfrak{u}\circledast \varphi \in 
\mathcal{B}_{A}^{p}$ with 
\begin{equation}
\overline{D}_{y}^{\alpha }(\mathfrak{u}\circledast \varphi )=\varrho (u\ast
D_{y}^{\alpha }\varphi )\text{, all }\alpha \in \mathbb{N}^{N}.  \label{1}
\end{equation}%
We deduce from (\ref{1}) that $\mathfrak{u}\circledast \varphi \in \mathcal{D%
}_{A}(\mathbb{R}^{N})$ since $u\ast \varphi \in A^{\infty }$. Moreover, we
have 
\begin{equation}
\left\Vert \mathfrak{u}\circledast \varphi \right\Vert _{p}\leq \left\vert 
\text{Supp}\varphi \right\vert ^{\frac{1}{p}}\left\Vert \varphi \right\Vert
_{L^{p^{\prime }}(\mathbb{R}^{N})}\left\Vert \mathfrak{u}\right\Vert _{p},
\label{2}
\end{equation}%
where Supp$\varphi $ stands for the support of $\varphi $ and $\left\vert 
\text{Supp}\varphi \right\vert $ its Lebesgue measure. Indeed, we have 
\begin{equation*}
\left\Vert \mathfrak{u}\circledast \varphi \right\Vert _{p}=\left\Vert
\varrho (u\ast \varphi )\right\Vert _{p}=\left( \underset{r\rightarrow
+\infty }{\lim \sup }\left\vert B_{r}\right\vert ^{-1}\int_{B_{r}}\left\vert
(u\ast \varphi )(y)\right\vert ^{p}dy\right) ^{\frac{1}{p}},
\end{equation*}%
and 
\begin{eqnarray*}
\int_{B_{r}}\left\vert (u\ast \varphi )(y)\right\vert ^{p}dy &\leq &\left(
\int_{B_{r}}\left\vert \varphi \right\vert dy\right) ^{p}\left(
\int_{B_{r}}\left\vert u(y)\right\vert ^{p}dy\right)  \\
&\leq &\left\vert B_{r}\cap \text{Supp}\varphi \right\vert \left\Vert
\varphi \right\Vert _{L^{p^{\prime }}(B_{r})}^{p}\int_{B_{r}}\left\vert
u(y)\right\vert ^{p}dy,
\end{eqnarray*}%
hence (\ref{2}). For $u\in A$ and $\varphi \in \mathcal{C}_{0}^{\infty }(%
\mathbb{R}^{N})$ the convolution $\widehat{u}\circledast \varphi $ is
defined as follows 
\begin{equation}
\left( \widehat{u}\circledast \varphi \right) (s)=\int_{\mathbb{R}^{N}}%
\widehat{\tau _{y}u}(s)\varphi (y)dy\text{\ \ \ \ (}s\in \Delta (A)\text{),}
\label{a}
\end{equation}%
where $\widehat{u}=\mathcal{G}(u)$ and $\tau _{y}u=u(\cdot +y)$. It is
easily seen that $\widehat{u}\circledast \varphi \in \mathcal{C}(\Delta (A))$%
. We have 
\begin{equation}
\widehat{u\ast \varphi }=\widehat{u}\circledast \varphi \text{ for all }u\in
A\text{ and }\varphi \in \mathcal{C}_{0}^{\infty }(\mathbb{R}^{N}).
\label{0}
\end{equation}%
Indeed, for $x\in \mathbb{R}^{N}$, we have 
\begin{eqnarray*}
\left( \widehat{u}\circledast \varphi \right) (\delta _{x}) &=&\int_{\mathbb{%
R}^{N}}\widehat{\tau _{y}u}(\delta _{x})\varphi (y)dy=\int_{\mathbb{R}%
^{N}}\tau _{y}u(x)\varphi (y)dy \\
&=&\left( u\ast \varphi \right) (x)=\widehat{u\ast \varphi }(\delta _{x}),
\end{eqnarray*}%
and (\ref{0}) follows by the continuity of both $\widehat{u}\circledast
\varphi $ and $\widehat{u\ast \varphi }$, and the denseness of $\{\delta
_{x}:x\in \mathbb{R}^{N}\}$ in $\Delta (A)$. As claimed above, (\ref{0})
justifies that $\mathfrak{u}\circledast \varphi $ is well-defined by (\ref%
{1.11}) for $\mathfrak{u}\in \mathcal{B}_{A}^{p}(\mathbb{R}^{N})$ and $%
\varphi \in \mathcal{C}_{0}^{\infty }(\mathbb{R}^{N})$. Indeed, for $u,v\in 
\mathfrak{u}$, we have $u,v\in B_{A}^{p}$ with $\widehat{u}=\widehat{v}$ and
so $\mathfrak{u}=u+\mathcal{N}=v+\mathcal{N}$. It emerges from (\ref{0})
that $\widehat{u\ast \varphi }=\widehat{u}\ast \varphi =\widehat{v}\ast
\varphi =\widehat{v\ast \varphi }$, hence $u\ast \varphi +\mathcal{N}=v\ast
\varphi +\mathcal{N}$.

We also have the obvious equality 
\begin{equation}
\partial _{i}(\widehat{u}\circledast \varphi )=\widehat{u}\circledast \frac{%
\partial \varphi }{\partial y_{i}}\text{ for all }1\leq i\leq N\text{.}
\label{b}
\end{equation}

The following De Rham type result holds.

\begin{theorem}
\label{t1}Let $1<p<\infty $. Let $L$ be a bounded linear functional on $(%
\mathcal{B}_{A}^{1,p^{\prime }})^{N}$ which vanishes on the kernel of the
divergence. Then there exists a function $f\in \mathcal{B}_{A}^{p}$ such
that $L=\overline{\nabla }_{y}f$, i.e., 
\begin{equation*}
L(v)=-\int_{\Delta (A)}\widehat{f}~\widehat{\Div}\widehat{v}d\beta \text{
for all }v\in (\mathcal{B}_{A}^{1,p^{\prime }})^{N}\text{.}
\end{equation*}%
Moreover $f$ is unique modulo $I_{A}^{p}$, that is, up to an additive
function $g\in \mathcal{B}_{A}^{p}$ verifying $\overline{\nabla }_{y}g=0$.
\end{theorem}

\begin{proof}
Let $u\in A^{\infty }$ (hence $\varrho (u)\in \mathcal{D}_{A}(\mathbb{R}%
^{N}) $). Define $L_{u}:\mathcal{D}(\mathbb{R}^{N})^{N}\rightarrow \mathbb{R}
$ by 
\begin{equation*}
L_{u}(\varphi )=L(\varrho (u\ast \varphi ))\text{ for }\varphi =(\varphi
_{i})\in \mathcal{D}(\mathbb{R}^{N})^{N}
\end{equation*}%
where $u\ast \varphi =(u\ast \varphi _{i})_{i}\in (A^{\infty })^{N}$. Then $%
L_{u}$ defines a distribution on $\mathcal{D}(\mathbb{R}^{N})^{N}$. Moreover
if $\Div_{y}\varphi =0$ then $\overline{\Div}_{y}(\varrho (u\ast \varphi
))=\varrho (u\ast \Div_{y}\varphi )=0$, hence $L_{u}(\varphi )=0$, that is, $%
L_{u}$ vanishes on the kernel of the divergence in $\mathcal{D}(\mathbb{R}%
^{N})^{N}$. By the De Rham theorem, there exists a distribution $S(u)\in 
\mathcal{D}^{\prime }(\mathbb{R}^{N})$ such that $L_{u}=\nabla _{y}S(u)$.
This defines an operator 
\begin{equation*}
S:A^{\infty }\rightarrow \mathcal{D}^{\prime }(\mathbb{R}^{N});\ u\mapsto
S(u)
\end{equation*}%
satisfying the following properties:

\begin{itemize}
\item[(i)] $S(\tau _{y}u)=\tau _{y}S(u)$ for all $y\in \mathbb{R}^{N}$ and
all $u\in A^{\infty }$;

\item[(ii)] $S$ maps linearly and continuously $A^{\infty }$ into $L_{\text{%
loc}}^{p^{\prime }}(\mathbb{R}^{N})$;

\item[(iii)] There is a positive constant $C_{r}$ (that is locally bounded
as a function of $r$) such that 
\begin{equation*}
\left\Vert S(u)\right\Vert _{L^{p^{\prime }}(B_{r})}\leq C_{r}\left\Vert
L\right\Vert \left\vert B_{r}\right\vert ^{\frac{1}{p^{\prime }}}\left\Vert
\varrho (u)\right\Vert _{p^{\prime }}.
\end{equation*}
\end{itemize}

The property (i) easily comes from the obvious equality 
\begin{equation*}
L_{\tau _{y}u}(\varphi )=L_{u}(\tau _{y}\varphi )\ \ \forall y\in \mathbb{R}%
^{N}.
\end{equation*}%
Let us check (ii) and (iii). For that, let $\varphi \in \mathcal{D}(\mathbb{R%
}^{N})^{N}$ with Supp$\varphi _{i}\subset B_{r}$ for all $1\leq i\leq N$.
Then 
\begin{eqnarray*}
\left\vert L_{u}(\varphi )\right\vert &=&\left\vert L(\varrho (u\ast \varphi
))\right\vert \\
&\leq &\left\Vert L\right\Vert \left\Vert \varrho (u)\circledast \varphi
\right\Vert _{(\mathcal{B}_{A}^{1,p^{\prime }})^{N}} \\
&\leq &\max_{1\leq i\leq N}\left\vert \text{Supp}\varphi _{i}\right\vert ^{%
\frac{1}{p^{\prime }}}\left\Vert L\right\Vert \left\Vert \varrho
(u)\right\Vert _{p^{\prime }}\left\Vert \varphi \right\Vert
_{W^{1,p}(B_{r})^{N}},
\end{eqnarray*}%
the last inequality being due to (\ref{2}). Hence, as Supp$\varphi
_{i}\subset B_{r}$ ($1\leq i\leq N$), 
\begin{equation}
\left\Vert L_{u}\right\Vert _{W^{-1,p^{\prime }}(B_{r})^{N}}\leq \left\Vert
L\right\Vert \left\vert B_{r}\right\vert ^{\frac{1}{p^{\prime }}}\left\Vert
\varrho (u)\right\Vert _{p^{\prime }}.  \label{3}
\end{equation}%
Now, let $g\in \mathcal{C}_{0}^{\infty }(B_{r})$ with $\int_{B_{r}}gdy=0$;
then by \cite[Lemma 3.15]{Novotny} there exists $\varphi \in \mathcal{C}%
_{0}^{\infty }(B_{r})^{N}$ such that $\Div\varphi =g$ and $\left\Vert
\varphi \right\Vert _{W^{1,p}(B_{r})^{N}}\leq C(p,B_{r})\left\Vert
g\right\Vert _{L^{p}(B_{r})}$. We have 
\begin{eqnarray*}
\left\vert \left\langle S(u),g\right\rangle \right\vert &=&\left\vert
-\left\langle \nabla _{y}S(u),\varphi \right\rangle \right\vert =\left\vert
\left\langle L_{u},\varphi \right\rangle \right\vert \\
&\leq &\left\Vert L_{u}\right\Vert _{W^{-1,p^{\prime
}}(B_{r})^{N}}\left\Vert \varphi \right\Vert _{W^{1,p}(B_{r})^{N}} \\
&\leq &C(p,B_{r})\left\Vert L\right\Vert \left\vert B_{r}\right\vert ^{\frac{%
1}{p^{\prime }}}\left\Vert \varrho (u)\right\Vert _{p^{\prime }}\left\Vert
g\right\Vert _{L^{p}(B_{r})},
\end{eqnarray*}%
and by a density argument, we get that $S(u)\in (L^{p}(B_{r})/\mathbb{R}%
)^{\prime }=L^{p^{\prime }}(B_{r})/\mathbb{R}$ for any $r>0$, where $%
L^{p^{\prime }}(B_{r})/\mathbb{R}=\{\psi \in L^{p^{\prime
}}(B_{r}):\int_{B_{r}}\psi dy=0\}$. The properties (ii) and (iii) therefore
follow from the above series of inequalities. Taking (ii) as granted it
comes that 
\begin{equation}
L_{u}(\varphi )=-\int_{\mathbb{R}^{N}}S(u)\Div_{y}\varphi dy\text{ for all }%
\varphi \in \mathcal{D}(\mathbb{R}^{N})^{N}.  \label{4}
\end{equation}%
We claim that $S(u)\in \mathcal{C}^{\infty }(\mathbb{R}^{N})$ for all $u\in
A^{\infty }$. Indeed let $e_{i}=(\delta _{ij})_{1\leq j\leq N}$ ($\delta
_{ij}$ the Kronecker delta). Then owing to (i) and (iii) above, we have 
\begin{eqnarray*}
\left\Vert t^{-1}(\tau _{te_{i}}S(u)-S(u))-S\left( \frac{\partial u}{%
\partial y_{i}}\right) \right\Vert _{L^{p^{\prime }}(B_{r})} &=&\left\Vert
S\left( t^{-1}(\tau _{te_{i}}u-u)-\frac{\partial u}{\partial y_{i}}\right)
\right\Vert _{L^{p^{\prime }}(B_{r})} \\
&\leq &c\left\Vert t^{-1}(\varrho (\tau _{te_{i}}u-u))-\varrho \left( \frac{%
\partial u}{\partial y_{i}}\right) \right\Vert _{p^{\prime }}.
\end{eqnarray*}%
Hence, passing to the limit as $t\rightarrow 0$ above leads us to 
\begin{equation*}
\frac{\partial }{\partial y_{i}}S(u)=S\left( \frac{\partial u}{\partial y_{i}%
}\right) \text{ for all }1\leq i\leq N.
\end{equation*}%
Repeating the same process we end up with 
\begin{equation*}
D_{y}^{\alpha }S(u)=S(D_{y}^{\alpha }u)\text{ for all }\alpha \in \mathbb{N}%
^{N}\text{.}
\end{equation*}%
So all the weak derivative of $S(u)$ of any order belong to $L_{\text{loc}%
}^{p^{\prime }}(\mathbb{R}^{N})$. Our claim is therefore a consequence of 
\cite[Theorem XIX, p. 191]{LS}.

This being so, we derive from the mean value theorem the existence of $\xi
\in B_{r}$ such that 
\begin{equation*}
S(u)(\xi )=\left\vert B_{r}\right\vert ^{-1}\int_{B_{r}}S(u)dy.
\end{equation*}%
On the other hand, the map $u\mapsto S(u)(0)$ is a linear functional on $%
A^{\infty }$, and by the above equality we get 
\begin{eqnarray*}
\left\vert S(u)(0)\right\vert  &\leq &~\underset{r\rightarrow 0}{\lim \sup }%
\left\vert B_{r}\right\vert ^{-1}\int_{B_{r}}\left\vert S(u)\right\vert dy \\
&\leq &~\underset{r\rightarrow 0}{\lim \sup }\left\vert B_{r}\right\vert ^{-%
\frac{1}{p^{\prime }}}\left( \int_{B_{r}}\left\vert S(u)\right\vert
^{p^{\prime }}dy\right) ^{\frac{1}{p^{\prime }}} \\
&\leq &c\left\Vert L\right\Vert \left\Vert \varrho (u)\right\Vert
_{p^{\prime }}.
\end{eqnarray*}%
Hence, defining $\widetilde{S}:\mathcal{D}_{A}(\mathbb{R}^{N})\rightarrow 
\mathbb{R}$ by $\widetilde{S}(v)=S(u)(0)$ for $v=\varrho (u)$ with $u\in
A^{\infty }$, we get that $\widetilde{S}$ is a linear functional on $%
\mathcal{D}_{A}(\mathbb{R}^{N})$ satisfying 
\begin{equation}
\left\vert \widetilde{S}(v)\right\vert \leq c\left\Vert L\right\Vert
\left\Vert v\right\Vert _{p^{\prime }}\ \ \forall v\in \mathcal{D}_{A}(%
\mathbb{R}^{N})\text{.}  \label{5}
\end{equation}%
We infer from both the density of $\mathcal{D}_{A}(\mathbb{R}^{N})$ in $%
\mathcal{B}_{A}^{p^{\prime }}$ and (\ref{5}) the \ existence of a function $%
f\in \mathcal{B}_{A}^{p}$ with $\left\Vert f\right\Vert _{p}\leq c\left\Vert
L\right\Vert $ such that 
\begin{equation*}
\widetilde{S}(v)=\int_{\Delta (A)}\widehat{f}\widehat{v}d\beta \text{ for
all }v\in \mathcal{B}_{A}^{p^{\prime }}\text{.}
\end{equation*}%
In particular 
\begin{equation*}
S(u)(0)=\int_{\Delta (A)}\widehat{f}\widehat{u}d\beta \ \ \forall u\in
A^{\infty }
\end{equation*}%
where $\widehat{u}=\mathcal{G}(u)=\mathcal{G}_{1}(\varrho (u))$. Now, let $%
u\in A^{\infty }$ and let $y\in \mathbb{R}^{N}$. By (i) we have 
\begin{equation*}
S(u)(y)=S(\tau _{y}u)(0)=\int_{\Delta (A)}\widehat{\tau _{y}u}\widehat{f}%
d\beta \text{.}
\end{equation*}%
Thus 
\begin{eqnarray*}
L_{u}(\varphi ) &=&L(\varrho (u\ast \varphi ))=-\int_{\mathbb{R}^{N}}S(u)(y)%
\Div_{y}\varphi dy\text{ \ (by (\ref{4}))} \\
&=&-\int_{\mathbb{R}^{N}}\left( \int_{\Delta (A)}\widehat{\tau _{y}u}%
\widehat{f}d\beta \right) \Div_{y}\varphi dy \\
&=&-\int_{\Delta (A)}\left( \int_{\mathbb{R}^{N}}\widehat{\tau _{y}u}(s)\Div%
_{y}\varphi dy\right) \widehat{f}d\beta  \\
&=&-\int_{\Delta (A)}\widehat{f}(\widehat{u}\circledast \Div_{y}\varphi
)d\beta \text{ \ (by (\ref{a}))} \\
&=&-\int_{\Delta (A)}\widehat{f}~\mathcal{G}\left( u\ast \Div_{y}\varphi
\right) d\beta \text{ \ (by (\ref{0}))} \\
&=&-\int_{\Delta (A)}\widehat{f}~\mathcal{G}\left( \Div_{y}(u\ast \varphi
)\right) d\beta  \\
&=&-\int_{\Delta (A)}\widehat{f}~\mathcal{G}_{1}\left( \overline{\Div}%
_{y}(\varrho (u\ast \varphi ))\right) d\beta  \\
&=&\left\langle \overline{\nabla }_{y}f,\varrho (u\ast \varphi
)\right\rangle .
\end{eqnarray*}%
Finally let $v\in (\mathcal{B}_{A}^{1,p^{\prime }})^{N}$ and let $(\varphi
_{n})_{n}\subset \mathcal{D}(\mathbb{R}^{N})$ be a mollifier. Then $%
v\circledast \varphi _{n}\rightarrow v$ in $(\mathcal{B}_{A}^{1,p^{\prime
}})^{N}$ as $n\rightarrow \infty $, where $v\circledast \varphi
_{n}=(v_{i}\circledast \varphi _{n})_{i}$. We have $v\circledast \varphi
_{n}\in \mathcal{D}_{A}(\mathbb{R}^{N})^{N}$ and $L(v\circledast \varphi
_{n})\rightarrow L(v)$ by the continuity of $L$. On the other hand 
\begin{equation*}
\int_{\Delta (A)}\widehat{f}~\mathcal{G}_{1}\left( \overline{\Div}%
_{y}(v\circledast \varphi _{n})\right) d\beta \rightarrow \int_{\Delta (A)}%
\widehat{f}~\widehat{\Div}\widehat{v}d\beta .
\end{equation*}%
We deduce that $L$ and $\overline{\nabla }_{y}f$ agree on $(\mathcal{B}%
_{A}^{1,p^{\prime }})^{N}$, i.e., $L=\overline{\nabla }_{y}f$.

For the uniqueness, let $f_{1}$ and $f_{2}$ in $\mathcal{B}_{A}^{p}$ be such
that $L=\overline{\nabla }_{y}f_{1}=\overline{\nabla }_{y}f_{2}$, then $%
\overline{\nabla }_{y}(f_{1}-f_{2})=0$, which means that $f_{1}-f_{2}\in
I_{A}^{p}$.
\end{proof}

The preceding result together with its proof are still valid mutatis
mutandis when the function spaces are complex-valued. In that case, we only
require the algebra $A$ to be closed under complex conjugation ($\overline{u}%
\in A$ whenever $u\in A$). As a result of the preceding theorem, we have the

\begin{corollary}
\label{c1}Let $f\in (\mathcal{B}_{A}^{p})^{N}$ be such that 
\begin{equation*}
\int_{\Delta (A)}\widehat{f}\cdot \widehat{g}d\beta =0\text{ }\forall g\in 
\mathcal{D}_{A}(\mathbb{R}^{N})^{N}\text{ with }\overline{\Div}_{y}g=0.
\end{equation*}%
Then there exists a function $u\in \mathcal{B}_{A}^{1,p}$, uniquely
determined modulo $I_{A}^{p}$, such that $f=\overline{\nabla }_{y}u$.
\end{corollary}

\begin{proof}
Define $L:(\mathcal{B}_{A}^{1,p^{\prime }})^{N}\rightarrow \mathbb{R}$ by $%
L(v)=\int_{\Delta (A)}\widehat{f}\cdot \widehat{v}d\beta $. Then $L$ lies in 
$\left[ (\mathcal{B}_{A}^{1,p^{\prime }})^{N}\right] ^{\prime }$, and it
follows from Theorem \ref{t1} the existence of $u\in \mathcal{B}_{A}^{p}$
such that $f=\overline{\nabla }_{y}u$. This shows at once that $u\in 
\mathcal{B}_{A}^{1,p}$. The uniqueness is shown as in Theorem \ref{t1}.
\end{proof}

\begin{remark}
\label{r1}\emph{Let }$u\in \mathcal{B}_{A}^{p}$\emph{\ be such that }$%
\overline{\nabla }_{y}u=0$\emph{; then }$u\in I_{A}^{p}$\emph{. Thus, for
the uniqueness argument, we may choose the function }$u$\emph{\ in Corollary %
\ref{c1} to belong to }$\mathcal{B}_{A}^{1,p}/I_{A}^{p}$\emph{, which space
we shall henceforth equip with the norm gradient norm as above; see (\ref%
{2.5'}).}
\end{remark}

\section{Sigma convergence method}

Throughout this section, $\Omega $ is an open subset of $\mathbb{R}^{N}$,
and unless otherwise specified, $A$ is an algebra with mean value on $%
\mathbb{R}^{N}$.

\begin{definition}
\label{d2.1}\emph{(1) A sequence }$\left( u_{\varepsilon }\right)
_{\varepsilon >0}\subset L^{p}\left( \Omega \right) $\emph{\ }$(1\leq
p<\infty )$\emph{\ is said to }weakly $\Sigma $-converge\emph{\ in }$%
L^{p}\left( \Omega \right) $\emph{\ to some }$u_{0}\in L^{p}(\Omega ;%
\mathcal{B}_{A}^{p})$\emph{\ if as }$\varepsilon \rightarrow 0$\emph{, }%
\begin{equation}
\int_{\Omega }u_{\varepsilon }\left( x\right) \psi ^{\varepsilon }\left(
x\right) dx\rightarrow \iint_{\Omega \times \Delta (A)}\widehat{u}_{0}\left(
x,s\right) \widehat{\psi }\left( x,s\right) dxd\beta \left( s\right)
\label{2.3'}
\end{equation}%
\emph{for all }$\psi \in L^{p^{\prime }}\left( \Omega ;A\right) $\emph{\ }$%
\left( 1/p^{\prime }=1-1/p\right) $\emph{\ where }$\psi ^{\varepsilon
}\left( x\right) =\psi \left( x,x/\varepsilon \right) $\emph{\ and }$%
\widehat{\psi }\left( x,\cdot \right) =\mathcal{G}(\psi \left( x,\cdot
\right) )$\emph{\ }$a.e.$\emph{\ in }$x\in \Omega $\emph{. We denote this by 
}$u_{\varepsilon }\rightarrow u_{0}$\emph{\ in }$L^{p}(\Omega )$\emph{-weak }%
$\Sigma $\emph{.}

\noindent \emph{(2) A sequence }$(u_{\varepsilon })_{\varepsilon >0}\subset
L^{p}(\Omega )$\emph{\ }$(1\leq p<\infty )$\emph{\ is said to }strongly $%
\Sigma $-converge\emph{\ in }$L^{p}(\Omega )$\emph{\ to some }$u_{0}\in
L^{p}(\Omega ;\mathcal{B}_{A}^{p})$\emph{\ if it is weakly }$\Sigma $\emph{%
-convergent and further satisfies the following condition: }%
\begin{equation*}
\left\Vert u_{\varepsilon }\right\Vert _{L^{p}(\Omega )}\rightarrow
\left\Vert \widehat{u}_{0}\right\Vert _{L^{p}(\Omega \times \Delta (A))}.
\end{equation*}%
\emph{We denote this by }$u_{\varepsilon }\rightarrow u_{0}$\emph{\ in }$%
L^{p}(\Omega )$\emph{-strong }$\Sigma $\emph{.}
\end{definition}

We recall here that $\widehat{u}_{0}=\mathcal{G}_{1}\circ u_{0}$ and $%
\widehat{\psi }=\mathcal{G}\circ \psi $, $\mathcal{G}_{1}$ being the
isometric isomorphism sending $\mathcal{B}_{A}^{p}$ onto $L^{p}(\Delta (A))$
and $\mathcal{G}$, the Gelfand transformation on $A$.

In the sequel the letter $E$ will throughout denote a fundamental sequence,
that is, any ordinary sequence $E=(\varepsilon _{n})_{n}$ (integers $n\geq 0$%
) with $0<\varepsilon _{n}\leq 1$ and $\varepsilon _{n}\rightarrow 0$ as $%
n\rightarrow \infty $. The following result holds.

\begin{theorem}
\label{t2.2}\emph{(i)} Any bounded sequence $(u_{\varepsilon })_{\varepsilon
\in E}$ in $L^{p}(\Omega )$ (where $1<p<\infty $) admits a subsequence which
is weakly $\Sigma $-convergent in $L^{p}(\Omega )$.

\noindent \emph{(ii)} Any uniformly integrable sequence $(u_{\varepsilon
})_{\varepsilon \in E}$ in $L^{1}(\Omega )$ admits a subsequence which is
weakly $\Sigma $-convergent in $L^{1}(\Omega )$.
\end{theorem}

Below is one fundamental result involving the gradient of sequences.

\begin{theorem}
\label{t2.3}Let $1<p<\infty $. Let $(u_{\varepsilon })_{\varepsilon \in E}$
be a bounded sequence in $W^{1,p}(\Omega )$. Then there exist a subsequence $%
E^{\prime }$ of $E$, and a couple $(u_{0},u_{1})\in W^{1,p}(\Omega
;I_{A}^{p})\times L^{p}(\Omega ;\mathcal{B}_{A}^{1,p})$ such that, as $%
E^{\prime }\ni \varepsilon \rightarrow 0$, 
\begin{equation*}
u_{\varepsilon }\rightarrow u_{0}\ \text{in }L^{p}(\Omega )\text{-weak }%
\Sigma \text{;}
\end{equation*}%
\begin{equation*}
\frac{\partial u_{\varepsilon }}{\partial x_{i}}\rightarrow \frac{\partial
u_{0}}{\partial x_{i}}+\frac{\overline{\partial }u_{1}}{\partial y_{i}}\text{%
\ in }L^{p}(\Omega )\text{-weak }\Sigma \text{, }1\leq i\leq N.
\end{equation*}
\end{theorem}

\begin{proof}
Since the sequences $(u_{\varepsilon })_{\varepsilon \in E}$ and $(\nabla
u_{\varepsilon })_{\varepsilon \in E}$ are bounded respectively in $%
L^{p}(\Omega )$ and in $L^{p}(\Omega )^{N}$, there exist a subsequence $%
E^{\prime }$ of $E$ and $u_{0}\in L^{p}(\Omega ;\mathcal{B}_{A}^{p})$, $%
v=(v_{j})_{j}\in L^{p}(\Omega ;\mathcal{B}_{A}^{p})^{N}$ such that $%
u_{\varepsilon }\rightarrow u_{0}\ $in $L^{p}(\Omega )$-weak $\Sigma $ and $%
\frac{\partial u_{\varepsilon }}{\partial x_{j}}\rightarrow v_{j}$ in $%
L^{p}(\Omega )$-weak $\Sigma $. For $\Phi \in (\mathcal{C}_{0}^{\infty
}(\Omega )\otimes A^{\infty })^{N}$ we have 
\begin{equation*}
\int_{\Omega }\varepsilon \nabla u_{\varepsilon }\cdot \Phi ^{\varepsilon
}dx=-\int_{\Omega }\left( u_{\varepsilon }(\Div_{y}\Phi )^{\varepsilon
}+\varepsilon u_{\varepsilon }(\Div\Phi )^{\varepsilon }\right) dx.
\end{equation*}%
Letting $E^{\prime }\ni \varepsilon \rightarrow 0$ we get%
\begin{equation*}
-\iint_{\Omega \times \Delta (A)}\widehat{u}_{0}\widehat{\Div}\widehat{\Phi }%
dxd\beta =0.
\end{equation*}%
This shows that $\overline{\nabla }_{y}u_{0}=0$, which means that $%
u_{0}(x,\cdot )\in I_{A}^{p}$ (see (\ref{2.5})), that is, $u_{0}\in
L^{p}(\Omega ;I_{A}^{p})$. Next let $\Phi _{\varepsilon }(x)=\varphi (x)\Psi
(x/\varepsilon )$ ($x\in \Omega $) with $\varphi \in \mathcal{C}_{0}^{\infty
}(\Omega )$ and $\Psi =(\psi _{j})_{1\leq j\leq N}\in (A^{\infty })^{N}$
with ${\Div}_{y}\Psi =0$. Clearly%
\begin{equation*}
\sum_{j=1}^{N}\int_{\Omega }\frac{\partial u_{\varepsilon }}{\partial x_{j}}%
\varphi \psi _{j}^{\varepsilon }dx=-\sum_{j=1}^{N}\int_{\Omega
}u_{\varepsilon }\psi _{j}^{\varepsilon }\frac{\partial \varphi }{\partial
x_{j}}dx
\end{equation*}%
where $\psi _{j}^{\varepsilon }(x)=\psi _{j}(x/\varepsilon )$. Passing to
the limit in the above equation when $E^{\prime }\ni \varepsilon \rightarrow
0$ we get 
\begin{equation}
\sum_{j=1}^{N}\iint_{\Omega \times \Delta (A)}\widehat{v}_{j}\varphi 
\widehat{\psi }_{j}dxd\beta =-\sum_{j=1}^{N}\iint_{\Omega \times \Delta (A)}%
\widehat{u}_{0}\widehat{\psi }_{j}\frac{\partial \varphi }{\partial x_{j}}%
dxd\beta .  \label{2.3}
\end{equation}%
First, taking $\Phi =(\varphi \delta _{ij})_{1\leq i\leq N}$ with $\varphi
\in \mathcal{C}_{0}^{\infty }(\Omega )$ (for each fixed $1\leq j\leq N$) in (%
\ref{2.3}) we obtain 
\begin{equation}
\int_{\Omega }M(v_{j})\varphi dx=-\int_{\Omega }M(u_{0})\frac{\partial
\varphi }{\partial x_{j}}dx  \label{2.4}
\end{equation}%
and reminding that $M(v_{j})\in L^{p}(\Omega )$ we have by (\ref{2.4}) that $%
\frac{\partial u_{0}}{\partial x_{j}}\in L^{p}(\Omega ;I_{A}^{p})$, where $%
\frac{\partial u_{0}}{\partial x_{j}}$ is the distributional derivative of $%
u_{0}$ with respect to $x_{j}$. We deduce that $u_{0}\in W^{1,p}(\Omega
;I_{A}^{p})$. Coming back to (\ref{2.3}) we get 
\begin{equation*}
\iint_{\Omega \times \Delta (A)}\left( \widehat{\mathbf{v}}-\nabla \widehat{u%
}_{0}\right) \cdot \widehat{\Psi }\varphi dxd\beta =0\text{,}
\end{equation*}%
and so, as $\varphi $ is arbitrarily fixed, 
\begin{equation*}
\int_{\Delta (A)}\left( \widehat{\mathbf{v}}(x,s)-\nabla \widehat{u}%
_{0}(x,s)\right) \cdot \widehat{\Psi }(s)d\beta =0
\end{equation*}%
for all $\Psi $ as above and for a.e. $x$. Therefore we infer from Corollary %
\ref{c1} the existence of a function $u_{1}(x,\cdot )\in \mathcal{B}%
_{A}^{1,p}$ such that 
\begin{equation*}
\mathbf{v}(x,\cdot )-\nabla u_{0}(x,\cdot )=\overline{\nabla }%
_{y}u_{1}(x,\cdot )
\end{equation*}%
for a.e. $x$. From which the existence of a function $u_{1}:x\mapsto
u_{1}(x,\cdot )$ with values in $\mathcal{B}_{A}^{1,p}$ such that $\mathbf{v}%
=\nabla u_{0}+\overline{\nabla }_{y}u_{1}$.
\end{proof}

\begin{remark}
\label{r2.1}\emph{If we assume the algebra }$A$\emph{\ to be ergodic, then }$%
I_{A}^{p}$\emph{\ consists of constant functions, so that the function }$%
u_{0}$\emph{\ in Theorem \ref{t2.3} does not depend on }$y$\emph{, that is, }%
$u_{0}\in W^{1,p}(\Omega )$\emph{. We thus recover the already known result
proved in \cite{NA} in the case of ergodic algebras.}
\end{remark}

\section{On the interplay between $\Sigma $-convergence and convolution}

In order to take full advantage the results of Section 3, we assume
throughout this section that the algebra $A$ is introverted. Then its
spectrum is a compact topological semigroup.

With this in mind, let $a\in \mathbb{R}^{N}$ be fixed. Appealing to Theorem %
\ref{th5}, we may assume without lost of generality that $(\delta
_{a/\varepsilon })_{\varepsilon >0}$ is a net in the compact group $K(\Delta
(A))$, hence it possesses a weak$\ast $ cluster point $r\in K(\Delta (A))$.
In the sequel we shall consider a subnet still denoted by $(\delta
_{a/\varepsilon })_{\varepsilon >0}$ (if there is no danger of confusion)
that converges to $r$ in the relative topology of $\Delta (A)$, i.e. 
\begin{equation}
\delta _{\frac{a}{\varepsilon }}\rightarrow r\text{ in }K(\Delta (A))\text{%
-weak}\ast \text{ as }\varepsilon \rightarrow 0\text{.}  \label{2.9}
\end{equation}%
Finally let $\Omega $ be an open subset of $\mathbb{R}^{N}$, and let $%
(u_{\varepsilon })_{\varepsilon >0}$ be a sequence in $L^{p}(\Omega )$ ($%
1\leq p<\infty $) which is weakly $\Sigma $-convergent to $u_{0}\in
L^{p}(\Omega ;\mathcal{B}_{A}^{p})$. We will see that a micro-translation of
the sequence $u_{\varepsilon }$ induces a micro-translation on its limit,
while a macro-translation of $u_{\varepsilon }$ induces both micro- and
macro-translations on the limit $u_{0}$. This is the main goal of the
following result.

\begin{theorem}
\label{p2.4}Let $(u_{\varepsilon })_{\varepsilon >0}$ be a sequence in $%
L^{p}(\Omega )$ ($1\leq p<\infty $) and let the sequences $(v_{\varepsilon
})_{\varepsilon >0}$ and $(w_{\varepsilon })_{\varepsilon >0}$ be defined by 
\begin{equation*}
v_{\varepsilon }(x)=u_{\varepsilon }(x+a)\ (x\in \Omega -a)\text{\ and\ }%
w_{\varepsilon }(x)=u_{\varepsilon }(x+\varepsilon a)\text{\ \ }(x\in \Omega
).
\end{equation*}%
Finally, let $u_{0}\in L^{p}(\Omega ;\mathcal{B}_{A}^{p})$ and assume that $%
u_{\varepsilon }\rightarrow u_{0}$ in $L^{p}(\Omega )$-weak $\Sigma $ (resp.
-strong $\Sigma $).
\end{theorem}

\begin{proposition}
\begin{itemize}
\item[(i)] If $p>1$ then 
\begin{equation}
w_{\varepsilon }\rightarrow w_{0}\text{ in }L^{p}(\Omega )\text{-weak }%
\Sigma \text{ (resp. -strong }\Sigma \text{)}  \label{2.10}
\end{equation}%
where $w_{0}$ is defined by $w_{0}(x,y)=u_{0}(x,y+a)$\ for $(x,y)\in \Omega
\times \mathbb{\mathbb{R}}^{N}$. Moreover if $p=1$ then \emph{(\ref{2.10})}
is still valid provided that the sequence $(u_{\varepsilon })_{\varepsilon
>0}$ is uniformly integrable.

\item[(ii)] If \emph{(\ref{2.9})} holds true then 
\begin{equation}
v_{\varepsilon }\rightarrow v_{0}\text{ in }L^{p}(\Omega -a)\text{-weak }%
\Sigma \text{ (resp. -strong }\Sigma \text{)}  \label{2.11}
\end{equation}%
where $v_{0}\in L^{p}(\Omega -a,\mathcal{B}_{A}^{p})$ is defined by $%
\widehat{v}_{0}(x,s)=\widehat{u}_{0}(x+a,sr)$ for $(x,s)\in (\Omega
-a)\times \Delta (A)$.
\end{itemize}
\end{proposition}

\begin{proof}
We split the proof into two parts.

\textit{Part} 1. We consider first the case of weak $\Sigma $-convergence.
Let us first check (\ref{2.10}). Prior to this, let us note that the proof
of (\ref{2.10}) in the special case $p=1$ has been done in \cite{NA}. In the
general situation when $p>1$, the proof is very similar to the one of the
previous case, and for that reason, we just sketch it here. Let $\varphi \in 
\mathcal{C}_{0}^{\infty }(\Omega )$ and let $\psi \in A$; we have 
\begin{eqnarray*}
\int_{\Omega }w_{\varepsilon }\varphi \psi ^{\varepsilon }dx &=&\int_{\Omega
}u_{\varepsilon }(x+\varepsilon a)\varphi (x)\psi \left( \frac{x}{%
\varepsilon }\right) dx \\
&=&\int_{\Omega +\varepsilon a}u_{\varepsilon }(x)\varphi (x-\varepsilon
a)\psi \left( \frac{x}{\varepsilon }-a\right) dx \\
&=&\int_{\Omega }u_{\varepsilon }(x)\varphi (x-\varepsilon a)\psi \left( 
\frac{x}{\varepsilon }-a\right) dx \\
&&\ \ \ -\int_{\Omega \backslash (\Omega +\varepsilon a)}u_{\varepsilon
}(x)\varphi (x-\varepsilon a)\psi \left( \frac{x}{\varepsilon }-a\right) dx
\\
&&\;\;\;\ \ \ +\int_{(\Omega +\varepsilon a)\backslash \Omega
}u_{\varepsilon }(x)\varphi (x-\varepsilon a)\psi \left( \frac{x}{%
\varepsilon }-a\right) dx \\
&=&(I)-(II)+(III).
\end{eqnarray*}%
As in \cite{NA}, we have that 
\begin{eqnarray*}
(I) &\rightarrow &\iint_{\Omega \times K(\Delta (A))}\widehat{u}%
_{0}(x,s)\varphi (x)\widehat{\tau _{-a}\psi }(s)dxd\beta \\
&=&\iint_{\Omega \times K(\Delta (A))}\widehat{u}_{0}(x,s)\varphi (x)\tau
_{\delta _{-a}}\widehat{\psi }(s)dxd\beta \\
&=&\iint_{\Omega \times K(\Delta (A))}\tau _{\delta _{a}}\widehat{u}%
_{0}(x,s)\varphi (x)\widehat{\psi }(s)dxd\beta .
\end{eqnarray*}%
We infer from the inequality 
\begin{equation*}
\int_{(\Omega +\varepsilon a)\Delta \Omega }\left\vert u_{\varepsilon
}(x)\right\vert \left\vert \varphi (x-\varepsilon a)\right\vert \left\vert
\psi \left( \frac{x}{\varepsilon }-a\right) \right\vert dx\leq \left\Vert
\varphi \right\Vert _{\infty }\left\Vert \psi \right\Vert _{\infty }\left( 
\text{meas}[(\Omega +\varepsilon a)\Delta \Omega ]\right) ^{\frac{1}{%
p^{\prime }}}\left\Vert u_{\varepsilon }\right\Vert _{L^{p}(\Omega )}
\end{equation*}%
($(\Omega +\varepsilon a)\Delta \Omega $ being denoting the symmetric
difference between $\Omega +\varepsilon a$ and $\Omega $) that $(II)$ and $%
(III)$ tend to $0$ as $\varepsilon \rightarrow 0$. This proves (\ref{2.10}).

The proof of (\ref{2.11}) is more involved. It has just been done in \cite%
{SW} (in the case of weak $\Sigma $-convergence of course!) under the
restricted assumption that $\Omega $ is bounded and $\Delta (A)$ is a group.
We present here a general proof in which all these assumptions are relaxed.

Let $\varphi \in \mathcal{C}_{0}^{\infty }(\Omega -a)$ and $\psi \in A$. Let 
$(y_{n})_{n}$ be a net in $\mathbb{R}^{N}$ (independent of $\varepsilon $)
such that $\delta _{y_{n}}\in K(\Delta (A))$ and 
\begin{equation*}
\delta _{y_{n}}\rightarrow r\text{ in }K(\Delta (A))\text{-}weak\ast \text{
with }n\text{.}
\end{equation*}%
We have 
\begin{eqnarray*}
&&\int_{\Omega -a}u_{\varepsilon }(x+a)\varphi (x)\psi \left( \frac{x}{%
\varepsilon }\right) dx \\
&=&\int_{\Omega }u_{\varepsilon }(x)\varphi (x-a)\psi \left( \frac{x}{%
\varepsilon }-\frac{a}{\varepsilon }\right) dx \\
&=&\int_{\Omega }u_{\varepsilon }(x)\varphi (x-a)\left[ \psi \left( \frac{x}{%
\varepsilon }-\frac{a}{\varepsilon }\right) -\psi \left( \frac{x}{%
\varepsilon }-y_{n}\right) \right] dx \\
&&+\int_{\Omega }u_{\varepsilon }(x)\varphi (x-a)\psi \left( \frac{x}{%
\varepsilon }-y_{n}\right) dx \\
&=&(I)+(II)
\end{eqnarray*}%
where $(I)=\int_{\Omega }u_{\varepsilon }(x)\varphi (x-a)\left[ \psi \left( 
\frac{x}{\varepsilon }-\frac{a}{\varepsilon }\right) -\psi \left( \frac{x}{%
\varepsilon }-y_{n}\right) \right] dx$ and $(II)=\int_{\Omega
}u_{\varepsilon }(x)\varphi (x-a)\psi \left( \frac{x}{\varepsilon }%
-y_{n}\right) dx$. We first consider $(II)$. For $\varepsilon \rightarrow 0$%
, it holds that 
\begin{eqnarray*}
(II) &\rightarrow &\iint_{\Omega \times K(\Delta (A))}\widehat{u}%
_{0}(x,s)\varphi (x-a)\tau _{\delta _{-y_{n}}}\widehat{\psi }(s)dxd\beta (s)
\\
&=&\iint_{(\Omega -a)\times K(\Delta (A))}\tau _{\delta _{y_{n}}}\widehat{u}%
_{0}(x+a,s)\varphi (x)\widehat{\psi }(s)dxd\beta (s),
\end{eqnarray*}%
and 
\begin{eqnarray*}
&&\lim_{n}\iint_{(\Omega -a)\times K(\Delta (A))}\tau _{\delta _{y_{n}}}%
\widehat{u}_{0}(x+a,s)\varphi (x)\widehat{\psi }(s)dxd\beta (s) \\
&=&\iint_{(\Omega -a)\times K(\Delta (A))}\tau _{r}\widehat{u}%
_{0}(x+a,s)\varphi (x)\widehat{\psi }(s)dxd\beta (s) \\
&=&\iint_{(\Omega -a)\times K(\Delta (A))}\widehat{u}_{0}(x+a,sr)\varphi (x)%
\widehat{\psi }(s)dxd\beta (s).
\end{eqnarray*}%
As for $(I)$, we easily verify (using the fact that $\mathcal{G}$ is an
isometry) that 
\begin{equation*}
\left\vert (I)\right\vert \leq c\left\Vert \tau _{\delta _{-\frac{a}{%
\varepsilon }}}\widehat{\psi }-\tau _{\delta _{-y_{n}}}\widehat{\psi }%
\right\Vert _{\infty }.
\end{equation*}%
Since the mapping $s\mapsto s^{-1}$ is continuous in $K(\Delta (A))$, we get
that $\delta _{-\frac{a}{\varepsilon }}=\delta _{\frac{a}{\varepsilon }%
}^{-1}\rightarrow r^{-1}$ in $K(\Delta (A))$ weak$\ast $ and $\delta
_{-y_{n}}=\delta _{y_{n}}^{-1}\rightarrow r^{-1}$ in $K(\Delta (A))$ weak$%
\ast $. Invoking the uniform continuity of $\widehat{\psi }$, we are led to 
\begin{equation*}
\left\Vert \tau _{\delta _{-\frac{a}{\varepsilon }}}\widehat{\psi }-\tau
_{\delta _{-y_{n}}}\widehat{\psi }\right\Vert _{\infty }\rightarrow 0\text{
as }\varepsilon \rightarrow 0\text{ and next }n\rightarrow \infty \text{.}
\end{equation*}%
(\ref{2.11}) follows thereby.\medskip

\textit{Part} 2. We now assume strong $\Sigma $-convergence of $%
(u_{\varepsilon })_{\varepsilon }$. Then from Part 1, we have (\ref{2.10})
and (\ref{2.11}) in the weak sense. Moreover we deduce from both the
equality $\left\Vert u_{\varepsilon }\right\Vert _{L^{p}(\Omega
)}=\left\Vert u_{\varepsilon }(\cdot +a)\right\Vert _{L^{p}(\Omega -a)}$ and
the translation invariance of the measure $\beta $ that $\left\Vert
u_{\varepsilon }(\cdot +a)\right\Vert _{L^{p}(\Omega -a)}\rightarrow
\left\Vert \widehat{u}_{0}(\cdot +a,\cdot r)\right\Vert _{L^{p}((\Omega
-a)\times K(\Delta (A)))}$. This concludes the proof of (\ref{2.11}) in both
cases (weak and strong $\Sigma $-convergence). The same also holds for (\ref%
{2.10}) in the case of strong $\Sigma $-convergence. The proof is complete.
\end{proof}

The next important result deals with the convergence of convolution
sequences. Let $p,q,m\geq 1$ be real numbers such that $\frac{1}{p}+\frac{1}{%
q}=1+\frac{1}{m}$. Let $(u_{\varepsilon })_{\varepsilon >0}\subset
L^{p}(\Omega )$ and $(v_{\varepsilon })_{\varepsilon >0}\subset L^{q}(%
\mathbb{R}^{N})$ be two sequences. One may view $u_{\varepsilon }$ as
defined in the whole $\mathbb{R}^{N}$ by taking its extension by zero
outside $\Omega $. Define 
\begin{equation*}
(u_{\varepsilon }\ast v_{\varepsilon })(x)=\int_{\mathbb{R}%
^{N}}u_{\varepsilon }(t)v_{\varepsilon }(x-t)dt\ \ (x\in \mathbb{R}^{N}),
\end{equation*}%
which lies in $L^{m}(\mathbb{R}^{N})$ and satisfies the Young's inequality 
\begin{equation}
\left\Vert u_{\varepsilon }\ast v_{\varepsilon }\right\Vert _{L^{m}(\Omega
)}\leq \left\Vert u_{\varepsilon }\right\Vert _{L^{p}(\Omega )}\left\Vert
v_{\varepsilon }\right\Vert _{L^{q}(\Omega )}.  \label{11}
\end{equation}%
We have the following result.

\begin{theorem}
\label{t2.4}Let $(u_{\varepsilon })_{\varepsilon >0}$ and $(v_{\varepsilon
})_{\varepsilon >0}$ be as above. Assume that, as $\varepsilon \rightarrow 0$%
, $u_{\varepsilon }\rightarrow u_{0}$ in $L^{p}(\Omega )$-weak $\Sigma $ and 
$v_{\varepsilon }\rightarrow v_{0}$ in $L^{q}(\mathbb{R}^{N})$-strong $%
\Sigma $, where $u_{0}$ and $v_{0}$ are in $L^{p}(\Omega ;\mathcal{B}%
_{A}^{p})$ and $L^{q}(\mathbb{R}^{N};\mathcal{B}_{A}^{q})$ respectively.
Then, as $\varepsilon \rightarrow 0$, 
\begin{equation*}
u_{\varepsilon }\ast v_{\varepsilon }\rightarrow u_{0}\ast \ast v_{0}\text{
in }L^{m}(\Omega )\text{-weak }\Sigma \text{.}
\end{equation*}
\end{theorem}

\begin{proof}
In view of (\ref{11}), the sequence $(u_{\varepsilon }\ast v_{\varepsilon
})_{\varepsilon >0}$ is bounded in $L^{m}(\Omega )$. Now, let $\eta >0$ and
let $\psi _{0}\in \mathcal{K}(\mathbb{R}^{N};A)$ (the space of continuous
functions from $\mathbb{R}^{N}$ into $A$ with compact support in $\mathbb{R}%
^{N}$) be such that $\left\Vert \widehat{v}_{0}-\widehat{\psi }%
_{0}\right\Vert _{L^{q}(\mathbb{R}^{N}\times \Delta (A))}\leq \frac{\eta }{2}
$. Since $v_{\varepsilon }\rightarrow v_{0}$ in $L^{q}(\mathbb{R}^{N})$%
-strong $\Sigma $ we have that $v_{\varepsilon }-\psi _{0}^{\varepsilon
}\rightarrow v_{0}-\psi _{0}$ in $L^{q}(\mathbb{R}^{N})$-strong $\Sigma $,
hence $\left\Vert v_{\varepsilon }-\psi _{0}^{\varepsilon }\right\Vert
_{L^{q}(\mathbb{R}^{N})}\rightarrow \left\Vert \widehat{v}_{0}-\widehat{\psi 
}_{0}\right\Vert _{L^{q}(\mathbb{R}^{N}\times \Delta (A))}=\left\Vert 
\widehat{v}_{0}-\widehat{\psi }_{0}\right\Vert _{L^{q}(\mathbb{R}^{N}\times
K(\Delta (A)))}$ as $\varepsilon \rightarrow 0$. So, there is $\alpha >0$
such that 
\begin{equation}
\left\Vert v_{\varepsilon }-\psi _{0}^{\varepsilon }\right\Vert _{L^{q}(%
\mathbb{R}^{N})}\leq \eta \text{ for }0<\varepsilon \leq \alpha \text{.}
\label{2.6'}
\end{equation}%
For $f\in \mathcal{K}(\Omega ;A)$, we have (by still denoting by $%
u_{\varepsilon }$ the zero extension of $u_{\varepsilon }$ off $\Omega $)%
\begin{eqnarray*}
\int_{\Omega }(u_{\varepsilon }\ast v_{\varepsilon })(x)f\left( x,\frac{x}{%
\varepsilon }\right) dx &=&\int_{\Omega }\left( \int_{\mathbb{R}%
^{N}}u_{\varepsilon }(t)v_{\varepsilon }(x-t)dt\right) f\left( x,\frac{x}{%
\varepsilon }\right) dx \\
&=&\int_{\mathbb{R}^{N}}u_{\varepsilon }(t)\left[ \int_{\mathbb{R}%
^{N}}v_{\varepsilon }(x-t)f\left( x,\frac{x}{\varepsilon }\right) dx\right]
dt \\
&=&\int_{\mathbb{R}^{N}}u_{\varepsilon }(t)\left[ \int_{\mathbb{R}%
^{N}}v_{\varepsilon }(x)f\left( x+t,\frac{x}{\varepsilon }+\frac{t}{%
\varepsilon }\right) dx\right] dt \\
&=&\int_{\mathbb{R}^{N}}u_{\varepsilon }(t)\left[ \int_{\mathbb{R}%
^{N}}(v_{\varepsilon }(x)-\psi _{0}^{\varepsilon }(x))f^{\varepsilon }(x+t)dx%
\right] dt \\
&&+\int_{\mathbb{R}^{N}}u_{\varepsilon }(t)\left( \int_{\mathbb{R}^{N}}\psi
_{0}^{\varepsilon }(x)f^{\varepsilon }(x+t)dx\right) dt \\
&=&(I)+(II).
\end{eqnarray*}%
On the one hand one has $(I)=\int_{\Omega }[u_{\varepsilon }\ast
(v_{\varepsilon }-\psi _{0}^{\varepsilon })](x)f^{\varepsilon }(x)dx$ and 
\begin{eqnarray*}
\left\vert (I)\right\vert  &\leq &\left\Vert u_{\varepsilon }\right\Vert
_{L^{p}(\Omega )}\left\Vert v_{\varepsilon }-\psi _{0}^{\varepsilon
}\right\Vert _{L^{q}(\mathbb{R}^{N})}\left\Vert f^{\varepsilon }\right\Vert
_{L^{m^{\prime }}(\Omega )} \\
&\leq &c\left\Vert v_{\varepsilon }-\psi _{0}^{\varepsilon }\right\Vert
_{L^{q}(\mathbb{R}^{N})}
\end{eqnarray*}%
where $c$ is a positive constant independent of $\varepsilon $. It follows
from (\ref{2.6'}) that 
\begin{equation}
\left\vert (I)\right\vert \leq c\eta \text{ for }0<\varepsilon \leq \alpha 
\text{.}  \label{2.7'}
\end{equation}%
On the other hand, in view of Theorem \ref{p2.4}, we have, as $\varepsilon
\rightarrow 0$, 
\begin{equation*}
\int_{\mathbb{R}^{N}}\psi _{0}^{\varepsilon }(x)f^{\varepsilon
}(x+t)dx\rightarrow \iint_{\mathbb{R}^{N}\times K(\Delta (A))}\widehat{\psi }%
_{0}(x,s)\widehat{f}(x+t,sr)dxd\beta (s)
\end{equation*}%
where $r=\lim \delta _{t/\varepsilon }$ (for a suitable subsequence of $%
\varepsilon \rightarrow 0$) in $K(\Delta (A))$-weak$\ast $. So let $\Phi :%
\mathbb{R}^{N}\times \Delta (A)\rightarrow \mathbb{R}$ be defined by 
\begin{equation*}
\Phi (t,r)=\iint_{\mathbb{R}^{N}\times K(\Delta (A))}\widehat{\psi }_{0}(x,s)%
\widehat{f}(x+t,sr)dxd\beta (s),\ (t,r)\in \mathbb{R}^{N}\times \Delta (A).
\end{equation*}%
Then we easily check that $\Phi \in \mathcal{K}(\mathbb{R}^{N};\mathcal{C}%
(\Delta (A)))$, so that there is a function $\Psi \in \mathcal{K}(\mathbb{R}%
^{N};A)$ with $\Phi =\mathcal{G}\circ \Psi $. We can therefore define the
trace $\Psi ^{\varepsilon }(t)=\Psi (t,t/\varepsilon )$ ($t\in \mathbb{R}^{N}
$) and we have 
\begin{eqnarray*}
\Psi ^{\varepsilon }(t) &=&\left\langle \delta _{\frac{t}{\varepsilon }%
},\Psi (t,\cdot )\right\rangle =\widehat{\Psi }\left( t,\delta _{\frac{t}{%
\varepsilon }}\right) =\Phi \left( t,\delta _{\frac{t}{\varepsilon }}\right) 
\\
&=&\iint_{\mathbb{R}^{N}\times K(\Delta (A))}\widehat{\psi }_{0}(x,s)%
\widehat{f}(x,s\delta _{\frac{t}{\varepsilon }})dxd\beta (s).
\end{eqnarray*}

Next, we have 
\begin{eqnarray*}
(II) &=&\int_{\mathbb{R}^{N}}u_{\varepsilon }(t)\left( \int_{\mathbb{R}%
^{N}}\psi _{0}^{\varepsilon }(x)f^{\varepsilon }(x+t)dx-\Psi ^{\varepsilon
}(t)\right) dt+\int_{\mathbb{R}^{N}}u_{\varepsilon }(t)\Psi ^{\varepsilon
}(t)dt \\
&=&(II_{1})+(II_{2}).
\end{eqnarray*}%
As far as $(II_{1})$ is concerned, set 
\begin{equation*}
V_{\varepsilon }(t)=\int_{\mathbb{R}^{N}}\psi _{0}^{\varepsilon
}(x)f^{\varepsilon }(x+t)dx-\Psi ^{\varepsilon }(t)\text{ for a.e. }t\in 
\mathbb{R}^{N}.
\end{equation*}%
Then the following claims hold.

\begin{claim}
\label{cl1}For a.e. $t$, $V_{\varepsilon }(t)\rightarrow 0$ as $\varepsilon
\rightarrow 0$ (possibly up to a subsequence)
\end{claim}

\begin{claim}
\label{cl2}$\int_{\mathbb{R}^{N}}u_{\varepsilon }(t)V_{\varepsilon
}(t)dt\rightarrow 0$ as $\varepsilon \rightarrow 0$.
\end{claim}

Indeed, for Claim \ref{cl1}, applying Theorem \ref{p2.4} leads one to 
\begin{equation*}
\int_{\mathbb{R}^{N}}\psi _{0}^{\varepsilon }(x)f^{\varepsilon
}(x+t)dx\rightarrow \iint_{\mathbb{R}^{N}\times K(\Delta (A))}\widehat{\psi }%
_{0}(x,s)\widehat{f}(x+t,sr)dxd\beta (s)\text{ as }\varepsilon \rightarrow 0
\end{equation*}%
where $r$ is such that $\delta _{t/\varepsilon }\rightarrow r$ in $K(\Delta
(A))$ weak$\ast $ for some subsequence of $\varepsilon $. Moreover, since $%
\Psi ^{\varepsilon }(t)=\Phi (t,\delta _{t/\varepsilon })$, we have by the
continuity of $\Phi (t,\cdot )$ that, for the same subsequence, 
\begin{equation*}
\Psi ^{\varepsilon }(t)\rightarrow \iint_{\mathbb{R}^{N}\times K(\Delta (A))}%
\widehat{\psi }_{0}(x,s)\widehat{f}(x+t,sr)dxd\beta (s).
\end{equation*}%
Whence Claim \ref{cl1} is justified. As for Claim \ref{cl2}, first and
foremost we have 
\begin{equation*}
\left\vert V_{\varepsilon }(t)\right\vert \leq c\text{ for a.e. }t\in \Omega 
\end{equation*}%
where $c$ is a positive constant independent of $t$ and $\varepsilon $.
Since $f$ and $\psi _{0}$ belong to $\mathcal{K}(\mathbb{R}^{N};A)$ we have
that $f^{\varepsilon }$ and $\psi _{0}^{\varepsilon }$ lie in $\mathcal{K}(%
\mathbb{R}^{N})$ and their support are contained in a fixed compact set of $%
\mathbb{R}^{N}$. Therefore $\psi _{0}^{\varepsilon }\ast f^{\varepsilon }\in 
\mathcal{K}(\mathbb{R}^{N})$. As a result, $V_{\varepsilon }\in \mathcal{K}(%
\mathbb{R}^{N})$ and further its support is contained in a fixed compact set 
$K\subset \mathbb{R}^{N}$ independent of $\varepsilon $.

This being so, let $\gamma >0$. From Egorov's theorem there exists $D\subset 
\mathbb{R}^{N}$ such that meas$(\mathbb{R}^{N}\backslash D)<\gamma $ and $%
V_{\varepsilon }$ converges uniformly to $0$ on $D$. We have the following
series of inequalities 
\begin{eqnarray*}
\left\vert \int_{\mathbb{R}^{N}}u_{\varepsilon }(t)V_{\varepsilon
}(t)dt\right\vert &\leq &\left\Vert u_{\varepsilon }\right\Vert
_{L^{p}(D)}\left\Vert V_{\varepsilon }\right\Vert _{L^{p^{\prime
}}(D)}+\left\Vert u_{\varepsilon }\right\Vert _{L^{p}(\mathbb{R}%
^{N}\backslash D)}\left\Vert V_{\varepsilon }\right\Vert _{L^{p^{\prime }}(%
\mathbb{R}^{N}\backslash D)} \\
&\leq &C\left\Vert V_{\varepsilon }\right\Vert _{L^{p^{\prime }}(D\cap K)}+C%
\text{meas}(\mathbb{R}^{N}\backslash D) \\
&\leq &C_{1}\text{meas}(K)\sup_{t\in D}\left\vert V_{\varepsilon
}(t)\right\vert +C_{1}\gamma
\end{eqnarray*}%
where $C_{1}$ is a positive constant independent of both $\varepsilon $ and $%
D$. It emerges from the uniform continuity of $V_{\varepsilon }$ in $D$ that
there exists $\varepsilon _{0}>0$ with $\alpha _{1}\leq \alpha $ such that 
\begin{equation*}
\left\vert \int_{\mathbb{R}^{N}}u_{\varepsilon }(t)V_{\varepsilon
}(t)dt\right\vert \leq C_{2}\gamma \text{ provided }0<\varepsilon \leq
\alpha _{1},
\end{equation*}%
where $C_{2}>0$ is independent of $\varepsilon $. This shows Claim \ref{cl2}%
. Thus $(II_{1})\rightarrow 0$ as $\varepsilon \rightarrow 0$.

As for $(II_{2})$, using once again the weak $\Sigma $-convergence of $%
(u_{\varepsilon })_{\varepsilon }$, we get 
\begin{equation*}
\int_{\Omega }u_{\varepsilon }(t)\Psi ^{\varepsilon }(t)dt\rightarrow
\iint_{\Omega \times K\Delta (A))}\widehat{u}_{0}(t,r)\widehat{\Psi }%
(t,r)dtd\beta ,
\end{equation*}%
and 
\begin{eqnarray*}
&&\iint_{\Omega \times K(\Delta (A))}\widehat{u}_{0}(t,r)\widehat{\Psi }%
(t,r)dtd\beta  \\
&=&\iint_{\Omega \times K(\Delta (A))}\widehat{u}_{0}(t,r)\Phi (t,r)dtd\beta 
\\
&=&\iint_{\Omega \times K(\Delta (A))}\widehat{u}_{0}(t,r)\left[ \iint_{%
\mathbb{R}^{N}\times K(\Delta (A))}\widehat{\psi }_{0}(x,s)\widehat{f}%
(x+t,sr)dxd\beta (s)\right] dtd\beta (r) \\
&=&\iint_{\Omega \times K(\Delta (A))}\left[ \iint_{\mathbb{R}^{N}\times
K(\Delta (A))}\widehat{u}_{0}(t,r)\widehat{\psi }_{0}(x-t,sr^{-1})dtd\beta
(r)\right] \widehat{f}(x,s)dxd\beta (s) \\
&=&\iint_{\Omega \times K(\Delta (A))}(\widehat{u}_{0}\widehat{\ast \ast }%
\widehat{\psi }_{0})(x,s)\widehat{f}(x,s)dxd\beta (s).
\end{eqnarray*}%
Thus, there is $0<\alpha _{2}\leq \alpha _{1}$ such that 
\begin{equation}
\left\vert \int_{\Omega }(u_{\varepsilon }\ast \psi _{0}^{\varepsilon
})f^{\varepsilon }dx-\iint_{\Omega \times K(\Delta (A))}(\widehat{u}_{0}%
\widehat{\ast \ast }\widehat{\psi }_{0})\widehat{f}dxd\beta \right\vert \leq 
\frac{\eta }{2}\text{ for }0<\varepsilon \leq \alpha _{2}\text{.}
\label{2.8}
\end{equation}%
Now, let $0<\varepsilon \leq \alpha _{2}$ be fixed. Finally the
decomposition 
\begin{eqnarray*}
&&\int_{\Omega }(u_{\varepsilon }\ast v_{\varepsilon })f^{\varepsilon
}dx-\iint_{\Omega \times K(\Delta (A))}(\widehat{u}_{0}\widehat{\ast \ast }%
\widehat{v}_{0})\widehat{f}dxd\beta  \\
&=&\int_{\Omega }\left[ u_{\varepsilon }\ast (v_{\varepsilon }-\psi
_{0}^{\varepsilon })\right] f^{\varepsilon }dx+\iint_{\Omega \times K(\Delta
(A))}\left[ \widehat{u}_{0}\widehat{\ast \ast }(\widehat{\psi }_{0}-\widehat{%
v}_{0})\right] \widehat{f}dxd\beta  \\
&&+\int_{\Omega }(u_{\varepsilon }\ast \psi _{0}^{\varepsilon
})f^{\varepsilon }dx-\iint_{\Omega \times K(\Delta (A))}(\widehat{u}_{0}%
\widehat{\ast \ast }\widehat{\psi }_{0})\widehat{f}dxd\beta ,
\end{eqnarray*}%
associated to (\ref{2.6'})-(\ref{2.8}) allow one to see that 
\begin{equation*}
\left\vert \int_{\Omega }(u_{\varepsilon }\ast v_{\varepsilon
})f^{\varepsilon }dx-\iint_{\Omega \times K(\Delta (A))}(\widehat{u}_{0}%
\widehat{\ast \ast }\widehat{v}_{0})\widehat{f}dxd\beta \right\vert \leq
C\eta \text{ for }0<\varepsilon \leq \alpha _{2}.
\end{equation*}%
Here $C$ is a positive constant independent of $\varepsilon $. This
concludes the proof.
\end{proof}

\begin{remark}
\label{r2.2}\emph{Theorem \ref{t2.4} generalizes its homologue Theorem 2 in 
\cite{SW} which is concerned with the special case when }$q=1$\emph{\ and }$%
\Omega $\emph{\ bounded. Though the above proof is very similar to the one
in \cite{SW}, it is different from the latter in the sense that it takes
full advantage of Egorov's theorem.}
\end{remark}

In practise, we deal in this work with the evolutionary version of the
concept of $\Sigma $-convergence. This requires some further notion such as
the one related to the product of algebras with mean value. Let $A_{1}$
(resp. $A_{2}$) be an algebra with mean value on $\mathbb{R}^{m_{1}}$ (resp. 
$\mathbb{R}^{m_{2}}$). We define their product denoted by $A_{1}\odot A_{2}$
as the closure in BUC$(\mathbb{R}^{m_{1}}\times \mathbb{R}^{m_{2}})$ of the
tensor product $A_{1}\otimes A_{2}=\{\sum_{\text{finite}}u_{i_{1}}\otimes
u_{i_{2}}:u_{i_{j}}\in A_{j},\;j=1,2\}$. It is a well known fact that $%
A_{1}\odot A_{2}$ is an algebra with mean value on $\mathbb{R}^{m_{1}}\times 
\mathbb{R}^{m_{2}}$; see e.g. \cite{Hom1, CMP}.

With this in mind, let $A=A_{y}\odot A_{\tau }$ where $A_{y}$ (resp. $%
A_{\tau }$) is an algebra with mean value on $\mathbb{R}_{y}^{N}$ (resp. $%
\mathbb{R}_{\tau }$). The same letter $\mathcal{G}$ will denote the Gelfand
transformation on $A_{y}$, $A_{\tau }$ and $A$, as well. Points in $\Delta
(A_{y})$ (resp. $\Delta (A_{\tau })$) are denoted by $s$ (resp. $s_{0}$).
The compact space $\Delta (A_{y})$ (resp. $\Delta (A_{\tau })$) is equipped
with the $M$-measure $\beta _{y}$ (resp. $\beta _{\tau }$), for $A_{y}$
(resp. $A_{\tau }$). We have $\Delta (A)=\Delta (A_{y})\times \Delta
(A_{\tau })$ (Cartesian product) and the $M$-measure for $A$, with which $%
\Delta (A)$ is equipped, is precisely the product measure $\beta =\beta
_{y}\otimes \beta _{\tau }$ (see \cite{Hom1}). Finally, let $0<T<\infty $.
We set $Q_{T}=\Omega \times \left( 0,T\right) $, an open cylinder in $%
\mathbb{R}^{N+1}$.

This being so, a sequence $\left( u_{\varepsilon }\right) _{\varepsilon
>0}\subset L^{p}\left( Q_{T}\right) $\ $(1\leq p<\infty )$\ is said to
weakly $\Sigma $-converge\ in $L^{p}\left( Q_{T}\right) $\ to some $u_{0}\in
L^{p}(Q_{T};\mathcal{B}_{A}^{p})$\ if as $\varepsilon \rightarrow 0$, 
\begin{equation*}
\int_{Q_{T}}u_{\varepsilon }\left( x,t\right) f\left( x,t,\frac{x}{%
\varepsilon },\frac{t}{\varepsilon }\right) dxdt\rightarrow
\iint_{Q_{T}\times K(\Delta (A))}\widehat{u}_{0}\left( x,t,s,s_{0}\right) 
\widehat{f}\left( x,t,s,s_{0}\right) dxdtd\beta
\end{equation*}%
for all $f\in L^{p^{\prime }}\left( Q_{T};A\right) $.

\begin{remark}
\label{r2.4}\emph{The conclusions of Theorems \ref{t2.2} and \ref{t2.3} are
still valid mutatis mutandis in the present context (change there }$\Omega $%
\emph{\ into }$Q_{T}$\emph{, }$W^{1,p}(\Omega )$\emph{\ into }$%
L^{p}(0,T;W^{1,p}(\Omega ))$\emph{, }$W^{1,p}(\Omega ;I_{A}^{p})\times
L^{p}(\Omega ;\mathcal{B}_{A}^{1,p})$\emph{\ into }$L^{p}(0,T;W^{1,p}(\Omega
;I_{A}^{p}))\times L^{p}(Q_{T};\mathcal{B}_{A_{\tau }}^{p}(\mathbb{R}_{\tau
};\mathcal{B}_{A_{y}}^{1,p}))$\emph{).}
\end{remark}

Now, assume that $A_{y}$ and $A_{\tau }$ are introverted. Let $(a,\tau )\in 
\mathbb{R}^{N}\times \mathbb{R}$, and let $\left( u_{\varepsilon }\right)
_{\varepsilon >0}\subset L^{p}\left( Q_{T}\right) $\ $(1\leq p<\infty )$\ be
a weakly $\Sigma $-convergent subsequence in $L^{p}\left( Q_{T}\right) $ to $%
u_{0}\in L^{p}(Q_{T};\mathcal{B}_{A}^{p})$. Set 
\begin{equation*}
v_{\varepsilon }(x,t)=u_{\varepsilon }(x+a,t+\tau )\text{ for }(x,t)\in
Q_{T}-(a,\tau )\equiv (\Omega -a)\times (-\tau ,T-\tau ).
\end{equation*}%
Then $v_{\varepsilon }\rightarrow v_{0}$ in $L^{p}(Q_{T}-(a,\tau ))$-weak $%
\Sigma $ where $v_{0}\in L^{p}(Q_{T}-(a,\tau );\mathcal{B}_{A}^{p})$ is
defined by 
\begin{equation*}
\widehat{v}_{0}(x,t,s,s_{0})=\widehat{u}_{0}(x+a,t+\tau ,sr,s_{0}r_{0})\text{%
, \ }(x,t,s,s_{0})\in \lbrack Q_{T}-(a,\tau )]\times \Delta (A),
\end{equation*}%
the micro-translations $r$ and $r_{0}$ being determined as follows: $\delta
_{a/\varepsilon }\rightarrow r$ in $K(\Delta (A_{y}))$ and $\delta _{\tau
/\varepsilon }\rightarrow r_{0}$ in $K(\Delta (A_{\tau }))$ up to a
subsequence of $\varepsilon $. A similar conclusion holds in Theorem \ref%
{t2.4} mutatis mutandis.

\section{Homogenization of a parameterized Wilson-Cowan type equation with
finite and infinite delays}

We consider the parameterized Wilson-Cowan model with delay \cite{WC72, WC73}
\begin{equation}
\left\{ 
\begin{array}{l}
\frac{\partial u_{\varepsilon }}{\partial t}(x,t)=-u_{\varepsilon
}(x+a,t)+\int_{\mathbb{R}^{N}}K^{\varepsilon }(x-\xi )f\left( \frac{\xi }{%
\varepsilon },u_{\varepsilon }(\xi ,t)\right) d\xi \ \text{in }\mathbb{R}%
_{T}^{N}=\mathbb{R}^{N}\times \left( 0,T\right) \\ 
u_{\varepsilon }(x,0)=u^{0}(x),\ x\in \mathbb{R}^{N}%
\end{array}%
\right.  \label{3.1'}
\end{equation}%
where $a\in \mathbb{R}^{N}$ is fixed, $u_{\varepsilon }$ denotes the
electrical activity level field, $f$ the firing rate function and $%
K^{\varepsilon }=K^{\varepsilon }(x)=K(x,x/\varepsilon )$ the connectivity
kernel. We assume that $K\in \mathcal{K}(\mathbb{R}^{N};A)$ ($A$ an
introverted algebra with mean value on $\mathbb{R}^{N}$) is nonnegative and
is such that $\int_{\mathbb{R}^{N}}K^{\varepsilon }(x)dx\leq 1$, $f:\mathbb{R%
}^{N}\times \mathbb{R}\rightarrow \mathbb{R}$ is a nonnegative Carath\'{e}%
odory function satisfying the following conditions:

\begin{itemize}
\item[(H1)] For almost all $y\in \mathbb{R}^{N}$, the function $f(y,\cdot
):\lambda \mapsto f(y,\lambda )$ is continuous; for all $\lambda \in \mathbb{%
R}$, the function $f(\cdot ,\lambda ):y\mapsto f(y,\lambda )$ is measurable
and $f(\cdot ,0)$ lies in $L^{1}(\mathbb{R}^{N})\cap L^{2}(\mathbb{R}^{N})$;
there exists a positive constant $k_{1}$ such that 
\begin{equation*}
\left\vert f(y,\mu _{1})-f(y,\mu _{2})\right\vert \leq k_{1}\left\vert \mu
_{1}-\mu _{2}\right\vert \text{ for all }y\in \mathbb{R}^{N}\text{ and all }%
\mu _{1},\mu _{2}\in \mathbb{R}.
\end{equation*}

\item[(H2)] $f(\cdot ,\mu )\in A$ for all $\mu \in \mathbb{R}$.

\item[(H3)] For any sequence $(v_{\varepsilon })_{\varepsilon >0}\subset
L^{1}(\mathbb{R}_{T}^{N})$ such that $v_{\varepsilon }\rightarrow v_{0}$ in $%
L^{1}(\mathbb{R}_{T}^{N})$-weak $\Sigma $, we have $f^{\varepsilon }(\cdot
,v_{\varepsilon })\rightarrow f(\cdot ,v_{0})$ in $L^{1}(\mathbb{R}_{T}^{N})$%
-weak $\Sigma $.
\end{itemize}

Assumption (H1) is used to derive the existence and uniqueness result while
assumptions (H2) and (H3) are the cornerstone in the homogenization process.
It is worth noticing that assumption (H3) is meaningful. It concerns the
convergence of fluxes to flux. Indeed the convergence result $v_{\varepsilon
}\rightarrow v_{0}$ in $L^{1}(\mathbb{R}_{T}^{N})$-weak $\Sigma $ does not
in general ensure the convergence result $f^{\varepsilon }(\cdot
,v_{\varepsilon })\rightarrow f(\cdot ,v_{0})$ in $L^{1}(\mathbb{R}_{T}^{N})$%
-weak $\Sigma $. However, under some circumstances, this becomes true; see
e.g. \cite[Section 4]{SW} for some concrete examples.

In \cite{SW}, Eq. (\ref{3.1'}) with $a=0$ has been considered. Here we aim
at showing how the translation induces a memory effect at the microscopic
level. Before we can do that, however, we need to provide an existence
result. But by repeating the proof of Theorem 3 in \cite{SW}, we get the
following result.

\begin{theorem}
\label{t4.1'}Assume that $u^{0}\in L^{1}(\mathbb{R}^{N})\cap L^{2}(\mathbb{R}%
^{N})$. For each $\varepsilon >0$, there exists a unique solution $%
u_{\varepsilon }\in \mathcal{C}([0,\infty );L^{1}(\mathbb{R}^{N})\cap L^{2}(%
\mathbb{R}^{N}))$ to \emph{(\ref{3.1'})} satisfying 
\begin{equation}
\sup_{\varepsilon >0}\sup_{0\leq t\leq T}\left[ \left\Vert u_{\varepsilon
}(\cdot ,t)\right\Vert _{L^{1}(\mathbb{R}^{N})}+\left\Vert u_{\varepsilon
}(\cdot ,t)\right\Vert _{L^{2}(\mathbb{R}^{N})}\right] \leq C  \label{3.2'}
\end{equation}%
where $C$ is a positive constant depending only on $u^{0}$ and on $T$.
Moreover the sequence $(u_{\varepsilon })_{0<\varepsilon \leq 1}$\ is
uniformly integrable in $L^{1}(\mathbb{R}_{T}^{N})$.
\end{theorem}

In this section, we use the following version of $\Sigma $-convergence. A
sequence $(u_{\varepsilon })_{\varepsilon >0}\subset L^{p}(\mathbb{R}%
_{T}^{N})$ is weakly $\Sigma $-convergent in $L^{p}(\mathbb{R}_{T}^{N})$\ if 
\begin{equation*}
\int_{\mathbb{R}_{T}^{N}}u_{\varepsilon }\left( x,t\right) f\left( x,t,\frac{%
x}{\varepsilon }\right) dxdt\rightarrow \iint_{\mathbb{R}_{T}^{N}\times
K(\Delta (A))}\widehat{u}_{0}\left( x,t,s\right) \widehat{f}\left(
x,t,s\right) dxdtd\beta \left( s\right)
\end{equation*}%
for all $f\in L^{p^{\prime }}\left( \mathbb{R}_{T}^{N};A\right) $\ where $%
\widehat{f}\left( x,t,\cdot \right) =\mathcal{G}(f\left( x,t,\cdot \right) )$%
\ a.e. in $(x,t)\in \mathbb{R}_{T}^{N}$. With this in mind, we can now state
and prove the homogenization result.

\begin{theorem}
\label{t4.2'}Let $(u_{\varepsilon })_{\varepsilon >0}$ be the sequence of
solutions to \emph{(\ref{3.1'})}. Then there exists a subsequence of $%
\varepsilon $ still denoted by $\varepsilon $ such that, when $\varepsilon
\rightarrow 0$, it holds that 
\begin{equation}
\delta _{\frac{a}{\varepsilon }}\rightarrow r\text{ in }K(\Delta (A))\text{ }%
weak\ast \ \ \ \ \ \ \ \ \   \label{4.0}
\end{equation}%
and 
\begin{equation}
u_{\varepsilon }\rightarrow \mathcal{G}_{1}^{-1}\circ u_{0}\text{ in }L^{1}(%
\mathbb{R}_{T}^{N})\text{-weak }\Sigma  \label{4.1'}
\end{equation}%
where $u_{0}\in \mathcal{C}([0,T];L^{1}(\mathbb{R}^{N}\times \Delta (A)))$
is the unique solution to the following equation 
\begin{equation}
\left\{ 
\begin{array}{l}
\frac{\partial u_{0}}{\partial t}(x,t,s)=-u_{0}(x+a,t,sr)+(\widehat{K}\ast
\ast \widehat{f}(\cdot ,u_{0}))(x,t,s),\ (x,t,s)\in \mathbb{R}_{T}^{N}\times
\Delta (A) \\ 
u_{0}(x,0,s)=u^{0}(x),\ (x,s)\in \mathbb{R}^{N}\times \Delta (A).%
\end{array}%
\right.  \label{4.2'}
\end{equation}
\end{theorem}

\begin{proof}
Let $E$ be an ordinary sequence of positive real numbers $\varepsilon $. We
know from Theorem \ref{t4.1'} that the sequence $(u_{\varepsilon
})_{\varepsilon \in E}$ is uniformly integrable in $L^{1}(\mathbb{R}_{T}^{N})
$. As a result of [part (ii) of] Theorem \ref{t2.2}, there exist a
subsequence $E^{\prime }$ of $E$ and a function $v_{0}\in L^{1}(\mathbb{R}%
_{T}^{N};\mathcal{B}_{A}^{1})$ such that, as $E^{\prime }\ni \varepsilon
\rightarrow 0$, we have (\ref{4.1'}) with $u_{0}=\mathcal{G}_{1}\circ v_{0}$%
. On the other hand, in view of Theorem \ref{th5}, $(\delta _{a/\varepsilon
})_{\varepsilon \in E^{\prime }}$ is a sequence in the compact space $%
K(\Delta (A))$, so that there exists a subsequence of $E^{\prime }$ not
relabeled and $r\in K(\Delta (A))$ such that (\ref{4.0}) holds true. The
theorem will be proved once we will check that $u_{0}\in \mathcal{C}%
([0,T];L^{1}(\mathbb{R}^{N}\times \Delta (A)))$ and satisfies (\ref{4.2'}).
In order to do that, first recall that $K\in \mathcal{K}(\mathbb{R}%
^{N}\times (0,T);A)$ since $\mathcal{K}(\mathbb{R}^{N};A)\subset \mathcal{K}(%
\mathbb{R}^{N}\times (0,T);A)$, so that we have 
\begin{equation*}
K^{\varepsilon }\rightarrow K\text{ in }L^{1}(\mathbb{R}_{T}^{N})\text{%
-strong }\Sigma \text{ as }E^{\prime }\ni \varepsilon \rightarrow 0\text{.}
\end{equation*}%
On the other hand, using the assumption (H3) in conjunction with (\ref{4.1'}%
), we are led to 
\begin{equation*}
f^{\varepsilon }(\cdot ,u_{\varepsilon })\rightarrow f(\cdot ,v_{0})\text{
in }L^{1}(\mathbb{R}_{T}^{N})\text{-weak }\Sigma \text{ as }E^{\prime }\ni
\varepsilon \rightarrow 0.
\end{equation*}%
We infer from Theorem \ref{t2.4} that 
\begin{equation*}
K^{\varepsilon }\ast f^{\varepsilon }(\cdot ,u_{\varepsilon })\rightarrow
K\ast \ast f(\cdot ,v_{0})\text{ in }L^{1}(\mathbb{R}_{T}^{N})\text{-weak }%
\Sigma \text{ when }E^{\prime }\ni \varepsilon \rightarrow 0.
\end{equation*}%
It also holds that 
\begin{equation*}
u_{\varepsilon }(\cdot +a,\cdot )\rightarrow w_{0}\text{ in }L^{1}(\mathbb{R}%
_{T}^{N})\text{-when }\Sigma \text{ when }E^{\prime }\ni \varepsilon
\rightarrow 0
\end{equation*}%
where $\widehat{w}_{0}(x,t,s)=u_{0}(x+a,t,sr)$ ($(x,t,s)\in \mathbb{R}%
_{T}^{N}\times \Delta (A)$) with $r$ determined by (\ref{4.0}); this is a
consequence of [part (ii) of] Theorem \ref{p2.4}. Finally, noting that Eq. (%
\ref{3.1'}) is equivalent to the following integral equation 
\begin{equation*}
u_{\varepsilon }(x,t)=u^{0}(x)+\int_{0}^{t}\left[ (K^{\varepsilon }\ast
f^{\varepsilon }(\cdot ,u_{\varepsilon }))(x,\tau )-u_{\varepsilon
}(x+a,\tau )\right] d\tau ,
\end{equation*}%
we are led (after passing to the limit when $E^{\prime }\ni \varepsilon
\rightarrow 0$ and using Fubini and Lebesgue dominated convergence results
in the integral term) to 
\begin{equation*}
v_{0}(x,t,y)=u^{0}(x)+\int_{0}^{t}\left[ (K\ast \ast f(\cdot ,v_{0}))(x,\tau
,y)-w_{0}(x,\tau ,y)\right] d\tau .
\end{equation*}%
Therefore, composing both members of the above equality by $\mathcal{G}_{1}$%
, we end up with 
\begin{equation*}
u_{0}(x,t,s)=u^{0}(x)+\int_{0}^{t}\left[ (\widehat{K}\ast \ast \widehat{f}%
(\cdot ,u_{0}))(x,\tau ,s)-u_{0}(x+a,\tau ,sr)\right] d\tau 
\end{equation*}%
which is nothing else but the integral form of (\ref{4.2'}). We also
conclude from the preceding equation that $u_{0}$ lies in $\mathcal{C}%
([0,T];L^{1}(\mathbb{R}^{N}\times \Delta (A)))$ as expected. Whence the
proof.
\end{proof}

\section{Homogenization of a nonlocal nonlinear heat equation}

We consider the following non local boundary value problem 
\begin{equation}
\begin{array}{l}
\rho ^{\varepsilon }\frac{\partial u_{\varepsilon }}{\partial t}-\Div %
a^{\varepsilon }\left( \cdot ,\cdot ,\nabla u_{\varepsilon }\right)
+K^{\varepsilon }\ast a_{0}^{\varepsilon }\left( \cdot ,\cdot ,\nabla
u_{\varepsilon }\right) =f\text{ in }Q_{T} \\ 
\ \ \ \ \ \ \ \ \ \ \ \ \ \ \ \ \ \ \ \ \ \ \ \ \ \ \ \ \ \ \ \ \ \ \ \ \ \
\ \ \ u_{\varepsilon }=0\text{ on }\partial \Omega \times \left( 0,T\right)
\\ 
\ \ \ \ \ \ \ \ \ \ \ \ \ \ \ \ \ \ \ \ \ \ \ \ \ \ \ \ \ \ \ \ \ \ \ \ \ \
\ \ \ u_{\varepsilon }(x,0)=u^{0}(x)\text{ in }\Omega%
\end{array}
\label{4.1}
\end{equation}%
where $\Omega $ is a bounded smooth open set in $\mathbb{R}^{N}$, $\rho
^{\varepsilon }(x)=\rho (x/\varepsilon )$, $a^{\varepsilon }\left( \cdot
,\cdot ,\nabla u_{\varepsilon }\right) (x,t)=a\left( x/\varepsilon
,t/\varepsilon ,\nabla u_{\varepsilon }(x,t)\right) $ (same definition for $%
a_{0}^{\varepsilon }\left( \cdot ,\cdot ,\nabla u_{\varepsilon }\right) $), $%
K^{\varepsilon }(x,t)=K\left( x/\varepsilon ,t/\varepsilon \right) $, 
\begin{equation*}
(K^{\varepsilon }\ast a_{0}^{\varepsilon }\left( \cdot ,\cdot ,\nabla
u_{\varepsilon }\right) )(x,t)=\int_{0}^{\infty }K\left( \frac{x}{%
\varepsilon },\frac{t-\tau }{\varepsilon }\right) a_{0}\left( \frac{x}{%
\varepsilon },\frac{\tau }{\varepsilon },\nabla u_{\varepsilon }(x,\tau
)\right) d\tau \text{,}
\end{equation*}%
the functions $a:\mathbb{R}^{N}\times \mathbb{R}\times \mathbb{R}%
^{N}\rightarrow \mathbb{R}^{N}$, $a_{0}:\mathbb{R}^{N}\times \mathbb{R}%
\times \mathbb{R}^{N}\rightarrow \mathbb{R}$, $K:\mathbb{R}^{N}\times 
\mathbb{R}\rightarrow \mathbb{R}$ and $\rho :\mathbb{R}^{N}\rightarrow 
\mathbb{R}$ being constrained as follows:

\bigskip

\begin{itemize}
\item[\textbf{A1}] 
\begin{equation}
\text{For each }\lambda \in \mathbb{R}^{N}\text{, the functions }a(\cdot
,\cdot ,\lambda )\text{ and }a_{0}(\cdot ,\cdot ,\lambda )\text{ are
measurable;}  \label{4.2}
\end{equation}%
\begin{equation}
a(y,\tau ,0)=0\text{ almost everywhere (a.e.) in }(y,\tau )\in \mathbb{R}%
^{N}\times \mathbb{R}\text{;}  \label{4.3}
\end{equation}%
\begin{equation}
\begin{array}{l}
\text{There are three constants }c_{0},\,c_{1},\,c_{2}>0\text{ such that }
\\ 
\text{(i) }\left( a(y,\tau ,\lambda )-a(y,\tau ,\lambda ^{\prime })\right)
\cdot \left( \lambda -\lambda ^{\prime }\right) \geq c_{1}\left\vert \lambda
-\lambda ^{\prime }\right\vert ^{2} \\ 
\text{(ii) }\left\vert a(y,\tau ,\lambda )\right\vert +\left\vert
a_{0}(y,\tau ,\lambda )\right\vert \leq c_{2}(1+\left\vert \lambda
\right\vert ) \\ 
\text{(iii) }\left\vert a(y,\tau ,\lambda )-a(y,\tau ,\lambda ^{\prime
})\right\vert +\left\vert a_{0}(y,\tau ,\lambda )-a_{0}(y,\tau ,\lambda
^{\prime })\right\vert \leq c_{0}\left\vert \lambda -\lambda ^{\prime
}\right\vert \\ 
\text{for all }\lambda ,\lambda ^{\prime }\in \mathbb{R}^{N}\text{ and a.e.
in }(y,\tau )\in \mathbb{R}^{N}\times \mathbb{R}\text{, where the dot denotes%
} \\ 
\text{the usual Euclidean inner product in }\mathbb{R}^{N}\text{ and }%
\left\vert \cdot \right\vert \text{ the associated norm.}%
\end{array}
\label{4.4}
\end{equation}

\item[\textbf{A2}] $K\in L^{1}(\mathbb{R}^{N+1})$, $\rho \in L^{\infty }(%
\mathbb{R}^{N})$ and there exists $\Lambda >0$ such that $\Lambda ^{-1}\leq
\rho (y)\leq \Lambda $\ for a.e. $y\in \mathbb{R}^{N}$.
\end{itemize}

\medskip It is well known that the functions $a^{\varepsilon }(\cdot ,\cdot
,Dv)$ and $a_{0}^{\varepsilon }(\cdot ,\cdot ,Dv)$ (for fixed $v\in
L^{2}(0,T;W_{0}^{1,2}(\Omega ))$), $K^{\varepsilon }$ and $\rho
^{\varepsilon }$ are well defined as elements of $L^{2}(Q_{T})^{N}$, $%
L^{2}(Q_{T})^{N}$, $L^{1}(Q_{T})$ and $L^{\infty }(\Omega )$ respectively
and satisfy properties of the same type as in \textbf{A1}-\textbf{A2}.
Finally, choose $f$ in $L^{2}(Q_{T})$ and $u^{0}\in L^{2}(\Omega )$. Our
first objective in this section is to provide an existence and uniqueness
result for problem (\ref{4.1}).

\begin{theorem}
\label{t4.1}For any fixed $\varepsilon >0$, the problem \emph{(\ref{4.1})}
possesses a unique solution $u^{\varepsilon }\in
L^{2}(0,T;W_{0}^{1,2}(\Omega ))\cap \mathcal{C}(0,T;L^{2}(\Omega ))$ and the
following a priori estimates holds:%
\begin{equation}
\sup_{0\leq t\leq T}\left\Vert u_{\varepsilon }(t)\right\Vert _{L^{2}(\Omega
)}^{2}\leq C\text{ and }\int_{0}^{T}\left\Vert u_{\varepsilon
}(t)\right\Vert _{W_{0}^{1,2}(\Omega )}^{2}dt\leq C  \label{4.5}
\end{equation}%
where $C$ is a positive constant which does not depend on $\varepsilon $.
\end{theorem}

\begin{proof}
Let $z\in L^{2}(0,T;W_{0}^{1,2}(\Omega ))\cap \mathcal{C}(0,T;L^{2}(\Omega
)) $, and consider the following boundary value problem 
\begin{equation}
\begin{array}{l}
\rho ^{\varepsilon }\frac{\partial u_{\varepsilon }}{\partial t}-\Div %
a^{\varepsilon }\left( \cdot ,\cdot ,\nabla u_{\varepsilon }\right)
=f-K^{\varepsilon }\ast a_{0}^{\varepsilon }\cdot ,\cdot ,\nabla z)\text{ in 
}Q_{T} \\ 
\ \ \ \ \ \ \ \ \ \ \ \ \ \ \ \ \ \ \ \ \ \ \ \ \ \ \ \ \ \ \ \ \ \ \ \ \ \
\ \ \ u_{\varepsilon }=0\text{ on }\partial \Omega \times \left( 0,T\right)
\\ 
\ \ \ \ \ \ \ \ \ \ \ \ \ \ \ \ \ \ \ \ \ \ \ \ \ \ \ \ \ \ \ \ \ \ \ \ \ \
\ \ \ u_{\varepsilon }(x,0)=u^{0}(x)\text{ in }\Omega .%
\end{array}
\label{4.6}
\end{equation}%
Using a\ standard fashion (see e.g. \cite{Show}), we derive the existence
and uniqueness of a solution $u_{\varepsilon }\in
L^{2}(0,T;W_{0}^{1,2}(\Omega ))\cap \mathcal{C}(0,T;L^{2}(\Omega ))$ to (\ref%
{4.6}). Thus we have defined a mapping $z\mapsto u_{\varepsilon }$ from $%
X=L^{2}(0,T;W_{0}^{1,2}(\Omega ))\cap \mathcal{C}(0,T;L^{2}(\Omega ))$ into
itself. We need to show that this mapping is contractive. To this end, let
us endow $X$ with the norm 
\begin{equation*}
\left\Vert u\right\Vert _{X}=\left( \sup_{0\leq t\leq T}\left\Vert
u(t)\right\Vert _{L^{2}(\Omega )}^{2}+\left\Vert u\right\Vert
_{L^{2}(0,T;W_{0}^{1,2}(\Omega ))}^{2}\right) ^{\frac{1}{2}}\ \ (u\in X).
\end{equation*}%
Let $z_{1},z_{2}\in X$, and consider, for $j=1,2$, the solution $u_{j}$ of
the corresponding PDE. We have, for any $\phi \in
L^{2}(0,T;W_{0}^{1,2}(\Omega ))$ and any $0\leq t\leq T$, 
\begin{eqnarray*}
&&\int_{0}^{t}\left( \rho ^{\varepsilon }\frac{\partial u_{j}}{\partial t}%
,\phi \right) ds+\int_{Q_{t}}a^{\varepsilon }(\cdot ,\cdot ,\nabla
u_{j})\cdot \nabla \phi dxds \\
&=&\int_{Q_{t}}f\phi dxds+\int_{Q_{t}}(K^{\varepsilon }\ast
a_{0}^{\varepsilon }(\cdot ,\cdot ,\nabla z_{j}))\phi dxds,
\end{eqnarray*}%
where $Q_{t}=\Omega \times \left( 0,t\right) $, hence 
\begin{eqnarray*}
&&\int_{0}^{t}\left( \rho ^{\varepsilon }\frac{\partial (u_{1}-u_{2})}{%
\partial t},\phi \right) ds+\int_{Q_{t}}(a^{\varepsilon }(\cdot ,\cdot
,\nabla u_{1})-a^{\varepsilon }(\cdot ,\cdot ,\nabla u_{2}))\cdot \nabla
\phi dxds \\
&=&\int_{Q_{t}}(K^{\varepsilon }\ast (a_{0}^{\varepsilon }(\cdot ,\cdot
,\nabla z_{1})-a_{0}^{\varepsilon }(\cdot ,\cdot ,\nabla z_{2})))\phi dxds.
\end{eqnarray*}%
Taking $\phi =u_{1}-u_{2}$, assumptions \textbf{A1}-\textbf{A2} entail 
\begin{eqnarray*}
&&\frac{1}{2}\Lambda ^{-1}\left\Vert u_{1}(t)-u_{2}(t)\right\Vert
_{L^{2}(\Omega )}^{2}+c_{1}\int_{Q_{t}}\left\vert \nabla u_{1}-\nabla
u_{2}\right\vert ^{2}dxds \\
&\leq &\left\Vert K^{\varepsilon }\ast (a_{0}^{\varepsilon }(\cdot ,\cdot
,\nabla u_{1})-a_{0}^{\varepsilon }(\cdot ,\cdot ,\nabla u_{2}))\right\Vert
_{L^{2}(Q_{t})}\left\Vert u_{1}-u_{2}\right\Vert _{L^{2}(Q_{t})}.
\end{eqnarray*}%
But, in view of Poincar\'{e}'s inequality, there is a positive constant $%
\alpha $ depending only on $\Omega $ such that 
\begin{equation*}
\left\Vert u_{1}-u_{2}\right\Vert _{L^{2}(Q_{t})}\leq \alpha \left\Vert
\nabla u_{1}-\nabla u_{2}\right\Vert _{L^{2}(Q_{t})}^{2}\text{ for any }%
0\leq t\leq T\text{,}
\end{equation*}%
hence 
\begin{eqnarray*}
&&\left\Vert K^{\varepsilon }\ast (a_{0}^{\varepsilon }(\cdot ,\cdot ,\nabla
u_{1})-a_{0}^{\varepsilon }(\cdot ,\cdot ,\nabla u_{2}))\right\Vert
_{L^{2}(Q_{t})}\left\Vert u_{1}-u_{2}\right\Vert _{L^{2}(Q_{t})} \\
&\leq &\alpha \left\Vert K^{\varepsilon }\ast (a_{0}^{\varepsilon }(\cdot
,\cdot ,\nabla u_{1})-a_{0}^{\varepsilon }(\cdot ,\cdot ,\nabla
u_{2}))\right\Vert _{L^{2}(Q_{t})}\left\Vert \nabla u_{1}-\nabla
u_{2}\right\Vert _{L^{2}(Q_{t})} \\
&\leq &\frac{\alpha ^{2}}{2c_{1}}\left\Vert K^{\varepsilon }\ast
(a_{0}^{\varepsilon }(\cdot ,\cdot ,\nabla z_{1})-a_{0}^{\varepsilon }(\cdot
,\cdot ,\nabla z_{2}))\right\Vert _{L^{2}(Q_{t})}^{2}+\frac{c_{1}}{2}%
\left\Vert \nabla u_{1}-\nabla u_{2}\right\Vert _{L^{2}(Q_{t})}^{2}.
\end{eqnarray*}%
It readily holds that 
\begin{eqnarray*}
&&\frac{1}{2}\Lambda ^{-1}\left\Vert u_{1}(t)-u_{2}(t)\right\Vert
_{L^{2}(\Omega )}^{2}+\frac{c_{1}}{2}\int_{Q_{t}}\left\vert \nabla
u_{1}-\nabla u_{2}\right\vert ^{2}dxds \\
&\leq &\frac{\alpha ^{2}}{2c_{1}}\left\Vert K^{\varepsilon }\right\Vert
_{L^{1}(Q_{t})}^{2}\left\Vert a_{0}^{\varepsilon }(\cdot ,\cdot ,\nabla
z_{1})-a_{0}^{\varepsilon }(\cdot ,\cdot ,\nabla z_{2})\right\Vert
_{L^{2}(Q_{t})}^{2} \\
&\leq &\frac{\alpha ^{2}}{2c_{1}}\left\Vert K\right\Vert _{L^{1}(\mathbb{R}%
^{N+1})}^{2}\left\Vert \nabla z_{1}-\nabla z_{2}\right\Vert
_{L^{2}(Q_{T})}^{2}\text{ for a.e. }0\leq t\leq T.
\end{eqnarray*}%
Taking the supremum over $[0,T]$ we end up with 
\begin{eqnarray*}
&&\min \left( \frac{1}{2}\Lambda ^{-1},\frac{c_{1}}{2}\right) \left[
\sup_{0\leq t\leq T}\left\Vert u_{1}(t)-u_{2}(t)\right\Vert _{L^{2}(\Omega
)}^{2}+\int_{Q_{T}}\left\vert \nabla u_{1}-\nabla u_{2}\right\vert ^{2}dxds%
\right] \\
&\leq &\frac{\alpha ^{2}T}{2c_{1}}\left\Vert K\right\Vert _{L^{1}(\mathbb{R}%
^{N+1})}^{2}\left\Vert \nabla z_{1}-\nabla z_{2}\right\Vert
_{L^{2}(Q_{T})}^{2},
\end{eqnarray*}%
hence 
\begin{equation*}
\left\Vert u_{1}-u_{2}\right\Vert _{X}\leq C\sqrt{T}\left\Vert
z_{1}-z_{2}\right\Vert _{X}\text{ \ where }C=\left[ \frac{\frac{\alpha ^{2}}{%
2c_{1}}\left\Vert K\right\Vert _{L^{1}(\mathbb{R}^{N+1})}^{2}}{\min \left( 
\frac{1}{2}\Lambda ^{-1},\frac{c_{1}}{2}\right) }\right] ^{\frac{1}{2}}.
\end{equation*}%
Consequently, by shrinking $T>0$ in such a way that $C\sqrt{T}<1$, it
emerges that the above mapping is contractive. The rest of the proof follows
by mere routine.
\end{proof}

\begin{remark}
\label{r4.1}\emph{It follows from the inequalities (\ref{4.5}) that the
sequence }$(u_{\varepsilon })_{\varepsilon >0}$\emph{\ determined above is
relatively compact in the space }$L^{2}(Q_{T})$\emph{.}
\end{remark}

Throughout the rest of this section, $A_{y}$ and $A_{\tau }$ are algebras
with mean value on $\mathbb{R}_{y}^{N}$ and $\mathbb{R}_{\tau }$
respectively. We set $A=A_{y}\odot A_{\tau }$ and further we assume that $%
A_{\tau }$ is introverted. It is worth noting that property (\ref{2.3'}) is
still valid for $\psi \in \mathcal{C}(\overline{Q}_{T};B_{A}^{2,\infty })$
where $B_{A}^{2,\infty }=B_{A}^{2}\cap L^{\infty }(\mathbb{R}_{y,\tau
}^{N+1})$.

Bearing this in mind, let $(u_{\varepsilon })_{\varepsilon >0}$ be the
sequence of solutions to (\ref{4.1}). Our main objective here amounts to
study the asymptotic behaviour as $\varepsilon \rightarrow 0$, of $%
(u_{\varepsilon })_{\varepsilon >0}$. This will of course arise from the
following important assumption: 
\begin{equation}
\begin{array}{l}
a_{0}(\cdot ,\cdot ,\lambda )\in B_{A}^{2}\text{ and }a(\cdot ,\cdot
,\lambda )\in (B_{A}^{2})^{N}\text{ for any }\lambda \in \mathbb{R}^{N}, \\ 
K\in B_{A}^{1}\text{ and }\rho \in A_{y}\text{ with }M(\rho )>0.%
\end{array}
\label{4.8}
\end{equation}

The homogenization of problems of type (\ref{4.1}) has been left opened in 
\cite{AD86} in which the authors considered only the linear version of such
type of equations. They used the Laplace transform to perform the
homogenization process. Our work is therefore the first one in which the
homogenization of (\ref{4.1}) is considered, even in the periodic setting.

This being so, let $\Psi \in \mathcal{C}(\overline{Q}_{T};(A)^{N})$. Suppose
that (\ref{4.8}) is satisfied. It can be shown (as in \cite{EJDE}) that the
function $(x,t,y,\tau )\mapsto a(y,\tau ,\Psi (x,t,y,\tau ))$, denoted below
by $a(\cdot ,\cdot ,\Psi )$, lies in $\mathcal{C}(\overline{Q}%
_{T};(B_{A}^{2,\infty })^{N})$, and we can therefore define its trace $%
(x,t)\mapsto a(x/\varepsilon ,t/\varepsilon ,\Psi (x,t,x/\varepsilon
,t/\varepsilon ))$ ($\varepsilon >0$) denoted by $a^{\varepsilon }(\cdot
,\cdot ,\Psi ^{\varepsilon })$. The same is true for $a_{0}(\cdot ,\cdot
,\Psi )$ and $a_{0}^{\varepsilon }(\cdot ,\cdot ,\Psi ^{\varepsilon })$.

The proof of the next two results can be found in \cite{EJDE} (see
Proposition 3.1 and Corollary 3.1 therein).

\begin{proposition}
\label{p4.1}Suppose \emph{(\ref{4.8})} holds. For $\Psi \in \mathcal{C}(%
\overline{Q}_{T};(A)^{N}))$ we have 
\begin{equation*}
a^{\varepsilon }(\cdot ,\cdot ,\Psi ^{\varepsilon })\rightarrow a(\cdot
,\cdot ,\Psi )\text{ in }L^{2}(Q_{T})^{N}\text{-weak }\Sigma \text{ as }%
\varepsilon \rightarrow 0.
\end{equation*}%
The mapping $\Psi \mapsto a(\cdot ,\cdot ,\Psi )$ of $\mathcal{C}(\overline{Q%
}_{T};(A)^{N}))$ into $L^{2}(Q_{T};B_{A}^{2})^{N}$ extends by continuity to
a unique mapping still denoted by $a$, of $L^{2}(Q_{T};(B_{A}^{2})^{N})$
into $L^{2}(Q_{T};B_{A}^{2})^{N}$ such that 
\begin{equation*}
(a(\cdot ,\cdot ,\mathbf{v})-a(\cdot ,\cdot ,\mathbf{w}))\cdot (\mathbf{v}-%
\mathbf{w})\geq c_{1}\left\vert \mathbf{v}-\mathbf{w}\right\vert ^{2}\text{\
a.e. in }Q_{T}\times \mathbb{R}_{y}^{N}\times \mathbb{R}_{\tau }
\end{equation*}%
\begin{equation*}
\left\Vert a(\cdot ,\cdot ,\mathbf{v})-a(\cdot ,\cdot ,\mathbf{w}%
)\right\Vert _{L^{2}(Q_{T};B_{A}^{2})^{N}}\leq c_{0}\left\Vert \mathbf{v}-%
\mathbf{w}\right\Vert _{L^{2}(Q_{T};(B_{A}^{2})^{N})}
\end{equation*}%
\begin{equation*}
a(\cdot ,\cdot ,0)=0\;\;\text{a.e. in }\mathbb{R}_{y}^{N}\times \mathbb{R}%
_{\tau }\;\;\;\;\;\;\;\;\;\;\;\;\;\;\;\;\;\;\;\;\;\;\;\;\;\;
\end{equation*}%
for all $\mathbf{v},\mathbf{w}\in L^{2}(Q_{T};(B_{A}^{2})^{N})$.
\end{proposition}

\begin{corollary}
\label{c4.1}Let $\psi _{0}\in \mathcal{C}_{0}^{\infty }(Q_{T})$ and $\psi
_{1}\in \mathcal{C}_{0}^{\infty }(Q_{T})\otimes A^{\infty }$. For $%
\varepsilon >0$, let 
\begin{equation}
\Phi _{\varepsilon }=\psi _{0}+\varepsilon \psi _{1}^{\varepsilon
},\;\;\;\;\;\;\;\;\;\;\;\;\;\;  \label{4.9}
\end{equation}%
i.e., $\Phi _{\varepsilon }(x,t)=\psi _{0}(x,t)+\varepsilon \psi
_{1}(x,t,x/\varepsilon ,t/\varepsilon )$\ for $(x,t)\in Q_{T}$. Let $%
(v_{\varepsilon })_{\varepsilon \in E}$ is a sequence in $L^{2}(Q_{T})^{N}$
such that $v_{\varepsilon }\rightarrow v_{0}$ in $L^{2}(Q_{T})$-weak $\Sigma 
$ as $E\ni \varepsilon \rightarrow 0$ where $\mathbf{v}_{0}\in L^{2}(Q_{T};%
\mathcal{B}_{A}^{2})$, then, as $E\ni \varepsilon \rightarrow 0$, 
\begin{equation*}
\int_{Q_{T}}a^{\varepsilon }(\cdot ,\cdot ,D\Phi _{\varepsilon
})v_{\varepsilon }dxdt\rightarrow \iint_{Q_{T}\times \Delta (A_{y})\times
K(\Delta (A_{\tau }))}\widehat{a}(\cdot ,\cdot ,D\psi _{0}+\partial \widehat{%
\psi }_{1})\widehat{v}_{0}dxdtd\beta \text{.}
\end{equation*}
\end{corollary}

\begin{remark}
\label{r5.1}\emph{The conclusion of the above results also hold true for the
mapping }$a_{0}$\emph{, mutatis mutandis.}
\end{remark}

Now, let 
\begin{eqnarray*}
V &=&\{u\in L^{2}(0,T;W_{0}^{1,2}(\Omega )):u^{\prime }\in
L^{2}(0,T;W^{-1,2}(\Omega ))\}\text{ and } \\
\mathbb{F}_{0}^{1} &=&V\times L^{2}(Q_{T};\mathcal{B}_{A_{\tau }}^{2}(%
\mathbb{R}_{\tau };\mathcal{B}_{A_{y}}^{1,2})).
\end{eqnarray*}%
Endowed with its natural topology, $\mathbb{F}_{0}^{1}$ is a Hilbert space
admitting 
\begin{equation*}
F_{0}^{\infty }=\mathcal{C}_{0}^{\infty }(Q_{T})\times \left( \mathcal{C}%
_{0}^{\infty }(Q_{T})\otimes \left[ \mathcal{D}_{A_{\tau }}(\mathbb{R}_{\tau
})\otimes \mathcal{D}_{A_{y}}(\mathbb{R}_{y}^{N})\right] \right)
\end{equation*}%
as a dense subspace.

Bearing this in mind, let $(u_{\varepsilon })_{\varepsilon >0}$ be a
sequence of solutions to (\ref{4.1}), and let $E=(\varepsilon _{n})_{n}$ be
an ordinary sequence of positive real numbers converging to zero. Since $%
(u_{\varepsilon })_{\varepsilon \in E}$ is relatively compact in $%
L^{2}(Q_{T})$, there exists a subsequence $E^{\prime }$ of $E$ and $u_{0}\in
L^{2}(Q_{T})$ such that, as $E^{\prime }\ni \varepsilon \rightarrow 0$, 
\begin{equation}
u_{\varepsilon }\rightarrow u_{0}\text{ in }L^{2}(Q_{T}).  \label{4.10}
\end{equation}%
In view of (\ref{4.5}) and by the diagonal process, we find a subsequence of 
$(u_{\varepsilon })_{\varepsilon \in E^{\prime }}$ still denoted by $%
(u_{\varepsilon })_{\varepsilon \in E^{\prime }}$ which weakly converges in $%
L^{2}(0,T;W_{0}^{1,2}(\Omega ))$ to $u_{0}$, hence $u_{0}\in
L^{2}(0,T;W_{0}^{1,2}(\Omega ))$. We infer from both Theorem \ref{t2.3} and
Remark \ref{r2.4} the existence of a function $u_{1}\in L^{2}(Q_{T};\mathcal{%
B}_{A_{\tau }}^{2}(\mathbb{R}_{\tau };\mathcal{B}_{A_{y}}^{1,2}))$ such
that, as $E^{\prime }\ni \varepsilon \rightarrow 0$, 
\begin{equation}
\frac{\partial u_{\varepsilon }}{\partial x_{j}}\rightarrow \frac{\partial
u_{0}}{\partial x_{j}}+\frac{\overline{\partial }u_{1}}{\partial y_{j}}\text{
in }L^{2}(Q_{T})\text{-weak }\Sigma \ \ (1\leq j\leq N).  \label{4.11}
\end{equation}

This being so, for $\boldsymbol{v}=(v_{0},v_{1})\in \mathbb{F}_{0}^{1}$ we
set $\overline{\mathbb{D}}_{y}\boldsymbol{v}=\nabla v_{0}+\overline{\nabla }%
_{y}v_{1}$ and $\mathbb{D}\boldsymbol{v}=\mathcal{G}_{1}\circ \overline{%
\mathbb{D}}_{y}\boldsymbol{v}\equiv \nabla v_{0}+\partial \widehat{v}_{1}$.
We consider the variational problem 
\begin{equation}
\left\{ 
\begin{array}{l}
\text{Find }\boldsymbol{u}=(u_{0},u_{1})\in \mathbb{F}_{0}^{1}\text{ such
that} \\ 
M(\rho )\int_{0}^{T}\left\langle u_{0}^{\prime }(t),v_{0}(t)\right\rangle
dt+\iint_{Q_{T}\times \Delta (A_{y})\times K(\Delta (A_{\tau }))}\widehat{a}%
(\cdot ,\cdot ,\mathbb{D}\boldsymbol{u})\cdot \mathbb{D}\boldsymbol{v}%
dxdtd\beta  \\ 
\ \ +\iint_{Q_{T}\times \Delta (A_{y})\times K(\Delta (A_{\tau }))}(\widehat{%
K}\widehat{\ast \ast }\widehat{a}_{0}(\cdot ,\cdot ,\mathbb{D}\boldsymbol{u}%
))v_{0}dxdtd\beta =\int_{Q_{T}}fv_{0}dxdt \\ 
\text{for all }\boldsymbol{v}=(v_{0},v_{1})\in \mathbb{F}_{0}^{1},%
\end{array}%
\right.   \label{4.12}
\end{equation}%
where the brackets $\left\langle \cdot ,\cdot \right\rangle $ henceforth
stand for the duality paring between $W_{0}^{1,2}(\Omega )$ and $%
W^{-1,2}(\Omega )$ and $u_{0}^{\prime }(t)=u_{0}^{\prime }(\cdot ,t)$ (same
definition for $v_{0}(t)$), and the convolution is with respect to the time
argument, that is, 
\begin{equation*}
(\widehat{K}\ast \ast \widehat{a}_{0}(\cdot ,\cdot ,\mathbb{D}\boldsymbol{u}%
))(x,t,s,s_{0})=\int_{\mathbb{R}_{\zeta }}\int_{K(\Delta (A_{\tau }))}%
\widehat{K}(s,s_{0}r_{0}^{-1})\widehat{a}_{0}(s,r_{0},\mathbb{D}\boldsymbol{u%
}(x,\zeta ,s,r_{0}))d\beta _{\tau }(r_{0})d\zeta .
\end{equation*}%
It can be shown that the above problem possesses at most one solution. The
following global homogenization result holds.

\begin{theorem}
\label{t4.2}Assume \emph{(\ref{4.8})} holds true. Then the couple $%
\boldsymbol{u}=(u_{0},u_{1})$ determined by \emph{(\ref{4.10})-(\ref{4.11})}
solves the variational problem \emph{(\ref{4.12})}.
\end{theorem}

\begin{proof}
Let $(u_{0},u_{1})$ be as in (\ref{4.10})-(\ref{4.11}). Then it belongs to $%
L^{2}(0,T;W_{0}^{1,2}(\Omega ))\times L^{2}(Q_{T};\mathcal{B}_{A_{\tau
}}^{2}(\mathbb{R}_{\tau };\mathcal{B}_{A_{y}}^{1,2}))$. It remains to check
that $u_{0}^{\prime }\in L^{2}(0,T;W^{-1,2}(\Omega ))$ and that $\boldsymbol{%
u}$ solves (\ref{4.12}). First, let $\psi \in \mathcal{C}_{0}^{\infty
}(Q_{T})$. The variational formulation of (\ref{4.1}) with $\psi $ as a test
function yields 
\begin{eqnarray*}
&&-\int_{Q_{T}}\rho ^{\varepsilon }u_{\varepsilon }\frac{\partial \psi }{%
\partial t}dxdt+\int_{Q_{T}}a^{\varepsilon }(\cdot ,\cdot ,\nabla
u_{\varepsilon })\cdot \nabla \psi dxdt+\int_{Q_{T}}(K^{\varepsilon }\ast
a_{0}^{\varepsilon }(\cdot ,\cdot ,\nabla u_{\varepsilon }))\psi dxdt \\
&=&\int_{Q_{T}}f\psi dxdt.
\end{eqnarray*}%
Thanks to the second estimate in (\ref{4.5}), the sequences $a^{\varepsilon
}(\cdot ,\cdot ,\nabla u_{\varepsilon })$ and $(K^{\varepsilon }\ast
a_{0}^{\varepsilon }(\cdot ,\cdot ,\nabla u_{\varepsilon }))$ are bounded in 
$L^{2}(Q_{T})^{N}$ and in $L^{2}(Q_{T})$ respectively. Hence there exists a
subsequence of $E^{\prime }$ ($E^{\prime }$ determined in (\ref{4.11})) not
relabeled and two functions $z\in L^{2}(Q_{T})^{N}$ and $z_{0}\in
L^{2}(Q_{T})$ such that, as $E^{\prime }\ni \varepsilon \rightarrow 0$, $%
a^{\varepsilon }(\cdot ,\cdot ,\nabla u_{\varepsilon })\rightarrow z$ in $%
L^{2}(Q_{T})^{N}$-weak and $K^{\varepsilon }\ast a_{0}^{\varepsilon }(\cdot
,\cdot ,\nabla u_{\varepsilon })\rightarrow z_{0}$ in $L^{2}(Q_{T})$-weak.
Letting $E^{\prime }\ni \varepsilon \rightarrow 0$ in the above equation and
using (\ref{4.10}), we obtain 
\begin{eqnarray*}
&&-M(\rho )\int_{Q_{T}}u_{0}\frac{\partial \psi }{\partial t}%
dxdt+\int_{Q_{T}}z\cdot \nabla \psi dxdt+\int_{Q_{T}}z_{0}\psi dxdt \\
&=&\int_{Q_{T}}f\psi dxdt.
\end{eqnarray*}%
It follows that 
\begin{equation*}
M(\rho )\frac{\partial u_{0}}{\partial t}=\Div z-z_{0}+f.
\end{equation*}%
Since $M(\rho )\neq 0$, it follows by mere routine that $\frac{\partial u_{0}%
}{\partial t}\in L^{2}(0,T;W^{-1,2}(\Omega ))$, so that $u_{0}\in V$. This
shows that $\boldsymbol{u}=(u_{0},u_{1})\in \mathbb{F}_{0}^{1}$. Let us now
check that $\boldsymbol{u}$ satisfies (\ref{4.12}). To achieve this, let $%
\Phi =(\psi _{0},\varrho (\psi _{1}))\in F_{0}^{\infty }$ where $\psi
_{1}\in \mathcal{C}_{0}^{\infty }(Q_{T})\otimes A_{y}^{\infty }\otimes
A_{\tau }^{\infty }$ ($\varrho $ being denoting the canonical mapping of $%
B_{A}^{2}$ into $\mathcal{B}_{A}^{2}$), and define $\Phi _{\varepsilon }$ as
in (\ref{4.9}) (see Corollary \ref{c4.1}). Then we have $\Phi _{\varepsilon
}\in \mathcal{C}_{0}^{\infty }(Q_{T})$, and by the monotonicity of $a$, we
have 
\begin{equation*}
\int_{Q_{T}}\left( a^{\varepsilon }(\cdot ,\cdot ,\nabla u_{\varepsilon
})-a^{\varepsilon }(\cdot ,\cdot ,\nabla \Phi _{\varepsilon })\right) \cdot
(\nabla u_{\varepsilon }-\nabla \Phi _{\varepsilon })dxdt\geq 0,
\end{equation*}%
or, owing to (\ref{4.1}), 
\begin{equation}
\begin{array}{l}
\frac{1}{2}\int_{\Omega }\rho ^{\varepsilon }\left\vert u_{\varepsilon
}(T)\right\vert ^{2}dx\leq \frac{1}{2}\int_{\Omega }\rho ^{\varepsilon
}\left\vert u^{0}\right\vert ^{2}dx+\int_{Q_{T}}f(u_{\varepsilon }-\Phi
_{\varepsilon })dxdt \\ 
\ \ \ -\int_{Q_{T}}\rho ^{\varepsilon }u_{\varepsilon }\frac{\partial \Phi
_{\varepsilon }}{\partial t}dxdt+\int_{Q_{T}}(K^{\varepsilon }\ast
a_{0}^{\varepsilon }(\cdot ,\cdot ,\nabla u_{\varepsilon }))(u_{\varepsilon
}-\Phi _{\varepsilon })dxdt \\ 
\ \ \ \ \ -\int_{Q_{T}}a^{\varepsilon }(\cdot ,\cdot ,\nabla \Phi
_{\varepsilon })\cdot \nabla (u_{\varepsilon }-\Phi _{\varepsilon })dxdt.%
\end{array}
\label{4.13}
\end{equation}%
We pass to the limit in (\ref{4.13}) by considering each term separately.
First, in view of (\ref{4.10}), it is an easy exercise to see that 
\begin{equation*}
\iint_{\Omega \times \Delta (A_{y})}\widehat{\rho }(s)\left\vert
u_{0}(T)\right\vert ^{2}dxd\beta _{y}\leq ~\underset{E^{\prime }\ni
\varepsilon \rightarrow 0}{\lim \inf }\int_{\Omega }\rho ^{\varepsilon
}\left\vert u_{\varepsilon }(T)\right\vert ^{2}dx,
\end{equation*}%
that is, 
\begin{equation*}
M(\rho )\int_{\Omega }\left\vert u_{0}(T)\right\vert ^{2}dx\leq ~\underset{%
E^{\prime }\ni \varepsilon \rightarrow 0}{\lim \inf }\int_{\Omega }\rho
^{\varepsilon }\left\vert u_{\varepsilon }(T)\right\vert ^{2}dx.
\end{equation*}%
Next, we have from the definition of the mean value that 
\begin{equation*}
\int_{\Omega }\rho ^{\varepsilon }\left\vert u^{0}\right\vert
^{2}dx\rightarrow M(\rho )\int_{\Omega }\left\vert u^{0}\right\vert ^{2}dx%
\text{ when }E^{\prime }\ni \varepsilon \rightarrow 0.
\end{equation*}%
Considering the next term, we obviously have, as $E^{\prime }\ni \varepsilon
\rightarrow 0$, 
\begin{equation*}
\int_{Q_{T}}f(u_{\varepsilon }-\Phi _{\varepsilon })dxdt\rightarrow
\int_{Q_{T}}f(u_{0}-\psi _{0})dxdt.
\end{equation*}%
In view of (\ref{4.10}) associated to the convergence result $\left( \frac{%
\partial \psi _{1}}{\partial \tau }\right) ^{\varepsilon }\rightarrow
M\left( \frac{\partial \psi _{1}}{\partial \tau }\right) =0$ in $L^{2}(Q_{T})
$-weak, it holds that 
\begin{equation*}
\int_{Q_{T}}\rho ^{\varepsilon }u_{\varepsilon }\frac{\partial \Phi
_{\varepsilon }}{\partial t}dxdt\rightarrow M(\rho )\int_{Q_{T}}u_{0}\frac{%
\partial \psi _{0}}{\partial t}dxdt=M(\rho )\int_{0}^{T}\left\langle
u_{0}^{\prime }(t),\psi _{0}(t)\right\rangle dt.
\end{equation*}%
Now, because of Corollary \ref{c4.1} we obtain, when $E^{\prime }\ni
\varepsilon \rightarrow 0$, 
\begin{equation*}
\int_{Q_{T}}a^{\varepsilon }(\cdot ,\cdot ,\nabla \Phi _{\varepsilon })\cdot
\nabla (u_{\varepsilon }-\Phi _{\varepsilon })dxdt\rightarrow
\iint_{Q_{T}\times \Delta (A_{y})\times K(\Delta (A_{\tau }))}\widehat{a}%
(\cdot ,\cdot ,\mathbb{D}\Phi )\cdot \mathbb{D}(\boldsymbol{u}-\Phi
)dxdtd\beta .
\end{equation*}%
Finally, for the last term, due to the estimate (\ref{4.5}), we have that
the sequence $(a_{0}^{\varepsilon }(\cdot ,\cdot ,\nabla u_{\varepsilon
}))_{\varepsilon \in E^{\prime }}$ is bounded in $L^{2}(Q_{T})$ so that
there exist a function $z_{0}\in L^{2}(Q_{T};\mathcal{B}_{A}^{2})$ such
that, up to a subsequence of $E^{\prime }$, $a_{0}^{\varepsilon }(\cdot
,\cdot ,\nabla u_{\varepsilon })\rightarrow z_{0}$ in $L^{2}(Q_{T})$-weak $%
\Sigma $. On the other hand, since $K\in B_{A}^{1}$ we have that $%
K^{\varepsilon }\rightarrow K$ in $L^{1}(\mathbb{R}^{N+1})$-strong $\Sigma $%
. It there emerges from Theorem \ref{t2.4} that 
\begin{equation*}
K^{\varepsilon }\ast a_{0}^{\varepsilon }(\cdot ,\cdot ,\nabla
u_{\varepsilon })\rightarrow K\ast \ast z_{0}\text{ in }L^{2}(Q_{T})\text{%
-weak }\Sigma \text{ as }E^{\prime }\ni \varepsilon \rightarrow 0.
\end{equation*}%
We use once again (\ref{4.10}) to get, when $E^{\prime }\ni \varepsilon
\rightarrow 0$, 
\begin{equation*}
\int_{Q_{T}}(K^{\varepsilon }\ast a_{0}^{\varepsilon }(\cdot ,\cdot ,\nabla
u_{\varepsilon }))(u_{\varepsilon }-\Phi _{\varepsilon })dxdt\rightarrow
\iint_{Q_{T}\times \Delta (A_{y})\times K(\Delta (A_{\tau }))}(\widehat{K}%
\widehat{\ast \ast }\widehat{z}_{0})(u_{0}-\psi _{0})dxdtd\beta .
\end{equation*}%
Putting together all the above convergence results and taking the $\lim
\inf_{E^{\prime }\ni \varepsilon \rightarrow 0}$ in (\ref{4.13}), we end up
with 
\begin{equation}
\begin{array}{l}
0\leq \int_{0}^{T}\left\langle f(t)-M(\rho )u_{0}^{\prime }(t),u_{0}(t)-\psi
_{0}(t)\right\rangle dt \\ 
\ \ -\iint_{Q_{T}\times \Delta (A_{y})\times K(\Delta (A_{\tau }))}\widehat{a%
}(\cdot ,\cdot ,\mathbb{D}\Phi )\cdot \mathbb{D}(\mathbf{u}-\Phi )dxdtd\beta 
\\ 
\ \ \ -\iint_{Q_{T}\times \Delta (A_{y})\times K(\Delta (A_{\tau }))}(%
\widehat{K}\widehat{\ast \ast }\widehat{z}_{0})(u_{0}-\psi _{0})dxdtd\beta 
\text{ for all }\Phi \in F_{0}^{\infty }.%
\end{array}
\label{4.14}
\end{equation}%
Since $F_{0}^{\infty }$ is dense in $\mathbb{F}_{0}^{1}$, after considering
a continuity argument, we see that (\ref{4.14}) still holds for any $\Phi
\in \mathbb{F}_{0}^{1}$. Taking in (\ref{4.14}) the particular functions $%
\Phi =\boldsymbol{u}-\lambda \boldsymbol{v}$ with $\lambda >0$ and $%
\boldsymbol{v}=(v_{0},v_{1})\in \mathbb{F}_{0}^{1}$, then dividing both
sides of the resulting inequality by $\lambda $, and letting $\lambda
\rightarrow 0$, and finally changing $\boldsymbol{v}$ into $-\boldsymbol{v}$%
, leads to 
\begin{equation*}
\begin{array}{l}
M(\rho )\int_{0}^{T}\left\langle u_{0}^{\prime }(t),v_{0}(t)\right\rangle
dt+\iint_{Q_{T}\times \Delta (A_{y})\times K(\Delta (A_{\tau }))}\widehat{a}%
(\cdot ,\cdot ,\mathbb{D}\boldsymbol{u})\cdot \mathbb{D}\boldsymbol{v}%
dxdtd\beta  \\ 
\;+\iint_{Q_{T}\times \Delta (A_{y})\times K(\Delta (A_{\tau }))}(\widehat{K}%
\widehat{\ast \ast }\widehat{z}_{0})v_{0}dxdtd\beta
=\int_{0}^{T}\left\langle f(t),v_{0}(t)\right\rangle dt \\ 
\text{for all }\boldsymbol{v}=(v_{0},v_{1})\in \mathbb{F}_{0}^{1}.%
\end{array}%
\end{equation*}%
The last part of the proof consists in identifying the function $z_{0}$. It
is sufficient to check that 
\begin{equation*}
K\ast \ast z_{0}=K\ast \ast a_{0}(\cdot ,\cdot ,\overline{\mathbb{D}}_{y}%
\boldsymbol{u}).
\end{equation*}%
To this end, fix $\eta >0$ and choose $\psi _{0}\in \mathcal{C}_{0}^{\infty
}(Q_{T})$ and $\psi _{1}\in \mathcal{C}_{0}^{\infty }(Q_{T})\otimes
A^{\infty }$ be such that 
\begin{equation*}
\left\Vert u_{0}-\psi _{0}\right\Vert _{L^{2}(0,T;W_{0}^{1,2}(\Omega ))}<%
\frac{\eta }{4c_{0}\left\Vert K\right\Vert _{B_{A}^{1}}}\text{ and }%
\left\Vert u_{1}-\varrho (\psi _{1})\right\Vert _{L^{2}(Q_{T};\mathcal{B}%
_{A_{\tau }}^{2}(\mathbb{R}_{\tau };\mathcal{B}_{A_{y}}^{1,2}))}<\frac{\eta 
}{4c_{0}\left\Vert K\right\Vert _{B_{A}^{1}}}.
\end{equation*}%
Then letting $\Phi =(\psi _{0},\varrho (\psi _{1}))\in F_{0}^{\infty }$, it
holds that 
\begin{equation*}
\begin{array}{l}
\left\Vert K\ast \ast \left( a_{0}(\cdot ,\cdot ,\overline{\mathbb{D}}_{y}%
\boldsymbol{u})-z_{0}\right) \right\Vert _{L^{2}(Q_{T};\mathcal{B}_{A}^{2})}
\\ 
\leq \left\Vert K\ast \ast \left( a_{0}(\cdot ,\cdot ,\overline{\mathbb{D}}%
_{y}\boldsymbol{u})-a_{0}(\cdot ,\cdot ,\nabla u_{0}+\nabla _{y}\psi
_{1})\right) \right\Vert _{L^{2}(Q_{T};\mathcal{B}_{A}^{2})} \\ 
\ \ \ \ \ +\left\Vert K\ast \ast \left( a_{0}(\cdot ,\cdot ,\nabla
u_{0}+\nabla _{y}\psi _{1})-a_{0}(\cdot ,\cdot ,\overline{\mathbb{D}}%
_{y}\Phi )\right) \right\Vert _{L^{2}(Q_{T};\mathcal{B}_{A}^{2})} \\ 
\ \ \ \ \ \ \ \ \ \ +\left\Vert K\ast \ast \left( a_{0}(\cdot ,\cdot ,%
\overline{\mathbb{D}}_{y}\Phi )-z_{0}\right) \right\Vert _{L^{2}(Q_{T};%
\mathcal{B}_{A}^{2})} \\ 
\leq \frac{\eta }{2}+\underset{E^{\prime }\ni \varepsilon \rightarrow 0}{%
\lim \inf }\left\Vert K^{\varepsilon }\ast (a_{0}^{\varepsilon }(\cdot
,\cdot ,\nabla \Phi _{\varepsilon })-a_{0}^{\varepsilon }(\cdot ,\cdot
,\nabla u_{\varepsilon }))\right\Vert _{L^{2}(Q_{T})}.%
\end{array}%
\end{equation*}%
But 
\begin{eqnarray*}
&&\left\Vert K^{\varepsilon }\ast (a_{0}^{\varepsilon }(\cdot ,\cdot ,\nabla
\Phi _{\varepsilon })-a_{0}^{\varepsilon }(\cdot ,\cdot ,\nabla
u_{\varepsilon }))\right\Vert _{L^{2}(Q_{T})} \\
&\leq &\left\Vert K^{\varepsilon }\right\Vert _{L^{1}(Q_{T})}\left\Vert
a_{0}^{\varepsilon }(\cdot ,\cdot ,\nabla \Phi _{\varepsilon
})-a_{0}^{\varepsilon }(\cdot ,\cdot ,\nabla u_{\varepsilon })\right\Vert
_{L^{2}(Q_{T})} \\
&\leq &C\left\Vert \nabla \Phi _{\varepsilon })-\nabla u_{\varepsilon
}\right\Vert _{L^{2}(Q_{T})}
\end{eqnarray*}%
where $C$ is a positive constant not depending on $\varepsilon $. On the
other hand, by part (i) of assumption (\ref{4.4}), we have 
\begin{equation*}
c_{0}\left\Vert \nabla \Phi _{\varepsilon })-\nabla u_{\varepsilon
}\right\Vert _{L^{2}(Q_{T})}^{2}\leq \int_{Q_{T}}\left( a^{\varepsilon
}(\cdot ,\cdot ,\nabla u_{\varepsilon })-a^{\varepsilon }(\cdot ,\cdot
,\nabla \Phi _{\varepsilon })\right) \cdot (\nabla u_{\varepsilon }-\nabla
\Phi _{\varepsilon })dxdt.
\end{equation*}%
Therefore, proceeding exactly as in the proof of Theorem 3.10 in \cite{ACAP}
(see also the proof of Theorem 4.1 in \cite{AMPA}), we are readily led to $%
K\ast \ast z_{0}=K\ast \ast a_{0}(\cdot ,\cdot ,\overline{\mathbb{D}}_{y}%
\boldsymbol{u})$. This concludes the proof of the theorem.
\end{proof}

\bigskip The \emph{global homogenized} problem (\ref{4.12}) is equivalent to
the following system 
\begin{equation}
\left\{ 
\begin{array}{l}
M(\rho )\int_{0}^{T}\left\langle u_{0}^{\prime }(t),v_{0}(t)\right\rangle
dt+\iint_{Q_{T}\times \Delta (A_{y})\times K(\Delta (A_{\tau }))}\widehat{a}%
(\cdot ,\cdot ,\mathbb{D}\boldsymbol{u})\cdot \nabla v_{0}dxdtd\beta \\ 
\ \ +\iint_{Q_{T}\times \Delta (A_{y})\times K(\Delta (A_{\tau }))}(\widehat{%
K}\ast \ast \widehat{a}_{0}(\cdot ,\cdot ,\mathbb{D}\boldsymbol{u}%
))v_{0}dxdtd\beta =\int_{Q_{T}}fv_{0}dxdt \\ 
\text{for all }v_{0}\in L^{2}(0,T;W_{0}^{1,2}(\Omega ))%
\end{array}%
\right.  \label{4.15}
\end{equation}%
and\ 
\begin{equation}
\iint_{Q_{T}\times \Delta (A_{y})\times K(\Delta (A_{\tau }))}\widehat{a}%
(\cdot ,\cdot ,\mathbb{D}\boldsymbol{u})\cdot \partial \widehat{v}%
_{1}dxdtd\beta =0\text{, all }v_{1}\in L^{2}(Q_{T};\mathcal{B}_{A_{\tau
}}^{2}(\mathbb{R}_{\tau };\mathcal{B}_{A_{y}}^{1,2})).  \label{4.16}
\end{equation}

Let us first deal with (\ref{4.16}). For that, let $\lambda \in \mathbb{R}%
^{N}$, and consider the following variational 
\begin{equation}
\left\{ 
\begin{array}{l}
\text{Find }v(\lambda )\in \mathcal{B}_{A_{\tau }}^{2}(\mathbb{R}_{\tau };%
\mathcal{B}_{A_{y}}^{1,2}): \\ 
\int_{\Delta (A_{y})\times K(\Delta (A_{\tau }))}\widehat{a}(\cdot ,\cdot
,\lambda +\partial \widehat{v(\lambda )})\cdot \partial \widehat{w}d\beta =0%
\text{ for all }w\in \mathcal{B}_{A_{\tau }}^{2}(\mathbb{R}_{\tau };\mathcal{%
B}_{A_{y}}^{1,2}).%
\end{array}%
\right.  \label{4.17}
\end{equation}%
Owing to the properties of the function $a$ (see Proposition \ref{p4.1}), it
holds that (\ref{4.17}) possesses a solution in $\mathcal{B}_{A_{\tau }}^{2}(%
\mathbb{R}_{\tau };\mathcal{B}_{A_{y}}^{1,2})$ which is unique in the space $%
\mathcal{B}_{A_{\tau }}^{2}(\mathbb{R}_{\tau };\mathcal{B}%
_{A_{y}}^{1,2}/I_{A}^{2})$. Now, taking $\lambda =\nabla u_{0}(x,t)$ with $%
(x,t)$ arbitrarily fixed in $Q_{T}$, and then choosing in (\ref{4.16}) the
particular test functions $v_{1}(x,t)=\varphi (x,t)w$ ($(x,t)\in Q_{T}$)
with $\varphi \in \mathcal{C}_{0}^{\infty }(Q_{T})$ and $w\in \mathcal{B}%
_{A_{\tau }}^{2}(\mathbb{R}_{\tau };\mathcal{B}_{A_{y}}^{1,2})$, and finally
comparing the resulting equation with (\ref{4.17}), it follows (by the
uniqueness argument) that $u_{1}=v(\nabla u_{0})$, where the right-hand side
of this equality stands for the function $(x,t)\mapsto v(\nabla u_{0}(x,t))$
from $Q_{T}$ into $\mathcal{B}_{A_{\tau }}^{2}(\mathbb{R}_{\tau };\mathcal{B}%
_{A_{y}}^{1,2}/I_{A}^{2})$.

We can now deal with (\ref{4.15}). To this end, we define the homogenized
coefficients as follows: For $\lambda \in \mathbb{R}^{N}$, 
\begin{eqnarray*}
b(\lambda ) &=&\int_{\Delta (A_{y})\times K(\Delta (A_{\tau }))}\widehat{a}%
(\cdot ,\cdot ,\lambda +\partial \widehat{v(\lambda )})d\beta ; \\
b_{0}(\lambda ) &=&\int_{\Delta (A_{y})\times K(\Delta (A_{\tau }))}(%
\widehat{K}\ast \widehat{a}_{0}(\cdot ,\cdot ,\lambda +\partial \widehat{%
v(\lambda )}))d\beta ; \\
\widetilde{\rho } &=&M(\rho ).
\end{eqnarray*}%
Then substituting $u_{1}=v(\nabla u_{0})$ in (\ref{4.15}) and choosing there
the special test function $v_{0}=\varphi \in \mathcal{C}_{0}^{\infty }(Q_{T})
$, we quickly obtain by disintegration, the macroscopic homogenized problem,
viz.,%
\begin{equation}
\left\{ 
\begin{array}{l}
\widetilde{\rho }\frac{\partial u_{0}}{\partial t}-\Div b(\nabla
u_{0})+b_{0}(\nabla u_{0})=f\text{ in }Q_{T} \\ 
u_{0}(0)=u^{0}\text{ in }\Omega .%
\end{array}%
\right.   \label{4.18}
\end{equation}%
By the uniqueness of the solution to (\ref{4.12}), the existence and the
uniqueness of the solution to (\ref{4.18}) is ensured. We are therefore led
to the following

\begin{theorem}
\label{t4.3}Assume that \emph{(\ref{4.8})} holds. For each $\varepsilon >0$
let $u_{\varepsilon }$ be the unique solution to \emph{(\ref{4.1})}. Then as 
$\varepsilon \rightarrow 0$, 
\begin{equation*}
u_{\varepsilon }\rightarrow u_{0}\text{\ in }L^{2}(Q_{T})
\end{equation*}%
where $u_{0}$ is the unique solution to \emph{(\ref{4.18})}.
\end{theorem}

\end{document}